\documentclass[12pt]{amsart}
\usepackage{euscript,epsf,amssymb,xr}

\let\oldmarginpar\marginpar
\renewcommand\marginpar[1]
{\oldmarginpar{\tiny\bf \begin{flushleft} #1 \end{flushleft}}}

\input xy
\xyoption{all}

\newcommand{\la}{\langle}
\newcommand{\ra}{\rangle}

\newtheorem{theorem}{\bf Theorem}[section]
\newtheorem{lemma}[theorem]{\bf Lemma}

\newtheorem{remark}[theorem]{\bf Remark}

\newcommand{\CC}{{\Bbb C}}

\newcommand{\CP}{{\Bbb CP}}

\newcommand{\NN}{{\Bbb N}}
\newcommand{\PP}{{\Bbb P}}
\newcommand{\QQ}{{\Bbb Q}}
\newcommand{\RR}{{\Bbb R}}

\newcommand{\ZZ}{{\Bbb Z}}

\newcommand{\slie}{{\frak s}}

\newcommand{\ggreat}{>\kern-.7ex>}
\newcommand{\ssmall}{<\kern-.7ex<}
\newcommand{\qu}{/\kern-.7ex/}
\newcommand{\exh}{\to\kern-1.8ex\to}

\newcommand{\cC}{{\EuScript{C}}}

\newcommand{\gG}{{\EuScript{G}}}

\newcommand{\jJ}{{\EuScript{J}}}

\newcommand{\pP}{{\EuScript{P}}}
\newcommand{\sS}{{\EuScript{S}}}
\newcommand{\tT}{{\EuScript{T}}}

\newcommand{\GL}{\operatorname{GL}}

\newcommand{\Aut}{\operatorname{Aut}}

\newcommand{\Cov}{\operatorname{Cov}}

\newcommand{\Diff}{\operatorname{Diff}}

\newcommand{\Ext}{\operatorname{Ext}}

\newcommand{\Hom}{\operatorname{Hom}}
\newcommand{\id}{\operatorname{id}}
\newcommand{\Id}{\operatorname{Id}}

\newcommand{\Ker}{\operatorname{Ker}}

\renewcommand{\O}{\operatorname{O}}

\newcommand{\SO}{\operatorname{SO}}

\newcommand{\Symp}{\operatorname{Symp}}
\newcommand{\Tor}{\operatorname{Tor}}

\newcommand{\Vol}{\operatorname{Vol}}

\newcommand{\ov}{\overline}

\newcommand{\ord}{\operatorname{ord}}

\newcommand{\imag}{\mathbf{i}}

\title{Which finite groups act smoothly on a given $4$-manifold?}

\author{Ignasi Mundet i Riera}
\address{Facultat de Matem\`atiques i Inform\`atica\\
Universitat de Barcelona\\
Gran Via de les Corts Catalanes 585\\
08007 Barcelona \\
Spain}
\email{ignasi.mundet@ub.edu}

\author{Carles S\'aez--Calvo}
\address{Facultat de Matem\`atiques i Inform\`atica\\
Universitat de Barcelona\\
Gran Via de les Corts Catalanes 585\\
08007 Barcelona \\
Spain. Barcelona Graduate School of Mathematics (BGSMath)}
\email{csaez@crm.cat}

\thanks{Both authors have been partially supported
by the (Spanish) MEC Project MTM2015-65361-P. The second author acknowledges financial support from the Spanish Ministry of Economy and Competitiveness, through the Mar\'{\i}a de Maeztu Programme for Units of Excellence in R\&D (MDM-2014-0445).}

\date{January 11, 2019}

\begin{document}

\maketitle

\begin{abstract}
We prove that for any closed smooth $4$-manifold $X$ there exists a constant $C$ with the
property that each finite subgroup $G<\Diff(X)$ has a subgroup
$N$ which is abelian or nilpotent of class $2$, and which satisfies $[G:N]\leq C$.
We give sufficient conditions on $X$ for $\Diff(X)$ to be Jordan, meaning that there
exists a constant $C$ such that any finite subgroup $G<\Diff(X)$ has an abelian subgroup
$A$ satisfying $[G:A]\leq C$. Some of these conditions are homotopical,
such as having nonzero Euler characteristic or nonzero signature, others are geometric,
such as the absence of embedded tori of arbitrarily large self-intersection arising as
fixed point components of periodic diffeomorphisms. Relying on these results, we prove
that: (1) the symplectomorphism group of any closed symplectic $4$-manifold is Jordan, and
(2) the automorphism group of any almost complex closed $4$-manifold is Jordan.
\end{abstract}

\tableofcontents

\section{Introduction}

\subsection{Main results}
One of the most basic problems in the theory of finite transformation groups is to determine
which finite groups act smoothly and effectively on a given closed manifold.
This is in general a very difficult problem, and in dimension $4$ it has been solved for
very few manifolds: these include $S^4$ \cite{CKS, MZ2}, $\CP^2$ \cite{HL,W}, flat manifolds
as the torus $T^4$ \cite{LR} or (for trivial reasons) the asymmetric manifolds
(see \cite{Sch} for the construction of infinitely many asymmetric closed $4$-manifolds).

A simpler problem is to provide restrictions on which finite groups act smoothly and
effectively on a given closed 4-manifold, with the aim of narrowing as much as possible
the list of potential finite groups of symmetries of the given manifold.
One may hope that for any fixed 4-manifold the algebraic structure of the
finite groups acting effectively on it cannot be arbitrarily complicated. This has
been confirmed for manifolds with $b_1=0$ in \cite{Mc1}, where it is proved
that for such manifolds any finite group acting effectively but trivially on homology is abelian and
can be generated by one or two elements (see  \cite{MZ1} for more precise results on the
case $b_1=0$ and $b_2=2$). A weaker form of this result has been extended in \cite{M4} to
4-manifolds with nonzero Euler characteristic.

In this paper we prove a theorem that materializes the previous hope for all
effective smooth actions on closed 4-manifolds. To state this theorem we recall some standard terminology.
A group $G$ is said to be nilpotent of class at most $2$ if $[a,[b,c]]=1$ for every $a,b,c\in G$.
Equivalently, $G/Z(G)$ is abelian, where $Z(G)\leq G$ is the center of $G$.
For example, any abelian group is nilpotent of class at most $2$.
Note that in the literature on nilpotent Lie algebras the analogous property is sometimes called $2$-step nilpotency.

\begin{theorem}
\label{thm:2-step-nilpotent}
Let $X$ be a closed smooth $4$-manifold. There exists a constant $C$ such that every group $G$ acting in a smooth and effective way on $X$ has a subgroup $G_0 \leq G$ such that
$[G:G_0]\leq C$ and:
\begin{enumerate}
\item $G_0$ is nilpotent of class at most $2$,
\item $[G_0,G_0]$ is a (possibly trivial) cyclic group,
\item $X^{[G_0,G_0]}$ is either $X$ or a disjoint union of embedded tori.
\end{enumerate}
\end{theorem}

As a qualitative statement this theorem is as good as possible. Namely, if one replaces
"nilpotent of class at most $2$" by "abelian" then the statement is no longer true.
For example, it is false for $T^2\times S^2$, because this manifold has non Jordan
diffeomorphism group (see below). In contrast, in dimensions lower than $4$ the
previous theorem does hold with "nilpotent of class at most $2$" replaced by
"abelian" (the one dimensional case is elementary; see \cite{M1} for dimension $2$
and \cite{Z2} for dimension $3$).

The following result complements Theorem \ref{thm:2-step-nilpotent} by relating
the algebraic structure of nilpotent groups
of class at most 2 to the geometry of their potential smooth actions on a given oriented
4-manifold. We are interested on the following invariant
of a finite group $G$:
$$\alpha(G)=\min\{[G:A]\mid A\leq G\text{ abelian}\}.$$
The number $\alpha(G)$ may be understood as a measure of how far $G$ is from being abelian.
In the next theorem and in a few other sections of this paper
we are going to use the following standard fact. If $X$ is a closed, connected and
oriented $4$-manifold and $\Sigma\subset X$ is an embedded closed orientable curve, then
picking an orientation of $\Sigma$ we obtain a homology class $[\Sigma]\in H_2(X)$
whose self intersection can be identified with an integer. This
integer is independent of the orientation of $\Sigma$ and will be denoted
by $\Sigma\cdot\Sigma$.

\begin{theorem}
\label{thm:2-nilpotent-alpha}
Let $X$ be a closed, connected and oriented smooth $4$-manifold. There exist
a constant $C$ and a
function $f:\NN\to\NN$ (both $C$ and $f$ depend on $X$)
satisfying $\lim_{n\to\infty}f(n)=\infty$ and for every
finite nilpotent group $N$ of class at most $2$ acting in a smooth and effective way on $X$
and satisfying $\alpha(N)\geq C$
there is some $g\in[N,N]$ satisfying:
     \begin{enumerate}
     \item the order of $g$ satisfies $\ord(g)\geq f(\alpha(N))$,
     \item $X^g$ is a nonempty disjoint union of embedded tori $T_1,\dots,T_s\subset X$,
     \item for every $i$ we have $|T_i\cdot T_i|\geq C\,\alpha(N)$.
     \end{enumerate}
\end{theorem}

\subsection{Jordan property}
The results in this paper are closely related to a question of \'E. Ghys on the Jordan property of diffeomorphism groups. Recall (see \cite{Po0}) that a group $\gG$ is said to be Jordan if there is some number $C$ such that any finite
subgroup $G\leq\gG$ satisfies $\alpha(G)\leq C$. If such $C$ exists, it is called a Jordan constant
for $\gG$. If $\gG$ is Jordan, the optimal Jordan constant of $\gG$ is defined to be the maximum
of $\alpha(G)$ as $G$ runs along the set of the finite subgroups of $\gG$.
The most basic examples of Jordan group are the linear groups
$\GL(n,\RR)$: this is the content of a theorem of C. Jordan (see Theorem \ref{thm:classical-Jordan} below).

Ghys \cite{Gh} raised the question around 30 years ago
of whether diffeomorphism groups of closed manifolds are Jordan. A number of papers
have been written on this question in the last few years \cite{CPS,M1,M4,M6,Z2}.
In dimension at most $3$ any closed manifold has Jordan diffeomorphism group
\cite{M1,Z2}. In contrast, in dimension $4$ (and higher) one encounters
both manifolds with Jordan diffeomorphism group
(for example, those with nonzero Euler characteristic \cite{M4}, or the torus $T^4$ \cite{M1}) and manifolds with non Jordan diffeomorphism group \cite{CPS} (for example $T^2\times S^2$).

\begin{remark}
Given the existence of closed manifolds for which Ghys's question has
a negative answer, one may wonder whether there exists a weakening of
Jordan's property that is satisfied by the diffeomorphism groups of all
closed manifolds. Ghys himself asked in 2015 \cite{Gh2} 
whether for every closed smooth manifold $M$ there exists a constant
$C$ such that any finite subgroup $G<\Diff(M)$ has a nilpotent subgroup
$N\leq G$ with $[G:N]\leq C$. Theorem \ref{thm:2-step-nilpotent}
implies that this is true if $M$ is $4$-dimensional, and suggests
that one could perhaps bound the nilpotency class of $N$ as a function of
the dimension of $M$.
\end{remark}

It is a natural and very interesting problem to determine which closed 4-manifolds
have Jordan diffeomorphism group. The next theorem gives a partial solution to this problem by
providing necessary conditions for a 4-manifold to have non Jordan diffeomorphism group.
The statement actually applies more generally
to subgroups of the group of diffeomorphisms: this will
be crucial later when considering automorphisms of geometric structures
(see Subsection \ref{ss:geometric-structures} below).

\begin{theorem}
\label{thm:criterion-degree}
Let $X$ be a closed connected oriented smooth $4$-manifold, and let $\gG$ be a subgroup
of $\Diff(X)$. If $\gG$ is not Jordan then there exists a sequence $(\phi_i)_{i\in\NN}$ of elements
of $\gG$ such that:
\begin{enumerate}
\item each $\phi_i$ has finite order $\ord(\phi_i)$,
\item $\ord(\phi_i)\to\infty$ as $i\to\infty$,
\item all connected components of $X^{\phi_i}$ are embedded tori,
\item for every $C>0$ there is some $i_0$ such that if $i\geq i_0$ then any connected
component $\Sigma\subseteq X^{\phi_i}$ satisfies
$|\Sigma\cdot\Sigma|\geq C,$
\item we may pick for each $i$ two connected components $\Sigma_i^-,\Sigma_i^+\subseteq X^{\phi_i}$
in such a way that the resulting homology classes $[\Sigma_i^{\pm}]\in H_2(X)$ satisfy
$\Sigma_i^{\pm}\cdot \Sigma_i^{\pm}\to\pm\infty$ as $i\to\infty$.
\end{enumerate}
\end{theorem}
\begin{proof}
This is a consequence of Theorems \ref{thm:2-step-nilpotent} and \ref{thm:2-nilpotent-alpha},
together with Lemma \ref{lemma:positius-i-negatius}.
\end{proof}

In the next result we collect a few sufficient conditions for the diffeomorphism
group of a closed 4-manifold to be Jordan. We denote by $\chi(X)$, $\sigma(X)$
the Euler characteristic and the signature of a connected, oriented and closed manifold $X$.

\begin{theorem}
\label{thm:non-Jordan} Let $X$ be a connected, closed, oriented and smooth
$4$-manifold. If $X$ satisfies any of the following conditions
then $\Diff(X)$ is Jordan:
\begin{enumerate}
\item $\chi(X)\neq 0$,
\item $\sigma(X)\neq 0$,
\item $b_2(X)=0$,
\item $b_2^+(X)>1$ and $X$ has some nonzero Seiberg--Witten
    invariant,
\item $b_2^+(X)>1$ and $X$ has some symplectic structure.
\end{enumerate}
\end{theorem}
\begin{proof}
By the main result in \cite{M4}, if $\chi(X)\neq 0$ then $\Diff(X)$ is Jordan.
By Theorem \ref{thm:signature-nonzero-jordan} if $\sigma(X)\neq 0$ then $\Diff(X)$ is Jordan.
By Theorem \ref{thm:b-2-zero} if $b_2(X)=0$ then $\Diff(X)$ is Jordan.

Assume that $b_2^+(X)>1$ and that $X$ has some nonzero Seiberg--Witten
invariant, and let us prove that $\Diff(X)$ is Jordan. If $\Diff(X)$ were not Jordan, then
by Theorem \ref{thm:criterion-degree}
there would exist in $X$ some embedded torus of positive self-intersection.
According to the adjunction formula for Seiberg--Witten invariants (see e.g. \cite{KM}, \cite[Theorem 1.1]{OS}; here we are using $b_2^+(X)>1$), if $\slie$ is a Spin$^c$ structure on $X$  with nonzero Seiberg--Witten invariant then the first Chern class
$c_1(\slie)\in H^2(X;\ZZ)$
of the determinant line bundle of $\slie$
satisfies the following formula for any embedded surface $\Sigma$ of positive genus
and positive self-intersection:
$$\chi(\Sigma) + \Sigma \cdot \Sigma \leq -|\la c_1(\slie),\Sigma\ra|.$$
This formula cannot be true if $\Sigma$ is an embedded torus of positive self-intersection,
and hence all Seiberg--Witten invariants of $X$ are trivial, contradicting our hypothesis.

If $X$ has $b_2^+(X)>1$ and admits a symplectic structure, then it has some non-vanishing Seiberg--Witten invariant by Taubes' theorem \cite{T1}. Therefore
$\Diff(X)$ must be Jordan by the previous argument.
\end{proof}

Hence if $X$ is a closed 4-manifold such that $\Diff(X)$ is not Jordan,
and $H^*(X;\QQ)$ is not isomorphic to $H^*(T^2\times S^2;\QQ)$ as a graded vector space,
then all Seiberg--Witten invariants of $X$ are zero.

\subsection{Geometric structures}
\label{ss:geometric-structures}

The existence of 4-manifolds whose diffeomorphism group is not Jordan leads
naturally to the consideration of Jordan's property for subgroups of the
diffeomorphism group. The most natural examples of such subgroups are
the automorphism groups of geometric structures
such as symplectic structures (one may consider here the symplectomorphism group
or its subgroup of Hamiltonian diffeomorphisms), or complex structures
(see \cite{PS2,PS3}; this question is connected to a whole set of results about Jordan property in the algebraic world, see \cite{MZ,Po2,PS,S3,S2} and the references therein).
This makes sense not only in dimension $4$ but in higher dimensions.
It turns out that there exist many closed manifolds whose diffeomorphism group is not
Jordan but which admit some structure (symplectic or complex) whose automorphism
group is Jordan. The most basic example is $T^2\times S^2$ \cite{CPS,M5,PS2}, but
there are infinitely many higher dimensional examples \cite{M6,M7}.

In this paper we address this question in dimension $4$ for almost
complex structures and for symplectic structures.

The following theorem extends the main result of \cite{PS2,PS3} from complex structures
to almost complex structures.

\begin{theorem}
\label{thm:almost-complex}
Let $X$ be a closed and connected smooth $4$-manifold, and let $J$ be an almost complex structure on $X$.
Let $\Aut(X,J)\subset\Diff(X)$ be the group of diffeomorphisms preserving $J$. Then $\Aut(X,J)$
is Jordan.
\end{theorem}

The proof of Theorem \ref{thm:almost-complex} is based on Theorem \ref{thm:criterion-degree}
(whose proof on its turn is based on \cite{MT}, which uses the CFSG), and hence is very different from that in
\cite{PS2}, which is based on the classification of compact complex surfaces.

The following theorem generalizes the results in \cite{M5} to arbitrary symplectic 4-manifolds.

\begin{theorem}
\label{thm:symp}
For any closed symplectic $4$-manifold $(X, \omega)$ we have:
\begin{enumerate}
\item $\Symp (X, \omega)$ is Jordan.
\item If $X$ is not an $S^2$-bundle over $T^2$, then a Jordan constant for $\Symp (X, \omega)$
can be chosen independently of $\omega$.
\item If $b_1(X) \neq 2$, then $\Diff(X)$ is Jordan.
\end{enumerate}
\end{theorem}

\begin{remark}
Regarding statement (2) in the previous theorem, note that for a Jordan group $G$
depending on a parameter $\omega$ the following two assertions are in general different:
\begin{enumerate}
\item[(i)] the optimal Jordan constant of $G$ does not depend on $\omega$,
\item[(ii)] one can pick a Jordan constant of $G$ which is independent of $\omega$.
\end{enumerate}
Of course (i) is stronger than (ii), and statement (2) in Theorem \ref{thm:symp} refers to (ii).
\end{remark}

Statement (2) in Theorem \ref{thm:symp} is sharp in the sense that
if $X$ is an $S^2$-bundle over $T^2$ then it is impossible to find some number $C$ which is a
Jordan constant for $\Symp(X,\omega)$ for all symplectic forms $\omega$ on $X$. More precisely, if
$X=T^2\times S^2$, Theorems 1.1 and 1.2 in \cite{M5} imply that the optimal Jordan constant for $\Symp(X,\omega)$ is equal to $\mu(\omega)+C_0(\omega)$, where
$\mu(\omega)=|12\la [\omega],T^2\ra/\la [\omega],S^2\ra|$ and
$C_0(\omega)$ is bounded independently of $\omega$, and if $X$ is the twisted $S^2$-bundle
over $T^2$ then the arguments in the proof of Lemma \ref{lemma:covering-lemma} allow
to obtain a similar estimate for the optimal Jordan constant for $\Symp(X,\omega)$ for
all $\omega$.

Statement (3) in Theorem \ref{thm:symp} is also sharp: this follows from
\cite{CPS} (see also \cite{M6}).

A similar theorem can be proved
for isometry groups of closed Lorentz 4-manifolds
(see \cite{M8}, which uses Theorem \ref{thm:criterion-degree} in this paper).
Note that isometry groups of closed Riemannian manifolds $(X,g)$
are always Jordan, because they are compact Lie groups and hence they can
be identified, by the Peter--Weyl theorem, with subgroups of some linear
group $\GL(n,\RR)$.
This is no longer true for Lorentz metrics: while their isometry group is a finite
dimensional Lie group, it may be noncompact and even have infinitely many connected
components (see the references in \cite{M8}).

\subsection{Main ideas of the proofs}
\label{ss:structure-proof}

The proof of Theorem \ref{thm:2-step-nilpotent} follows different routes depending
on whether $b_2(X)$ is zero or not. A common ingredient in both situations is the
main result in \cite{MT}. This result is concerned
with the following analogue of the Jordan property:
given positive integers $C$ and $d$,
a collection of finite groups $\cC$ satisfies $\jJ(C,d)$
if each $G\in\cC$ has an abelian subgroup $A$ such that $[G:A]\leq C$ and
$A$ can be generated by $d$ elements. Let $\tT(\cC)$ be the set of
all $G\in\cC$ which fit in an exact sequence of groups
$1\to P\to G\to Q\to 1$ such that the orders of $P$ and $Q$ are both
prime powers. The main result in \cite{MT} (see Theorem \ref{thm:TM} below)
states that if $\cC$ is closed under taking subgroups and $\tT(\cC)$
satisfies $\jJ(C,d)$ for some $C$ and $d$ then $\cC$ satisfies $\jJ(C',d)$
for some $C'$. We remark that this result uses the classification of finite simple
groups.

If $M$ is a closed manifold and
$\gG$ denotes the collection of all finite subgroups of $\Diff(M)$ then
$\Diff(M)$ is Jordan if and only if $\gG$ satisfies $\jJ(C,d)$ for some
$C$ and $d$: this is a nontrivial fact that follows from
a theorem of Mann and Su (see Theorem \ref{thm:MS} below).

In all the results stated in this introduction it suffices to consider connected
manifolds, because closed manifolds have finitely many connected components.
Let $X$ be a closed connected $4$-manifold.
By the main result in \cite{M5}, if the Euler characteristic of $X$ is nonzero then $\Diff(X)$ is Jordan. Consequently,
to prove Theorem \ref{thm:2-step-nilpotent} it suffices to consider the case
in which $X$ is connected and $\chi(X)=0$.

Assume first that $b_2(X)=0$. In this case we directly prove that $\Diff(X)$ is Jordan.
This is the main result in Section \ref{s:b-2-zero} (see Theorem \ref{thm:b-2-zero}),
and we next briefly explain the structure of the proof.
Let $\gG$ be the collection of all finite subgroups of $\Diff(X)$, and let
$\pP\subseteq\gG$ be the collection of $p$-groups (for all primes $p$).
Since $\chi(X)=0$ we have $b_1(X)=1$ so $H^1(X)\simeq\ZZ$. One can prove that if
a finite group $G$ acts on $X$ trivially on $H^1(X)$ then there exists a classifying
map $c:X\to S^1$ for a generator of $H^1(X)$ which is equivariant with respect to an
action of $G$ on $S^1$ given by a character $\rho:G\to S^1$. The latter is called the
rotation morphism, and is defined in Subsection \ref{ss:rotation}.
To study groups $G$ acting on $X$ trivially on $H^1(X)$ we consider separately
$\Ker\rho$ and $\rho(G)$. In particular we prove that if $G$ is an abelian $p$-group acting
freely on $X$ and
$\rho(G)$ is trivial then $G$ must contain a cyclic subgroup of bounded index. This is the main
ingredient in the proof that  $\pP$ satisfies $\jJ(C,d)$ for some $C$ and $d$. We deduce from this
that $\tT(\gG)$ satisfies $\jJ(C',d)$ using the main result in Section \ref{s:diffeomorphisms-normalizing}. Applying the main result in \cite{MT} we
conclude that $\Diff(X)$ is Jordan.

Now assume that $b_2(X)\neq 0$. The proof of Theorem \ref{thm:2-step-nilpotent} in this case
is contained in Section \ref{s:proofs-main-theorems}.
The main step consists in proving that the set $\gG_0$ of all finite subgroups of $\Diff(X)$
which are of the form $[G,G]$, where $G<\Diff(X)$ is finite, satisfies the property $\jJ(C,r)$
for some $C$ and $r$. To prove this we apply the main result in \cite{MT} to $\gG_0$, so it suffices
to prove $\jJ(C',r)$ for $\tT(\gG_0)$. For the purpose of proving Theorem \ref{thm:2-step-nilpotent} we may assume that $X$ is orientable (see Section \ref{s:first-simplifications}).
Choose one orientation and let $[X]$ denote the fundamental class. The assumption $b_2(X)\neq 0$, combined with Poincar\'e duality, implies the existence of line bundles $L_1,L_2\to X$ such that $\la c_1(L_1)c_1(L_2),[X]\ra=1$. It is proved in \cite{M7} that any $\Gamma\in\gG_0$ has a
central extension $\widehat{\Gamma}$ which acts on $L_1,L_2$ lifting the action of $\Gamma$.
If $\Gamma$ is cyclic a simple trick implies that the action of $\Gamma$ itself can be lifted
to $L_1,L_2$, and then the equality $\la c_1(L_1)c_1(L_2),[X]\ra=1$ prevents the action of
$\Gamma$ on $X$ from being free. This can be used to prove that any cyclic $p$-group
$\Gamma\in\gG_0$ has a subgroup of bounded index with nonempty fixed point set (this is Lemma \ref{lemma:commutator-fixed-points}), and from this one easily deduces that any $p$-group
$\Gamma\in\gG_0$ (cyclic or not) has an abelian subgroup of bounded index $B\leq\Gamma$
such that at least one of these properties is true: $X^B\neq\emptyset$, or $B$ fixes a nonorientable
embedded surface in $X$. This then leads to the proof that $\gG_0$ satisfies $\jJ(C',r)$
(Lemma \ref{lemma:Jordan-commutator}).

Theorem \ref{thm:2-nilpotent-alpha} on its turn follows from Theorem \ref{thm:2-step-nilpotent}
and from results on actions of finite groups on line bundles over closed surfaces
proved in \cite{M5} and recalled in Section \ref{s:surfaces-line-bundles}.

In the previous sketch we have mentioned some of the sections of the paper, and we
now explain the contents of the other ones. Section \ref{s:first-simplifications} proves
that for the purposes of the present paper we may assume that all closed manifolds
are orientable and that all finite group actions are trivial on cohomology.
Section \ref{s:abelian-groups} contains some auxiliary lemmas on finite abelian $p$-groups
with a bound on the number of generators; these results are used in the present paper
in combination with the theorem of Mann and Su. Section \ref{s:linearization} gathers some
basic consequences of the fact that a group action that preserves a submanifold induces
an action on the normal bundle of the submanifold by vector bundle automorphisms.
These results are crucial in many of our arguments, and this is the reason why our
results can not be automatically transferred from diffeomorphism to homeomorphism groups. The results in Section \ref{s:surfaces-line-bundles} refer to actions of finite groups on line bundles over closed
surfaces. Section \ref{s:non-free-actions} proves that if a finite abelian $p$-group acts smoothly
on a closed $4$-manifold and no element has an isolated fixed point then the homology of the complement of the set where the action is free is bounded independently of $p$ and
the  group action. This has some important consequences for non-free actions of finite $p$-groups,
which are also proved in Section \ref{s:non-free-actions}. Section \ref{s:diffeomorphisms-normalizing} proves a technical result that in many situations
allow to pass from property $\jJ(C,d)$ for the finite $p$-subgroups of $\Diff(X)$ to
property $\jJ(C',d)$ for the finite subgroups $G<\Diff(X)$ sitting
in an exact sequence of groups
$1\to P\to G\to Q\to 1$ where both $P$ and $Q$ have
prime power order.
The contents of Sections \ref{s:b-2-zero} and \ref{s:proofs-main-theorems} have already
been explained. Section \ref{s:Atiyah-Singer} extracts some consequences of the Atiyah--Singer
$G$-signature theorem. Sections \ref{s:almost-complex} and \ref{s:symplectomorphisms}
contain the proofs of Theorems
\ref{thm:almost-complex} and \ref{thm:symp} respectively.
The paper finishes with an appendix computing of the cohomology
of $(\ZZ/p^r)^d$ with coefficients in $\ZZ/n$ (for all $p,r,d,n$).

\subsection{Conventions and notation}
We fix here some basic terminology and notation.
Suppose that a group $G$ acts on a space $X$. We denote
by $X^G$ the set of points $x\in X$ such that $gx=x$.
If $g\in G$ then we denote by $X^g$ the set of points satisfying
$gx=x$. We say that a
subspace $Y\subseteq X$ is {\bf invariant} (or {\bf $G$-invariant}, or
{\bf preserved by $G$}) if for any
$y\in Y$ and $g\in G$ we have $gy\in Y$. This is a standard convention
(see e.g. Bredon, Chapter I, Section 1; or Brocker--tom Dieck Chap I,
Section 4).

For any space $X$ we denote the rational Betti numbers by
$b_j(X)=b_j(X;\QQ)=\dim_{\QQ}H_j(X;\QQ)$. Integer coefficients
will be implicitly assumed in homology and cohomology, so
we will denote $H_*(X)=H_*(X;\ZZ)$, and $H^*(X)=H^*(X;\ZZ)$.
Following the standard convention we denote by $\chi(X)$ the
Euler characteristic of $X$ and by $\sigma(X)$ the signature
of $X$ in case $X$ is a closed oriented manifold.

A continuous action of a group $G$ on a manifold $X$ induces an action on $H^*(X)$. We will say that the action is {\bf cohomologically trivial} ({\bf CT} for short),
if the induced action on $H^*(X)$ is trivial.
If $X$ is orientable and closed, then a CT action is orientation preserving.
An action of $G$ on $X$ is {\bf effective} if $g\cdot x = x$ for all $x \in X$ implies that $g=1$. We will write that an action is {\bf CTE} if it is CT and effective.

For any set $S$ we denote by $\sharp S$ the cardinal of $S$.

Whenever we say that a group $G$ can be generated by $d$ elements we mean that there is
a collection of {\it non necessarily distinct} elements $g_1,\dots,g_d$ which generate $G$.

All manifolds will be assumed by default to be {\bf smooth}.
A {\bf closed manifold} means a compact manifold without boundary.

\section{First simplifications}
\label{s:first-simplifications}

The following is a generalization of the results in \cite[Section 2.3]{M1}.

\begin{lemma}
\label{lemma:covering-lemma}
Let $X$ be a closed connected manifold
and let $X'\to X$ be an unramified covering of finite degree, where $X'$ is connected.
If $\Diff(X')$ is Jordan then $\Diff(X)$ is also Jordan.
Furthermore, if there exists a constant $C$ such that every finite subgroup $G \leq \Diff(X')$ has a nilpotent subgroup $H \leq G$ of class at most $2$ satisfying $[G:H] \leq C$, then the same is true for $\Diff(X)$
for a possibly different value of $C$.
\end{lemma}

\begin{proof}
Let $x_0\in X$ be a base point.
Since $X$ is closed, its fundamental group $\pi_1(X,x_0)$ is finitely generated.
Let $k$ be the degree of the covering $X' \rightarrow X$. Let $\Cov_k(X)$ be the set of isomorphism classes of non-necessarily connected unramified coverings of $X$ of degree $k$ (two coverings $X'\to X$ and
$X''\to X$ are isomorphic if there is a diffeomorphism $X'\to X''$ lifting the identity on $X$).
Let $S_k$ be the permutation group on $k$ letters, and consider the action of $S_k$ on $\Hom(\pi_1(X,x_0),S_k)$
by conjugation. There is a bijection
$$\Cov_k(X)\to \Hom(\pi_1(X,x_0),S_k)/S_k,$$
which sends each element of $\Cov_k(X)$ to its monodromy.
Since $\pi_1(X,x_0)$ is finitely generated, $\Hom(\pi_1(X,x_0),S_k)/S_k$ is finite, so $\Cov_k(X)$ is also finite.

Let now $[X'] \in \Cov_k(X)$ be the class of the covering $\pi:X' \rightarrow X$. Let $G \leq \Diff(X)$ be a finite subgroup. Then $G$ acts on $\Cov_k(X)$ by pullback. Let $G_0 \leq G$ be the stabilizer of $[X']$.
Since the orbit of $[X']$ in $\Cov_k(X)$ can be identified with $G/G_0$, we have
$[G:G_0] \leq \sharp \Cov_k(X)$.
Define
$$G_1 = \{(g, \phi) \in G_0 \times \Diff(X') \mid \pi \circ \phi = g \circ \pi \}.$$
We have an exact sequence:
$$1 \rightarrow \Aut(X') \xrightarrow{\rho} G_1 \xrightarrow{q} G_0 \rightarrow 1$$
where $\Aut(X') = \{ \phi \in \Diff(X') \mid \pi \circ \phi = \pi \}$ are the automorphisms of the covering, $\rho(\phi) = (1,\phi)$ and $q(g,\phi) = g$. The group $\Aut(X')$ is finite, hence so is $G_1$. The map $(g, \phi) \mapsto \phi$ defines an inclusion $G_1 \hookrightarrow \Diff(X')$.

If there exists an abelian (resp. nilpotent of class at most $2$)
subgroup $H \leq G_1$ satisfying $[G_1:H]\leq C$
then $q(H)$ is also abelian (resp. nilpotent of class at most $2$)
and satisfies $[G_0:q(H)]\leq C$, so
$[G:q(H)] \leq C \sharp \Cov_k(X).$ This proves the Lemma.
\end{proof}

As a corollary, we see that in order to prove that $\Diff(X)$ is Jordan, or to see that there is a constant $C$ such that any finite group $G \leq \Diff(X)$ has a nilpotent subgroup $H$ of
class at most $2$ with $[G:H] \leq C$, it is enough to show that some finite unramified covering of $X$ has that property. In particular, we may assume without loss of generality that $X$ is orientable.

The following result is a consequence of a classical theorem of Minkowski \cite{Min} which states  that the size of any finite subgroup of $\GL(k,\mathbb{Z})$ is bounded above by a constant depending only on $k$. For the proof see \cite[Lemma 2.6]{M4}.

\begin{lemma}
\label{lemma:minkowski}
Let $X$ be a closed manifold. There exists a constant $C$ such that for any continuous action on
$X$ of a finite group $G$ there is a subgroup $G_0\leq G$ satisfying $[G:G_0]\leq C$ and whose
action on $X$ is CT.
\end{lemma}

This implies that it suffices for our purposes
to consider smooth CTE actions
of finite groups. In particular these
actions are orientation preserving because here we only consider
closed manifolds (note that in \cite{M4} we consider more generally manifolds
with boundary, and for them a cohomologically trivial action need not be orientation preserving).

If a finite group $G$ acts smoothly and preserving the orientation on an oriented $4$-manifold $X$
then for every $g\in G$ each connected component of the fixed point set $X^g$ has even codimension.
Hence, if $X^g\neq X$ then $X^g$ is the union of finitely many points and finitely many disjoint
embedded closed and connected surfaces. We will use this fact repeatedly and without explanation
in the arguments that follow.

\section{Abelian groups with a bound on the number of generators}
\label{s:abelian-groups}

In this section we collect several lemmas involving finite abelian
groups and giving estimates on different quantities as a function of
the minimal number of generators of these abelians groups. These results
will be used in subsequent sections in combination with the classical theorem of Mann and Su,
that we recall in Subsection \ref{ss:MS}.

\subsection{Arbitrary finite groups}

\begin{lemma}
\label{lemma:MS-1} For any natural numbers $r,C$ there exists
a number $C'$ with the following property. Let
$G$ be a finite group and let $A\leq G$ be an abelian subgroup.
Suppose that $A$ can be generated by $r$ elements
and that $[G:A]\leq C$. Let
$\Aut_A(G)\leq \Aut(G)$ be the group of
automorphisms $\phi:G\to G$ satisfying $\phi(A)=A$. We have:
$$[\Aut(G):\Aut_A(G)]\leq C'.$$
\end{lemma}
\begin{proof}

Consider the map $\mu:A\to A$ defined as $\mu(a)=a^{C!}$
(we use multiplicative notation on $A$ and later on $G$), and let
$A_0=\mu(A)\leq A$.
Since $A$ can be generated by $r$ elements,
we have $[A:A_0]\leq (C!)^r$. Furthermore any
subgroup $B\leq A$ satisfying $[A:B]\leq C$ contains $A_0$.
Indeed, for any such $B$ and any $a\in A$ we have
$a^{[A:B]}\in B$, so $a^{C!}\in B$. In particular
we have $A_0\leq \phi(A)\cap A$ for every $\phi\in \Aut(G)$,
which implies $A_0\leq\phi(A)$.
So $A_1=\bigcap_{\phi\in\Aut(G)}\phi(A)$ satisfies $A_0\leq A_1\leq A$
and consequently
$$[A:A_1]\leq [A:A_0]\leq (C!)^r.$$
By its definition $A_1$ is clearly a characteristic subgroup of $G$
(i.e., it is invariant under the action of $\Aut(G)$ on $G$), so in particular
it is normal.

Let $\sS$ be the collection of all subsets of the quotient group $G/A_1$.
Since
$$\sharp G/A_1=[G:A_1]=[G:A][A:A_1]\leq C(C!)^r,$$
we can bound $\sharp\sS\leq C_1:=2^{C(C!)^r}$. The action of $\Aut(G)$ on $G$
induces an action on $G/A_1$ (because $A_1\leq G$ is characteristic)
which on its turn induces an action on $\sS$.
Denote by $[A]\in\sS$ the element corresponding to $A/A_1$ viewed
as a subset of $G/A_1$.
Then $\Aut_A(G)$ is the stabilizer of $[A]$, so we have
$[\Aut(G):\Aut_A(G)]\leq \sharp\sS\leq C_1$.
\end{proof}

\begin{lemma}
\label{lemma:quatre-u-dos}
For any natural numbers $r,C$ there exists
a number $C'$ with the following property.
Let $G$ be a finite group, let $G_0\trianglelefteq G$
be a normal subgroup, and let $A\leq G_0$ be an abelian subgroup.
Suppose that $[G_0:A]\leq C$ and that $A$ can be generated by $r$ elements.
Then the normalizer $N_G(A)$ of $A$ in $G$ satisfies
$$[G:N_G(A)]\leq C'.$$
\end{lemma}
\begin{proof}
Let $c:G\to\Aut(G_0)$ be the morphism defined by the action of $G$ on $G_0$
given by conjugation. Then
$N_G(A)=c^{-1}(\Aut_{A}(G_0)),$
so the lemma follows from Lemma \ref{lemma:MS-1}.
\end{proof}

\begin{lemma}
\label{lemma:rang-del-quocient}
Let $1 \rightarrow Z \rightarrow G \stackrel{\pi}{\rightarrow} A \rightarrow 1$ be an exact sequence of finite groups, where $Z \leq G$ is central and $A$ is abelian.
Let $r$ be an integer such that every abelian subgroup of $G$ is generated by $r$ elements. Then $A$ is generated by $[r(\log_2 (\sharp Z) + 1)]$ elements.
\end{lemma}

\begin{proof}
For any prime $p$ let $A_p\leq A$ denote the $p$-part (i.e., the subgroup of elements whose order
is a power of $p$), and let $s_p$ be the minimal number of generators of $A_p$. Let $A_p[p]\leq A_p$
be the $p$-torsion. Then $A_p[p]\simeq (\ZZ/p)^{s_p}$.
By the Chinese remainder theorem $A$ can be generated by $\max_p s_p$ elements, where $p$ runs over the set of prime numbers dividing $\sharp A$.
Hence it suffices to prove the
lemma when $A\simeq (\ZZ/p)^s$, because
the general case can be reduced to it replacing $G$ by
$\pi^{-1}(A_p[p])$ for every $p\mid\sharp A$.

Assume for the rest of the proof that $A\simeq(\ZZ/p)^s$.
Then $A$ has a natural structure of $s$-dimensional vector space over $\ZZ/p$.
Define a map $\Omega:A \times A \rightarrow Z$ by $\Omega(a,b) = [\tilde{a},\tilde{b}]$, where $\tilde{a},\tilde{b}$ are any lifts of $a,b \in A$ to $G$. This map is well-defined and it is a skew-symmetric bilinear form on $A$ because $Z$ is central. Hence the image of $\Omega$,
which we denote by $Z_{\Omega}$, is a $p$-group and all its nontrivial elements have order $p$.
That is, $Z_{\Omega}\simeq(\ZZ/p)^r$ for some $r$, so $Z_\Omega$ has a natural structure
of vector space over $\ZZ/p$.

For any vector subspace $I\subseteq A$ we denote $I^\perp = \{ a \in A \mid \Omega(a,i)=1 \text{ for every }i\in I\}$. Alternatively, if
we define $\Omega_I:A\to \Hom(I,Z_{\Omega})$ by $\Omega_I(a)(i)=\Omega(a,i)$
we can identify
\begin{equation}\label{eq:perp-subspace}
I^{\perp}=\Ker\Omega_I.
\end{equation}
We say that $I$ is isotropic if $I\subseteq I^{\perp}$. A trivial example of isotropic subspace is $I=0$.
If $I$ is isotropic and there exists some $\gamma\in I^{\perp}\setminus I$, then $I+\la\gamma\ra$
is isotropic (because $\Omega$ is skew-symmetric) and strictly bigger than $I$. Hence, any  maximal isotropic subspace $I$ satisfies
$I=I^{\perp}$.

Choose a maximal isotropic subspace $I\subseteq A$.
By (\ref{eq:perp-subspace}) we have $I=I^{\perp}=\Ker\Omega_I$, so
$$\dim I=\dim\Ker\Omega_I\geq \dim A-\dim\Hom(I,Z_{\Omega})= s-\dim I\dim Z_{\Omega},$$ and consequently
$$\dim I\geq\frac{s}{1+\dim Z_{\Omega}}=\frac{s}{1+\log_p\sharp Z_{\Omega}}\geq \frac{s}{1+\log_2\sharp Z}.$$
Since $I$ is isotropic, $B:=\pi^{-1}(I)\leq G$ is abelian, and it cannot be generated
by less than $\dim I$ elements (because $B$ surjects onto $I$).
Consequently $\dim I\leq r$,
which, combined with the previous estimates, gives
$s\leq \dim I(1+\log_2\sharp Z)\leq r(1+\log_2\sharp Z)$, so the proof of the lemma
is complete.
\end{proof}

\subsection{Finite $p$-groups and MNAS's}

\begin{lemma}
\label{lemma:MS-2} Let $p$ be a prime
and let $B\leq A$ be finite abelian $p$-groups.
Suppose that $A$ can be generated by $r$ elements.
Let $\Aut_B^0(A)\leq\Aut(A)$ denote the automorphisms of $A$
whose restriction to $B$ is the identity. Then
$$\sharp \Aut_B^0(A)\leq [A:B]^{r^2}.$$
\end{lemma}

Note that an analogous lemma can be proved for arbitrary finite
abelian groups, but for our purposes the case of $p$-groups will be
sufficient.

\begin{proof}
Denote $C=[A:B]$. Choose generators
$a_1,\dots,a_r$ of $A$. An automorphism $\phi\in\Aut(A)$ is
determined by the images $\phi(a_1),\dots,\phi(a_r)$. Suppose
that $\phi\in\Aut^0_B(A)$ and write $\phi(a_j)=a_j+d_j$
(additive notation on $A$)
for every $j$, with $d_j\in A$. For each $j$ we have
$Ca_j\in B$, so $Ca_j=\phi(Ca_j)=Ca_j+Cd_j$. It follows
that $Cd_j=0$, so $d_j$ belongs to the $C$-torsion
$A[C]\leq A$. We have $\sharp A[C]\leq {C}^r$, so the set of all possible
choices for $d_1,\dots,d_r$ has at most $({C}^r)^r={C}^{r^2}$ elements.
Hence, $\sharp\Aut^0_B(A)\leq {C}^{r^2}.$
\end{proof}

Let $G$ be a finite group, and let $A$ be an abelian normal subgroup
of $G$. The action of $G$ on itself by conjugation induces a
morphism of groups
$$c:G/A\to\Aut(A).$$
We will write that $A$ is a MNAS (of $G$) if $A$ is a maximal normal abelian subgroup of $G$.
It is well known that if $G$ is a $p$-group (for any prime $p$) and $A$ is a MNAS then $c$
is injective (see e.g. \cite[\S 5.2.3]{Rob}).

\begin{lemma}
\label{lemma:index-MNA}
Let $G$ be a finite $p$-group and let $A\leq G$ be a MNAS.
Suppose that $A$ can be generated by $r$ elements.
For every
abelian subgroup
$B\leq G$ we have 
$$[G:A]\leq [G:B]^{r^2+1}.$$
\end{lemma}
\begin{proof}
Choose an abelian subgroup $B\leq G$.
Let $\pi:G\to G/A$ be the quotient map. Then
$[G/A:\pi(B)]=[\pi(G):\pi(B)]\leq
[G:B].$ If $b\in B$, then $c(\pi(b))\in\Aut^0_{A\cap B}(A)$ because
$B$ is abelian. Using the injectivity of $c$ and Lemma \ref{lemma:MS-2}
we have
$$\sharp\pi(B)\leq \sharp \Aut^0_{A\cap B}(A)\leq [A:A\cap B]^{r^2}\leq [G:B]^{r^2}.$$
Combining the inequalities we have:
$$[G:A]=\sharp G/A=[G/A:\pi(B)]\cdot \sharp \pi(B)\leq [G:B]\cdot [G:B]^{r^2}=[G:B]^{r^2+1}.$$
\end{proof}

\subsection{Mann--Su theorem}
\label{ss:MS}

The following classical result due to Mann and Su will play a prominent role in our arguments.

\begin{theorem}[Theorem 2.5 in \cite{MS}]
\label{thm:MS} For any closed manifold $Y$ there exists some
integer $r\in\ZZ$ depending only on $H^*(Y)$ with the property
that for any prime $p$ and any elementary $p$-group $(\ZZ/p)^s$
admitting an effective action on $Y$ we have $s\leq r$. Equivalently,
any finite abelian group acting effectively on $Y$ can be generated
by $r$ elements.
\end{theorem}

\begin{lemma}
\label{lemma:proceedings}
Given a closed manifold $X$ and natural numbers $C_0,C_1$ there exists a constant $C$
with the following property. Suppose that $G$ is a finite group sitting in an exact sequence
$$1 \rightarrow G_0 \rightarrow G \rightarrow G_1 \rightarrow 1$$
satisfying $\sharp G_0<C_0$, and suppose that there exists an abelian subgroup $B \leq G_1$ satisfying $[G_1:B] \leq C_1$.
If there exists a CTE action of $G$ on $X$ then there is
an abelian subgroup $A \leq G$ satisfying $[G:A]\leq C$.
\end{lemma}

\begin{proof}
Denote by $\pi:G\to G_1$ the projection.
Substituting $G$ by $\pi^{-1}(B)$ 
we may assume without loss of generality that $G_1$ is abelian.
Let $G'$ be the centralizer of $G_0$ inside G. Since $\sharp G_0 \leq C_0$ we have $[G:G'] \leq C_0!$. Define $Z = G_0 \cap G'$ and $G_1' = \pi(G') \leq G_1$. Clearly, $[G_1:G_1'] \leq C_0!$ and $\sharp Z \leq C_0$, and we have an exact sequence
$$1 \rightarrow Z \rightarrow G' \rightarrow G_1' \rightarrow 1,$$
where $Z$ is central in $G'$. Let $r$ be the number given by Theorem \ref{thm:MS} applied to $X$.
By Lemma \ref{lemma:rang-del-quocient}, $G_1'$ can be generated by $[r(\log_2(\sharp Z) + 1)] \leq [r(\log_2 C_0 + 1)]$ elements. Therefore we can apply \cite[Lemma 2.2]{M1} to obtain the result.
\end{proof}

\section{Linearization of finite group actions}
\label{s:linearization}

The following is a classical theorem of Camille Jordan (see \cite{J}, and \cite{CR,M1}
for modern proofs).

\begin{theorem}[Jordan]
\label{thm:classical-Jordan}
For any natural number $n$ there exists a constant $C_n$ such that every finite subgroup $G \leq \GL(n,\mathbb{C})$ has an abelian subgroup $A \leq G$ satisfying $[G:A] \leq C_n$.
\end{theorem}

The next lemma follows from the results in \cite[VI.2]{Bre}.

\begin{lemma}
\label{lemma:lin-fixed-point}
Let $X$ be a connected $4$-manifold and let $G$ be a finite group acting smoothly and
effectively on $X$. Suppose that $X^G \neq \emptyset$. Then, for every $p \in X^G$ the linearization
of the $G$-action at $T_pX$ defines an embedding $G\hookrightarrow \GL(T_pX)$.
In particular, we can identify $G$ with a subgroup of $\GL(4, \mathbb{R})$.
\end{lemma}

Combining Theorem \ref{thm:classical-Jordan} and Lemma \ref{lemma:lin-fixed-point}
and taking $C=C_4$ we obtain:

\begin{lemma}
\label{lemma:linearization-Jordan}
There is a constant $C$ with the following property.
Let $X$ be a connected $4$-manifold and let
$G$ be a finite group acting smoothly and effectively on $X$ with $X^G \neq \emptyset$.
There exists an abelian subgroup $A \leq G$ such that $[G:A] \leq C$.
\end{lemma}

\begin{lemma}
\label{lemma:lin-invariant-surface}
Let $X$ be a connected $4$-manifold and let $G$ be a finite group acting smoothly and effectively on $X$. Suppose that $G$ preserves a connected embedded surface $\Sigma\subset X$.
Let $N=TX|_{\Sigma}/T\Sigma$ be the normal bundle of $\Sigma$.
\begin{enumerate}
\item Linearizing the action
along $\Sigma$ we obtain an effective action of $G$ on $N$ by bundle automorphisms.
\item $G$ sits in an exact sequence $1 \rightarrow G_0 \rightarrow G \rightarrow G_\Sigma \rightarrow 1$, where $G_0$ fixes $\Sigma$ pointwise and is either cyclic or dyhedral, and $G_\Sigma$ acts effectively on $\Sigma$. If in addition $X$ is oriented and $G$ acts on $X$ preserving the orientation
then $G_0$ is cyclic.
\end{enumerate}
\end{lemma}
\begin{proof}
Since $G$ acts smoothly on $X$ preserving $\Sigma$, there is a naturally
induced action of $G$ on $N$.
Let $G_0 \leq G$ be the subgroup of elements of $G$ that fix $\Sigma$ pointwise.
Since $G_0$ is normal in $G$ we have an exact sequence:
$$1 \rightarrow G_0 \rightarrow G \rightarrow G_{\Sigma}:=G/G_0 \rightarrow 1$$
Since $G$ preserves $\Sigma$ we obtain an effective action of  $G_{\Sigma}$ on $\Sigma$.
Take a $G$-invariant Riemannian metric on $X$. Let $x\in\Sigma$ be any point, and choose
an orthogonal basis $e_1,\dots,e_4$ of $T_xX$ with $e_1,e_2\in T_x\Sigma$.
The pair $e_3,e_4$ is mapped to a basis of $N_x$ via the projection map
$T_xX\to N_x=T_xX/T_x\Sigma$.
Expressing the linearization of the action of $G_0$ in terms of the basis
$(e_i)$ we obtain a morphism $G_0\to\O(4,\RR)$ which by Lemma \ref{lemma:lin-fixed-point} is injective. The image of any element of $G_0$ is of the form
$$\left(\begin{array}{cc} \Id_2 & 0 \\ 0 & M\end{array}\right)$$
where $\Id_2\in\SO(2,\RR)$ is the identity and $M\in\O(2,\RR)$. This proves (1).
We thus get a monomorphism
$\iota:G_0\hookrightarrow\O(2,\RR)$ and hence $G_0$ being finite is cyclic or dihedral.
If $X$ is oriented and the action of $G$ on $X$ is orientation preserving
then $\det M=1$, so $\iota(G_0)\leq \SO(2,\RR)$ and hence $G_0$ is cyclic.
This finishes the proof of (2).
\end{proof}

Assume for the rest of this section that $X$ is an oriented closed $4$-manifold.

\begin{lemma}
\label{lemma:orientable-normal-bundle}
Suppose that $\Sigma\subset X$ is a connected embedded surface and that
one of the following two assumptions holds true:
\begin{enumerate}
\item there exists a finite cyclic group $G$ with more than two elements acting
smoothly and effectively on $X$ and fixing $\Sigma$ pointwise;
\item there exists a prime $p>2$ and a finite $p$-group $G$ acting smoothly and effectively
on $X$, preserving $\Sigma$ and inducing a noneffective action on $\Sigma$.
\end{enumerate}
Then $\Sigma$ is orientable.
\end{lemma}
\begin{proof}
Let $N$ be the normal bundle of $\Sigma$.
Suppose that assumption (1) holds true, so that $G$ is cyclic.
Let $\gamma$ be a generator of $G$, and let $d=\sharp G>2$ be its order.
Take any point $x\in\Sigma$. The eigenvalues of the action of $\gamma$ on the fiber $N_x$
are primitive $d$-roots of unit, which are not real because $d>2$. Hence they are of the
form $\zeta,\ov{\zeta}=\zeta^{-1}$. These eigenvalues are independent of $x$ because $\Sigma$ is connected.
Let $N_\mathbb{C} = N\otimes \mathbb{C}$ and define
$$N^\pm = \{ w \in N_\mathbb{C} \mid \gamma\cdot w = \zeta^{\pm 1}w\}.$$
Then $N_\mathbb{C} = N^+ \oplus N^-$
and $N^+$ and $N^-$ are complex line bundles preserved by the action of $G$ on $N_{\CC}$.
Composing the inclusion $N\hookrightarrow N_{\CC}$, $v\mapsto v\otimes 1$, with the
projection $N_{\CC}\to N^+$ we obtain an isomorphism of real vector bundles
$N\to N^+$ which can be used to transport the complex structure on $N^+$ to $N$. Hence $N$ is orientable and since $TX|_{\Sigma}$ is also orientable, we conclude that
$T\Sigma\simeq TX|_{\Sigma}/N$ is orientable as well.

Now suppose that assumption (2) holds true.
Let $G_0$ be the normal subgroup of $G$ consisting of the elements of $G$ fixing $\Sigma$ pointwise.
By Lemma \ref{lemma:lin-invariant-surface}, $G_0$ acts effectively on $N$ by bundle automorphisms and is cyclic or dyhedral.
The action of $G$ on $\Sigma$ is not effective, so $G_0$ is a nontrivial $p$-group with $p>2$ and hence it
has more than two elements and cannot be dihedral, so it is cyclic. Applying the case (1) to the action of
$G_0$ we conclude that $\Sigma$ is orientable.
\end{proof}

\begin{lemma}
\label{lemma:non-orientable}
Let $\Sigma\subset X$ be a connected non-orientable embedded surface.
There exists a constant $C>0$, depending only on $X$ and the genus of $\Sigma$,
such that any finite group $G$ acting smoothly and in a CTE way on $X$ and preserving $\Sigma$
has an abelian subgroup $A \leq G$ satisfying $[G:A]<C$.
\end{lemma}
\begin{proof}
By Lemma \ref{lemma:lin-invariant-surface}, $G$ sits in an exact sequence $1 \rightarrow G_0 \rightarrow G \rightarrow G_\Sigma \rightarrow 1$, where $G_0$ fixes $\Sigma$ pointwise and acts effectively on the fibers of the normal bundle $N$, and $G_{\Sigma}$ acts effectively on $\Sigma$.
Since $G$ acts on $X$ preserving the orientation, $G_0$ is cyclic. By Lemma \ref{lemma:orientable-normal-bundle}, $G_0$ has at most $2$ elements, for otherwise $\Sigma$ would be orientable.  By Lemma \ref{lemma:Jordan-superficies} below, there is an abelian subgroup $B\leq G_\Sigma$ satisfying $[G_{\Sigma}:B]\leq C_0$, where $C_0$ depends only on the genus of $\Sigma$. According to Lemma \ref{lemma:proceedings} this implies the existence of
an abelian subgroup $A\leq G$ satisfying $[G:A]\leq C$ for some constant $C$ depending only on $C_0$
and $X$.
\end{proof}

\section{Surfaces and line bundles}
\label{s:surfaces-line-bundles}

\begin{lemma}
\label{lemma:Jordan-superficies}
Let $\Sigma$ be a closed connected surface. 
There is a constant $C$, depending only on the genus of $\Sigma$, such that any finite subgroup
$G<\Diff(\Sigma)$ has an abelian subgroup $A\leq G$ satisfying $[G:A]\leq C$.
\end{lemma}
\begin{proof}
By Lemma \ref{lemma:covering-lemma}
it suffices to consider the
case of orientable $\Sigma$, and this is proved in \cite[Theorem 1.3 (1)]{M1}.
\end{proof}

\begin{lemma}
\label{lemma:Jordan-fixed-point}
Let $\Sigma$ be a closed connected surface satisfying $\chi(\Sigma)\neq 0$. There exists
a constant $C$, depending only on the genus of $\Sigma$, such that for every
finite group $G$ acting smoothly on $\Sigma$ there exists some point
$x\in\Sigma$ such that $[G:G_x]\leq C$.
\end{lemma}
\begin{proof}
Again by Lemma \ref{lemma:covering-lemma} it suffices to consider the
case of orientable $\Sigma$. Choose an orientation of $\Sigma$.
If the genus $g$ of $\Sigma$ is $2$ or bigger
then any finite group acting effectively
on $\Sigma$ has at most $168(g-1)$ elements
(see e.g. \cite[Theorem 1.3 (2)]{M1}) and this immediately implies the lemma.
Now suppose that $\Sigma\cong S^2$ and let $G$ be a finite group acting smoothly on
$\Sigma$. The subgroup $G'\leq G$ of elements which act preserving the orientation
satisfies $[G:G']\leq 2$. An orientation preserving action on a surface has
isolated fixed points, so by \cite[Theorem 1.4]{M4} there exists some $x\in\Sigma$
such that $[G':G_x]\leq C$, where $C$ is independent of $G$. This proves the lemma.
\end{proof}

\begin{lemma}
\label{lemma:orientable-2-group}
Let $L$ be an oriented rank $2$ real vector bundle over a manifold $\Sigma$,
and let $G$ be a finite group acting on $L$ by orientation preserving
vector bundle automorphisms, lifting an arbitrary smooth action on $\Sigma$.
Then $L$ admits a $G$-invariant complex structure.
\end{lemma}
\begin{proof}
Choosing an arbitrary euclidean structure on $L$ and averaging over the action of $G$
we obtain a $G$-invariant euclidean structure on $L$. There is a unique vector bundle isomorphism
$I:L\to L$ lifting the identity on $\Sigma$ such that
for any $\lambda\in L$ the two vectors $\lambda,I\lambda$ are perpendicular and $\lambda\wedge I\lambda$ is compatible with the orientation of $L$. One checks immediately that $I$ is a $G$-invariant complex
structure on $L$.
\end{proof}

\begin{lemma}
\label{lemma:commuten-infinitesimalment}
Let $E\to\Sigma$ be a rank $2$ real vector bundle over a connected surface.
Suppose that the total space of $E$ is oriented.
Let $\Aut^+(E)$ be the group of vector bundle automorphisms of $E$, lifting arbitrary
diffeomorphisms of $\Sigma$, and preserving the orientation of $E$.
Let $G<\Aut^+(E)$ be a finite group and suppose that $\alpha\in G$ lifts the trivial action on $\Sigma$.
\begin{enumerate}
\item If $\Sigma$ is not orientable then $\alpha$ commutes with all elements of $G$;
\item if $\Sigma$ is orientable then $\alpha$ commutes with the elements of $G$ that act orientation preservingly on $\Sigma$.
\end{enumerate}
\end{lemma}
\begin{proof}
The case $\alpha=\Id_E$ being trivial, we may assume that $\alpha\neq\Id_E$.
Let $G_0\leq G$ be the subgroup of elements lifting the identity map on $\Sigma$.
We have $\alpha\in G_0$.
Applying Lemma \ref{lemma:lin-invariant-surface} to the zero section
of $E$ we conclude that $G_0$ is cyclic.

Suppose that $\Sigma$ is not orientable. Then by Lemma \ref{lemma:orientable-normal-bundle} $G_0$ has at most two elements, so
$\alpha$ has order $2$. Since $\alpha\in\Aut^+(E)$, the action of $\alpha$ on the fibers of $E$ is multiplication by $-1$, and this implies that $\alpha$ commutes with all elements of $\Aut^+(E)$.

Now suppose that $\Sigma$ is orientable. Then, since the total space of $E$
is orientable, $E$ is also orientable. Choose an orientation of $E$. We may replace for our purposes the group $G$ by its intersection with the elements of
$\Aut^+(E)$ that act on $\Sigma$ orientation preservingly. These elements preserve
the orientation of $E$ as a vector bundle. By Lemma \ref{lemma:orientable-2-group} there is a $G$-invariant complex structure on $E$, so we can look at $E$ as a complex line bundle.
Since $\alpha$ lifts the identity on $\Sigma$, its action is given by multiplication
by a smooth map $f:\Sigma\to\CC^*$, so $\alpha(\lambda)=f(\pi(\lambda))\lambda$
for every $\lambda\in E$, where $\pi:E\to\Sigma$
is the projection. Since $\alpha$ has finite order, there is some integer $k$ such that $f(x)^k=1$ for every $x\in\Sigma$.
This implies that $f$ is constant because $\Sigma$ is connected, and this implies that $\alpha$
commutes with all elements of $G$.
\end{proof}

\begin{lemma}
\label{lemma:line-bundle-on-torus}
Let $L\to T^2$ be a complex line bundle and let
$\Aut(L)\subset\Diff(L)$  denote the group of line bundle
automorphisms of $L$, lifting arbitrary diffeomorphisms of
$T^2$. Let $G<\Aut(L)$ be a finite subgroup.
\begin{enumerate}
\item There is an abelian subgroup $A\leq G$ satisfying
$[G:A]\leq 12\max\{1,|\deg L|\}$.
\item There is a nilpotent subgroup $N\leq G$ of class at most $2$ such that
$[G:N]\leq 12$ and $[N,N]$ is cyclic and acts trivially on $T^2$.
\end{enumerate}
\end{lemma}
\begin{proof}
Let $G_0\leq G$ be the group of elements acting on $T^2$ preserving the orientation.
We have $[G:G_0]\leq 2$.
Statement (1) follows from applying \cite[Proposition 2.10]{M5} to $G_0$.
Let us prove (2). Let $\rho:\Aut(L)\to\Diff(T^2)$ be defined by restricting to
the zero section.
By \cite[Lemma 2.5]{M5} there is an abelian subgroup $B\leq\rho(G_0)$
satisfying $[\rho(G_0):B]\leq 6$. Hence $N=\rho^{-1}(B)$ satisfies
$[G_0:N]\leq 6$. Now $[N,N]$ acts trivially on $T^2$, so $[[N,N],N]=1$
follows from (2) in Lemma \ref{lemma:commuten-infinitesimalment}.
Since the action of $N$ on
the total space of $L$ preserves the orientation,
Lemma \ref{lemma:lin-invariant-surface} implies that $[N,N]$ is cyclic.
\end{proof}

\begin{lemma}
\label{lemma:line-bundle-Jordan}
Let $\Sigma$ be a closed and connected surface, and let $L \rightarrow \Sigma$ be a complex line bundle.
Let $\Aut(L)\subset\Diff(L)$  denote the group of line bundle
automorphisms of $L$, lifting arbitrary diffeomorphisms of
$\Sigma$. Suppose that at least one of the following two conditions holds true.
\begin{enumerate}
\item $\chi(\Sigma)\neq 0$, or
\item $L$ is trivial.
\end{enumerate}
Then there is a constant $C$, depending only on the genus of $\Sigma$,
such that any finite subgroup $G<\Aut(L)$ has an abelian subgroup
$A\leq G$ satisfying $[G:A]\leq C$.
\end{lemma}
\begin{proof}
Suppose first that $\chi(\Sigma)\neq 0$
and that $G$ is a finite group acting effectively on $L$ by complex line bundle automorphisms.
Then $G$ preserves the zero section of $L$, which we identify with $\Sigma$,
and hence by Lemma \ref{lemma:Jordan-fixed-point}
there exists some point $x\in\Sigma$ satisfying $[G:G_x]\leq C_0$ for some $C_0$ depending only
on the genus of $\Sigma$. Applying now Lemma \ref{lemma:linearization-Jordan} to the action of
$G_x$ on the total space of $L$ we conclude the proof.

It suffices now to consider the case in which $L$ is trivial and $\chi(\Sigma)=0$.
By Lemma \ref{lemma:covering-lemma} 
we need only consider the case $\Sigma\cong T^2$,
which follows from (1) in Lemma \ref{lemma:line-bundle-on-torus}.
\end{proof}

\section{Non-free actions}
\label{s:non-free-actions}

In this section $p$ denotes a prime number.

The following result is well known,
see for example \cite[Lemma 5.1]{M7}.

\begin{lemma}
\label{lemma:betti-fixed-point-set} Let $G$ be a finite
$p$-group acting on a manifold $X$. Then
$$\sum_{j\geq 0}b_j(X^G;\ZZ/p)\leq \sum_{j\geq 0}b_j(X;\ZZ/p),$$
where $b_j$ denotes the $j$-th Betti number.
\end{lemma}

\begin{lemma}
\label{lemma:dos-plans} Let $p$ be an odd prime, and let
$U,V\in\SO(4,\RR)$ be two commuting matrices of order $p$. If
$\Ker(U-1)\neq \Ker(V-1)$ and both kernels are nonzero
then $\Ker(UV-1)=0$.
\end{lemma}
\begin{proof}
Let $A=\Ker(U-1)$ and $B=\Ker(V-1)$. By assumption $A\neq 0\neq B$.
Since $U$ and $V$ have odd nontrivial order we necessarily have $\dim A=\dim B=2$,
and $U$ (resp. $V$) acts as a nontrivial rotation on $A^{\perp}$ (resp. $B^{\perp}$).
One easily checks that the only two dimensional subspaces of $\RR^4$ that are
preserved by $U$ are $A$ and $A^{\perp}$. Since $U$ and $V$ commute, $U$ preserves
$B$, so the only possibilities are $B=A$ (which is ruled out by our assumptions)
or $B=A^{\perp}$ (and consequently $B^{\perp}=A$). Then $UV$ preserves both $A$ and
$A^{\perp}$ and its restriction to each of them is a nontrivial rotation. Hence $UV$ has no nonzero
fixed vector.
\end{proof}

\subsection{The set $W(X,A)$}
\label{ss:def-W(X,A)}

Let $X$ be a closed $4$-manifold and let $A$ be a finite abelian group acting
smoothly and effectively on $X$. Define the following subset of $X$:
$$W(X,A)=\bigcup_{a\in A\setminus\{1\}}X^a.$$
The set $W(X,A)$ will appear several times in our arguments, especially in situations where
no element of $A$ has an isolated fixed point. Assuming this condition, the following
lemma gives what for us will be the most important properties of $W(X,A)$.

\begin{lemma}
\label{lemma:sense-punts-fixos-aillats}
Let $X$ be a closed, connected and oriented $4$-manifold.
There exists a constant $C$ with the following property.
Let $p$ be any prime.
Suppose that $A$ is a finite abelian $p$-group acting in a smooth and CTE way on $X$, and that
there exists no $a\in A$ for which the fixed point set $X^a$ has an isolated point.
Let $W=W(X,A)$. Then
\begin{enumerate}
\item $W\subset X$ is a possibly disconnected embedded closed surface, and each
connected component of $W$ is a connected component of $X^a$ for some $a\in A\setminus\{1\}$,
\item $X^a$ is equal to the union of some connected components of $W$ for each $a\in A\setminus\{1\}$,
\item $\sharp\pi_0(W)\leq C$, and each connected component of $W$ has genus at most $C$.
\end{enumerate}
\end{lemma}
\begin{proof}
Let $\Gamma\leq A$ be the $p$-torsion and define:
$$F=\{(a,Z)\mid a\in \Gamma\setminus\{1\},\,Z\text{ connected
component of $X^a$}\}.$$
Since for every $a\in A\setminus\{1\}$ there exists some $r=p^s$
such that $a^r$ has order $p$, and clearly $X^a\subseteq X^{a^r}$, we have
\begin{equation}
\label{eq:formula-W}
W=\bigcup_{(a,Z)\in F}Z.
\end{equation}
By assumption, for each $(a,Z)\in F$, $Z$
is a connected and embedded surface in $X$. We claim that
for every two elements $(a,Z),(a',Z')\in F$ either $Z=Z'$ or $Z\cap
Z'=\emptyset$.
Indeed, if $Z\cap Z'\neq\emptyset$ but $Z\neq
Z'$, then $Z\cap Z'$ is a proper submanifold of $Z$. Applying
Lemma \ref{lemma:dos-plans} to the linearisation of the action
of $\la a,a'\ra$ at some point in $Z\cap Z'$ we would then
deduce the existence of $a''\in A$ with an isolated fixed
point, contradicting our assumption. This proves the claim, and the claim
immediately implies (1).

To prove (2) take an arbitrary $a\in A\setminus\{1\}$ and choose
as before some $r$ such that $a^r$ has order $p$. Since all connected
components of $X^a$ have dimension $2$ and the same happens with $X^{a^r}$,
we conclude that $X^a$ is the union of some (maybe all) connected components of
$X^{a^r}$. Combining this with formula (\ref{eq:formula-W}) we deduce statement (2).

We next prove (3).
Let $U\subset X$ be a $\Gamma$-invariant tubular neighborhood of $W$
and let $V = X \setminus W$. Consider the Mayer--Vietoris sequence with $\ZZ/p$ coefficients
applied to the covering $X=U\cup V$:
$$\cdots \rightarrow  H_\Gamma^n(U \cap V;\ZZ/p)
\rightarrow H_\Gamma^{n+1}(X;\ZZ/p)
\rightarrow
\begin{array}{c}
H_\Gamma^{n+1} (U;\ZZ/p) \\
\oplus \\
H_\Gamma^{n+1}(V;\ZZ/p)
\end{array}
\rightarrow H_\Gamma^{n+1} (U \cap V;\ZZ/p) \cdots$$
Since the action of $\Gamma$ on $V$ (hence also on $U \cap V$) is free, we have
$$H_\Gamma^*(V;\ZZ/p) \simeq H^*(V/\Gamma;\ZZ/p)$$
(and similarly for $U \cap V$). Since $V/\Gamma$ and $(U \cap V)/\Gamma$ are $4$-manifolds, for $n>4$ we have $H_\Gamma^n(V;\ZZ/p) = H_\Gamma^n(U\cap V;\ZZ/p) = 0$. Therefore, the above exact sequence gives us isomorphisms
$$H_\Gamma^6(X;\ZZ/p) \simeq H_\Gamma^6(U;\ZZ/p) \simeq H_\Gamma^6(W;\ZZ/p)$$
Considering the Serre spectral sequence for the fibration $X_G \rightarrow X$, we obtain
\begin{equation}
\label{eq:Serre}
\dim H_\Gamma^6(X;\ZZ/p) \leq \sum_{i+j=6}  \dim H^i(X;\ZZ/p) \otimes H^j(\Gamma;\ZZ/p).
\end{equation}
By Theorem \ref{thm:MS} the $p$-rank of $\Gamma$ is bounded above by a constant depending only on
$X$, so it follows from (\ref{eq:Serre}) that
\begin{equation}
\label{eq:C-prima}
\dim H_\Gamma^6(X;\ZZ/p)\leq C',
\end{equation}
where $C'$ only depends on $X$.

Let $D =\max_p  \sum_{j\geq 0}b_j(X;\ZZ/p)$, where $p$ runs over the set of all primes.
If $a,b\in\Gamma$ and $Z$ is a connected component of $X^a$ then $bZ$ is a
connected component of $X^{bab^{-1}}=X^a$. Since, by Lemma \ref{lemma:betti-fixed-point-set},
$X^a$ has at most $D$
connected components, and each connected component of $W$ is a connected
component of $X^a$ for some $a\in\Gamma$,
the orbits of the action of $\Gamma$ on $\pi_0(W)$
have at most $D$ elements.
Let $Z_1,\dots,Z_l$ be a collection of connected components of $W$ such that
$W$ is the disjoint union of the $\Gamma$-orbits of $Z_1,\dots,Z_l$. Then
\begin{equation}
\label{eq:cota-l}
l\geq\frac{\sharp\pi_0(W)}{D}.
\end{equation}
Define for each $i$ the following two subgroups of $\Gamma$:
\begin{align*}
\Gamma_i &=\{a\in\Gamma\mid aZ_i=Z_i\},\\
\Gamma_{i0}&=\{a\in\Gamma\mid ax=x\text{ for every $x\in Z_i$}\}.
\end{align*}
Remark that $\Gamma_{i0}$ is not the trivial group, since by assumption
$Z_i$ is a connected component of $X^a$ for some $a\in\Gamma\setminus\{1\}$.
Fix some model $E\Gamma\to B\Gamma$ for the universal principal $\Gamma$-bundle.
The inclusion $$E\Gamma\times_{\Gamma_i}Z_i\to E\Gamma\times_{\Gamma_i}(\Gamma Z_i)$$ followed
by the projection $E\Gamma\times_{\Gamma_i}(\Gamma Z_i)\to E\Gamma\times_{\Gamma}(\Gamma Z_i)$
gives a homeomorphism
$$E\Gamma\times_{\Gamma_i}Z_i\cong E\Gamma\times_{\Gamma}(\Gamma Z_i),$$
which induces an isomorphism
\begin{equation}
\label{eq:iso-cohom-equiv}
H_{\Gamma_i}^*(Z_i;\ZZ/p)\simeq H_{\Gamma}^*(\Gamma Z_i;\ZZ/p).
\end{equation}
Hence
\begin{equation}
\label{eq:hac-sis}
H^6_{\Gamma}(X;\ZZ/p)\simeq\bigoplus_{i=1}^l H_{\Gamma}^6(\Gamma Z_i;\ZZ/p)\simeq
\bigoplus_{i=1}^l H_{\Gamma_i}^6(Z_i;\ZZ/p).
\end{equation}
The action of $\Gamma_i$ on $Z_i$ descends to an action of $J_i:=\Gamma_i/\Gamma_{i0}$
on $Z_i$. We claim that $J_i$ acts freely on $Z_i$. This is equivalent to the statement that
if an element $b\in\Gamma$ preserves $Z_i$ and fixes some point of $Z_i$ then it necessarily fixes all points of $Z_i$; this is true because $X^b$ is a possibly disconnected embedded surface (without
isolated points, as we are assuming at this point) and because $X^b\cap Z_i\neq \emptyset$ implies
$Z_i\subseteq X^b$ by Lemma \ref{lemma:dos-plans} (see the argument after formula
(\ref{eq:formula-W})).
Now, an argument similar to the one that led to the isomorphism (\ref{eq:iso-cohom-equiv}) combined
with K\"unneth implies
$$H^6_{\Gamma_i}(Z_i;\ZZ/p)\simeq H^6_{\Gamma_{i0}}(Z_i/J_i;\ZZ/p)\simeq
\bigoplus_{u+v=6}H^u(B\Gamma_{i0};\ZZ/p)\otimes H^v(Z_i/J_i;\ZZ/p),$$
where the second term is the equivariant cohomology of the trivial action of $\Gamma_{i0}$
on $Z_i/J_i$. The rightmost term in the previous formula contains the summand
$$H^6(B\Gamma_{i0};\ZZ/p)\otimes H^0(Z_i/J_i;\ZZ/p),$$
which is nonzero because $\Gamma_{i0}$ is not the trivial group, and hence is of the form
$(\ZZ/p)^s$ for some $s>0$.
It then follows from (\ref{eq:hac-sis}) that
$\dim H^6_{\Gamma}(X;\ZZ/p)\geq l$. Using (\ref{eq:C-prima}) we get $C'\geq l$, so
using (\ref{eq:cota-l}) we obtain
$$\sharp\pi_0(W)\leq C'D.$$
Finally, if $Z$ is a connected component of $W$ then by (1) there exists some $a\in A\setminus\{1\}$
such that $Z$ is a connected component of
$X^a$. Then Lemma \ref{lemma:betti-fixed-point-set} implies that
$b_0(Z;\ZZ/p)+b_1(Z;\ZZ/p)+b_2(Z;\ZZ/p)\leq D$, which implies that the
genus of $Z$ is bounded above by a constant depending on $X$.
\end{proof}

\subsection{Normal abelian $p$-subgroups}

\begin{lemma}
\label{lemma:punt-fix-aillat}
Let $X$ be a closed, connected and oriented $4$-manifold.
There exists a constant $C$ with the following property.
Suppose that $G$ is a finite group acting in a smooth and CTE way on $X$,
let $p$ be any prime and let $A\leq G$ be a normal abelian $p$-subgroup.
If there exists some $a\in A$ and an isolated point in $X^a$ then there
is an abelian subgroup $B\leq G$ satisfying $[G:B]\leq C$ and $X^B\neq\emptyset$.
\end{lemma}
\begin{proof}
Suppose that $a\in A$ and that $X^a$ contains an isolated point. Let
$S\subset X$ be the set of isolated points of $X^a$.
Let $$D =\max_p  \sum_{j\geq 0}b_j(X;\ZZ/p),$$ where $p$ runs over the set of all primes.
Applying Lemma \ref{lemma:betti-fixed-point-set} to
the action of $\la a\ra$ on $X$ we deduce that $\sharp S\leq D$.
Take any point $s\in S$. Since $A$
is abelian, the action of any $a'\in A$ preserves $X^a$ and hence $S$.
Consequently, the stabilizer $A_0$ of $s$ in $A$ satisfies
$[A:A_0]\leq D$.
Let $G_0\leq G$ be the normalizer of $A_0$. Combining Theorem \ref{thm:MS} with
Lemma \ref{lemma:quatre-u-dos} 
we conclude that $[G:G_0]\leq C_1$ for some constant
$C_1$ depending only on $X$.
Applying Lemma
\ref{lemma:betti-fixed-point-set} to the action of $A_0$ on $X$
we deduce that $X^{A_0}$ contains at most $D$ isolated points.
Since $G_0$ normalizes $A_0$, its action on $X$ preserves the
set of isolated fixed points of $A_0$. Hence there is a
subgroup $G_1\leq G_0$ satisfying $[G_0:G_1]\leq  D$ and
preserving (hence, fixing) one of the isolated fixed points of
$A_0$.
By Lemma \ref{lemma:linearization-Jordan},
$G_1$ contains an abelian subgroup $B\leq G_1$ satisfying
$[G_1:B]\leq C_2$, where $C_2$ is a universal constant.
It follows that $[G:B]\leq DC_1C_2$, so we are done.
\end{proof}

\begin{lemma}
\label{lemma:superficie-no-orientable}
Let $X$ be a closed, connected and oriented $4$-manifold.
There exists a constant $C$ with the following property.
Suppose that $G$ is a finite group acting in a smooth and CTE way on $X$,
let $p$ be any prime and let $A\leq G$ be a normal abelian $p$-subgroup. Suppose that:
\begin{enumerate}
\item there is no $a\in A$ such that $X^a$ has an isolated fixed point,
\item there exists some $a\in A$ such that $X^a$ has a connected component
$Z$ which is a nonorientable surface;
\end{enumerate}
then there is an abelian subgroup $B\leq G$ satisfying $[G:B]\leq C$
and an element $b\in B$ such that $X^b$ has $Z$ as one of its connected components.
\end{lemma}
\begin{proof}

Let $C_1$ be the constant given by Lemma \ref{lemma:sense-punts-fixos-aillats} and
let $W=W(X,A)$.
Since $A$ is normal in $G$, the action of $G$ on $X$ preserves $W$.
By (2) in Lemma \ref{lemma:sense-punts-fixos-aillats},
$Z$ is a connected component of $W$.
By (3) in Lemma \ref{lemma:sense-punts-fixos-aillats},
$W$ contains at most $C_1$ connected components
and the genus of $Z$ is not bigger than $C_1$.

Let $G_0\leq G$ be the subgroup of elements that preserve $Z$. We have
$[G:G_0]\leq C_1$. By Lemma \ref{lemma:non-orientable}
there is an abelian subgroup
$B_0\leq G_0$  satisfying $[G_0:B_0]\leq C_2$, where $C_2$ depends only on $C_1$ and $X$, hence only $X$.
Let $a\in A$ be an element whose fixed point set contains $Z$ as a connected component and let $B=\la a,B_0\ra$.
Let $N\to Z$ be the normal bundle. There is a natural morphism $B\to\Aut(N)$ which is injective
by (1) in Lemma \ref{lemma:lin-invariant-surface}. Its image is contained in $\Aut^+(N)$,
the automorphisms preserving the orientation of the total space of $N$ (which is orientable
because $X$ is). By (1) in Lemma \ref{lemma:commuten-infinitesimalment},
it follows that $B$ is abelian. Since $B_0\leq B$, we have
$$[G:B]\leq [G:B_0]=[G:G_0][G_0:B_0]\leq C_1C_2,$$
so the proof of the lemma is complete.
\end{proof}

\begin{lemma}
\label{lemma:Jordan-nonfree}
Let $X$ be a closed, connected and oriented $4$-manifold.
There exists a constant $C$ with the following property.
Suppose that $G$ is a finite group acting in a smooth and CTE way on $X$,
let $p$ be any prime and let $A\leq G$ be a normal abelian $p$-subgroup.
If the action of $A$ on $X$ is not free, then at least one of these
statements holds true:
\begin{enumerate}
\item there exists an abelian subgroup $G_0 \leq G$ such that $[G:G_0]\leq C$,
\item there exists an embedded connected orientable surface $Z\subset X$ of
genus not bigger than $C$ preserved by a subgroup $G_0\leq G$ that satisfies
$[G:G_0]\leq C$.
\end{enumerate}
\end{lemma}
\begin{proof}
Let $C_1,C_2,C_3$ be the constants given by Lemmas \ref{lemma:sense-punts-fixos-aillats},
\ref{lemma:punt-fix-aillat} and \ref{lemma:superficie-no-orientable} respectively.
Let $C=\max\{C_1,C_2,C_3\}$.

If there exists some $a\in A$ such that $X^a$ has an isolated fixed point then
we can apply Lemma \ref{lemma:punt-fix-aillat} and conclude the existence of
an abelian subgroup $G_0\leq G$ satisfying $[G:G_0]\leq C_2$.
If there is no $a\in A$ such that $X^a$ contains an isolated point, and there
is some $b\in A$ such that $X^b$ has a connected component
which is a nonorientable surface then by Lemma \ref{lemma:superficie-no-orientable}
there is an abelian subgroup $G_0\leq G$ satisfying $[G:G_0]\leq C_3$.

Now suppose that for every $a\in A\setminus\{1\}$ the fixed point set $X^a$ is an
embedded orientable surface.
Then $W:=W(X,A)$
is nonempty because by assumption the action of $A$ on $X$ is not free.
By our assumptions and Lemma \ref{lemma:sense-punts-fixos-aillats}, $W$ is a possibly disconnected embedded
orientable surface, and $W$ has at most $C_1$ connected components. Let $Z\subseteq W$ be any connected component. Since $A$ is a normal subgroup of $G$, the action of $G$ on $X$ preserves $W$.
The subgroup $G_0\leq G$ preserving $Z$ satisfies $[G:G_0]\leq\sharp\pi_0(W)\leq C_1$.
By Lemma \ref{lemma:sense-punts-fixos-aillats} the genus of $Z$ is at most $C_1$.
\end{proof}

\section{Diffeomorphisms normalizing an action of $\ZZ/p$ or $\ZZ/p^2$}
\label{s:diffeomorphisms-normalizing}

The following
basic fact will be used in this section and in several other results on free actions to be proved in Subsection \ref{ss:fixed-points-rotation}: suppose that a finite
group $\Gamma$ acts freely and orientation preservingly on an oriented, closed and connected $4$-manifold $X$; then
the Borel construction $X_\Gamma$ is homotopy equivalent to $X/\Gamma$, which
is an orientable, closed and connected $4$-manifold; consequently,
\begin{equation}
\label{eq:coh-eq-accio-lliure}
H_\Gamma^4(X;A)\simeq A,\qquad H_\Gamma^k(X;A)=0\text{ if $k>4$},
\end{equation}
for every abelian group $A$.

\renewcommand{\b}{\operatorname{b}}

\begin{lemma}
\label{lemma:Gamma-automorphism} Let $X$ be a closed, connected and oriented
$4$-manifold. Let $\Gamma=(\ZZ/p)^r$,
where $r=1$ or $2$. Suppose that $\Gamma$ acts on $X$ in a
smooth and CTE
way, and that there exists an automorphism
$\phi\in\Aut(\Gamma)$ and a diffeomorphism $\psi\in\Diff(X)$
acting trivially on $H^*(X)$ in such a way that the diagram
$$\xymatrix{\Gamma\times X\ar[d]_{\phi\times\psi}\ar[r] & X \ar[d]^{\psi} \\
\Gamma\times X\ar[r] & X,}$$ in which the horizontal arrows are
the maps defining the action of $\Gamma$ on $X$, is commutative.
If the order of $\phi$ is not divisible by $p$ and is bigger
than $4$ then the action of $\Gamma$ on $X$ is not free.
\end{lemma}
\begin{proof}
The commutative diagram in the statement of the lemma gives the
following commutative diagram involving the Borel construction
of $X$:
$$\xymatrix{X_{\Gamma}\ar[r]\ar[d] & B\Gamma \ar[d] \\
X_{\Gamma}\ar[r] & B\Gamma,}$$ in which the left (resp. right) hand side vertical
arrow is induced by $(\phi,\psi)$ (resp. $\phi$).
The previous diagram implies the existence of an automorphism of the Serre
spectral sequence with
coefficients in $\ZZ/p$ for the fibration $X_{\Gamma}\to B\Gamma$
which is given, at the level of the second
page, by the morphism
$$\phi^*\otimes\psi^*:H^{\sigma}(B\Gamma;\ZZ/p)\otimes H^{\tau}(X;\ZZ/p)
\to H^{\sigma}(B\Gamma;\ZZ/p)\otimes H^{\tau}(X;\ZZ/p).$$
Crucially, $\phi^*\otimes\psi^*$ commutes
with all the differentials of the spectral sequence.
Since by assumption the action of $\Gamma$ on $X$ is CTE, we have $\psi^*=\id$.

Suppose from now on that $\Gamma$ acts freely on $X$.
Denote the Serre spectral sequence for the fibration $X_{\Gamma}\to B\Gamma$ by $\{(E_u^{\sigma,\tau},d_u^{\sigma,\tau})\}$.

We consider separately the cases $r=1$ and $r=2$.

Consider first the case $\Gamma=\ZZ/p$,
and suppose that $\phi$ acts on $\ZZ/p$ as multiplication by some
$\zeta\in (\ZZ/p)^*$. Then $\phi^*$ acts on
$H^1(B\Gamma;\ZZ/p)=\Hom(H_1(B\Gamma),\ZZ/p)$ as multiplication
by $\zeta$. Let $\b:H^*(B\Gamma;\ZZ/p)\to H^{*+1}(B\Gamma;\ZZ/p)$ be
the Bockstein morphism. To compute the action on higher cohomology
groups, note that if $\theta\in H^1(B\Gamma;\ZZ/p)$ is a
generator then $\b(\theta)$ is a generator of
$H^2(B\Gamma;\ZZ/p)$. By the naturality of $\b$, $\phi^*$ acts on $H^2(B\Gamma;\ZZ/p)$ as
multiplication by $\zeta$. More generally, for any natural
number $k$ and any $\epsilon\in\{0,1\}$,
$\theta^{\epsilon}\b(\theta)^k$ is a generator of
$H^{2k+\epsilon}(B\Gamma;\ZZ/p)$, which implies that the action
of $\phi^*$ on $H^n(B\Gamma;\ZZ/p)$ is given by multiplication
by $\zeta^{[(n+1)/2]}$, where $[t]$ denotes the integral
part of $t$.

Now suppose that the order of $\zeta$ is bigger than $4$. Then
in particular the elements
$1,\zeta,\zeta^2,\zeta^3\in (\ZZ/p)^*$ are pairwise
distinct. This implies that the differentials $d_2,d_3,d_4,d_5$
in the spectral sequence are identically zero, because they
commute with $\phi^*\otimes\id$. Since $E_2^{\sigma,\tau}=0$
for every $\tau>4$, the vanishing of $d_2,\dots,d_5$ implies that
the spectral sequence degenerates. In particular
\begin{align*}
\dim H_{\Gamma}^4(X;\ZZ/p) &=\dim E_2^{0,4}+\dim E_2^{1,3}+\dim E_2^{2,2}+\dim E_2^{3,1}+
\dim E_2^{4,0} \\
&=\sum_{j=0}^4 b_j(X;\ZZ/p)\geq 2,
\end{align*}
and this contradicts (\ref{eq:coh-eq-accio-lliure}).

We next consider the case $\Gamma=(\ZZ/p)^2$.
Suppose that $\alpha,\beta$
are the eigenvalues of $\phi^*$ acting on $H^1(B\Gamma;\ZZ/p)$
(in general $\alpha,\beta$ live in an algebraic extension $\overline{\ZZ/p}$ of
the field $\ZZ/p$, which we assume to be fixed for the arguments that follow).
We want to compute the action of $\phi^*$ on
$E_4^{4,0}\simeq H^4(B\Gamma;\ZZ/p)$. Take any basis
$(\theta_1,\theta_2)$ of $H^1(B\Gamma;\ZZ/p)$. Arguing as in
our discussion about $H^*(B\ZZ/p;\ZZ/p)$ and using K\"unneth we
deduce that $(\b(\theta_1),b(\theta_2),\theta_1\theta_2)$ is a
basis of $H^2(B\Gamma;\ZZ/p)$. Hence if we denote
$$W=H^1(B\Gamma;\ZZ/p)$$
then we can identify in a natural way
(in particular, {\it as representations of
$\la\phi^*\ra$}) $$H^2(B\Gamma;\ZZ/p)\simeq
W\otimes\Lambda^2W.$$
Similarly,
$(\theta_1\b(\theta_1),\theta_2\b(\theta_1),\theta_1\b(\theta_2),\theta_2\b(\theta_2))$
is a basis of $H^3(B\Gamma;\ZZ/p)$, hence
$$H^3(B\Gamma;\ZZ/p)\simeq
W\otimes W$$ canonically.
Similar arguments lead to the following natural isomorphism:
$$H^4(B\Gamma;\ZZ/p)\simeq S^2 W\oplus W\otimes
\Lambda^2W.$$
Accordingly, the eigenvalues of the action of
$\phi^*$ on $H^4(B\Gamma;\ZZ/p)$ are given by
\begin{equation}
\label{eq:eigenvalues}
\alpha^2,\,\alpha\beta,\,\beta^2,\,\alpha^2\beta,\,\alpha\beta^2.
\end{equation}
Of course, it may happen that these eigenvalues are not pairwise distinct; in general, the number of times that a given element $\lambda\in\overline{\ZZ/p}$ appears in the list (\ref{eq:eigenvalues}) is equal to
the dimension of $\Ker(\phi^*-\lambda\id_{H^4(B\Gamma;\ZZ/p)})$.

Since $\dim H^4_{\Gamma}(X;\ZZ/p)=1$ we must have $\dim E^{4,0}_{\infty}\leq 1$.
We have $\dim E^{4,0}_2=5$, hence the dimensions of the images of the differentials
$$d_2^{2,1}:E_2^{2,1}\to E_2^{4,0},\quad
d_3^{1,2}:E_3^{1,2}\to E_3^{4,0},\quad
d_4^{0,3}:E_4^{0,3}\to E_4^{4,0}$$
have to add up at least 4. The weights of the action of $\phi$ on
$E_2^{2,1},E_3^{1,2},E_4^{0,3}$ are the same as the weights of the
action on $H^0(B\Gamma;\ZZ/p)\oplus H^1(B\Gamma;\ZZ/p)\oplus H^2(B\Gamma;\ZZ/p)$, namely $1,\alpha,\beta,\alpha\beta$.
It follows that at least $4$ of the elements in (\ref{eq:eigenvalues})
must belong to the set $\{1,\alpha,\beta,\alpha\beta\}$.
Let us reformulate our last statement in algebraic terms.
Define the following subsets of $\ZZ^2$:
$$S=\{(2,0),(1,1),(0,2),(2,1),(1,2)\},
\qquad
T=\{(0,0),(1,0),(0,1),(1,1)\}.$$
We then have:
\begin{enumerate}
\item[($\star$)] there exists a subset $S'\subseteq S$ such
that $S\setminus S'$ contains at most one element, and a map $f=(f_u,f_v):S'\to T\subset\ZZ^2$ with the property that for every $(u,v)\in S$ we have
$\alpha^u\beta^v=\alpha^{f_u(u,v)}\beta^{f_v(u,v)}.$
\end{enumerate}
Let $R=\{f(w)-w\mid w\in S'\}\subset\ZZ^2$.
We claim that $R$ contains two linearly independent elements of $\ZZ^2$. First note that
$R\neq \{0\}$ for otherwise we would have $f(w)=w$ for all $w$, which is not compatible
with ($\star$) because $S\cap T$ contains a unique element. We also cannot have
$R\subset\ZZ w$ for any $w\in\ZZ^2$, because for every $w\in\ZZ^2$
the intersection $S\cap (T+\ZZ w)$ contains at most $3$ elements,
as one readily checks by plotting the elements of $S$ and $T$; hence $R\subset \ZZ w$ would again contradict ($\star$), so the
claim is proved.

Suppose $(u,v),(u',v')\in R$ are linearly independent, so that
$d:=uv'-u'v$ is nonzero. Since $S,T\subset\{0,1,2\}^2$,
we have $u,v,u',v'\in\{0,1,2\}$ and hence $|d|\leq 4$.
The equalities
$\alpha^u\beta^v=\alpha^{u'}\beta^{v'}=1$
imply that
\begin{equation}
\label{eq:vaps-1}
\alpha^d=\alpha^{uv'-u'v}=(\alpha^u\beta^v)^{v'}(\alpha^{u'}\beta^{v'})^{-v}=1=
(\alpha^u\beta^v)^{u'}(\alpha^{u'}\beta^{v'})^{-u}=
\beta^{vu'-v'u}=\beta^d.
\end{equation}
Consequently, the eigenvalues $\alpha^d,\beta^d$ of
$(\phi^*)^d\in\Aut(H^1(B\Gamma;\ZZ/p))$ are equal to one, so
$(\phi^*)^d$ is a unipotent automorphism. The order of a
unipotent automorphism of a vector space over $\ZZ/p$ is
necessarily a power of $p$. Since the order of $(\phi^*)^d$ is
prime to $p$, it follows that $(\phi^*)^d$ is the identity.
Hence $\phi^*$ is an automorphism of $H^1(B\Gamma;\ZZ/p)$ of
order at most $4$. Since there is a natural isomorphism
$H^1(B\Gamma;\ZZ/p)\simeq\Hom(\Gamma,\ZZ/p)$, the fact that
$(\phi^*)^d$ is trivial implies that $\phi^d\in\Aut(\Gamma)$ is
trivial, so the order of $\phi$ is at most $4$.
\end{proof}

\section{Finite groups acting smoothly on 4-Manifolds with $b_2=0$}
\label{s:b-2-zero}

The goal of this section is to prove the following:

\begin{theorem}
\label{thm:b-2-zero} Suppose that $X$ is a closed connected
$4$-manifold satisfying $b_2(X)=0$. Then $\Diff(X)$ is
Jordan.
\end{theorem}

Let $X$ be a closed connected $4$-manifold satisfying $b_2(X)=0$. To prove Theorem
\ref{thm:b-2-zero} we only need to consider the case $\chi(X)=0$, for if
$\chi(X)\neq 0$ then $\Diff(X)$ is Jordan by the main result in \cite{M4}.
By Lemma \ref{lemma:covering-lemma} we may also assume that $X$ is orientable, so the Betti
numbers of $X$ are $b_0(X)=b_4(X)=1$ and $b_1(X)=b_3(X)=1$.
Let $T$ be the torsion of $H_1(X)$. By the universal coefficient theorem the torsion of $H^2(X)$ is isomorphic to $T$, and by Poincar\'e duality we have
$H^3(X)\simeq H_1(X)$, so the torsion of $H^3(X)$ is also isomorphic to $T$. Hence we have
\begin{equation}
\label{eq:torsio-cohomologia}
H^0(X)\simeq H^1(X)\simeq H^4(X)\simeq\ZZ,
\quad
H^2(X)\simeq T,
\quad
H^3(X)\simeq \ZZ\oplus T.
\end{equation}
Assuming these conditions, we will prove Theorem \ref{thm:b-2-zero} in
Subsection \ref{ss:proof-thm:b-2-zero} below, after introducing a number of preliminary results.
The manifold $X$ will be fixed in the entire section.

\subsection{Rotation morphism}
\label{ss:rotation}

The following construction is used in \cite{M1}. We explain it
here in slightly more intrinsic terms.
Let $e:\RR\to S^1$ be the map $e(t)=e^{2\pi\imag t}$.
Fix a generator $\theta\in H^1(X)$.

Suppose that $\phi\in\Diff(X)$ has finite order and acts
trivially on $H^1(X)$. By the standard averaging trick, we may
then take a $\phi$-invariant 1-form $\alpha\in\Omega^1(X)$
representing $\theta$. Take any $x\in X$, choose a path
$\gamma:[0,1]\to X$ from $x$ to $\phi(x)$ (which means as usual
that $\gamma(0)=x$ and $\gamma(1)=\phi(x)$) and define
$$\rho(\phi)=e\left(\int_{\gamma}\alpha\right)\in S^1.$$
This is clearly independent of the choice of the path $\gamma$.
It is also independent of the choice of $x$. Indeed, if $y\in
X$ denotes another point we may take a path $\eta$ from $y$ to
$x$ and take, as a path from $y$ to $\phi(y)$, the
concatenation of the paths $\eta$, $\gamma$, and $\phi\circ
\eta_{-1}$, where $\eta_{-1}(t)=\eta(1-t)$. The resulting
integral of $\alpha$ is equal to
$$\int_\eta\alpha+\int_\gamma\alpha+\int_{\phi\circ\eta_{-1}}\alpha=
\int_\eta\alpha+\int_\gamma\alpha-\int_{\phi\circ\eta}\alpha
=\int_\eta\alpha+\int_\gamma\alpha-\int_{\eta}\alpha=\int_\gamma\alpha,$$
where the second inequality follows from the assumption that
$\alpha$ is $\phi$-invariant. Finally, we prove that
$\rho(\phi)$ is independent of the choice of $\alpha$. To see
this, suppose that $\beta$ is another $\phi$-invariant $1$-form
representing $\theta$. Then $\beta=\alpha+df$ for some function
$f$. We claim that $f$ is $\phi$-invariant. Indeed, the fact
that both $\alpha$ and $\beta$ are $\phi$-invariant implies
that $\phi^*df=df$, so $d(\phi^*f-f)=0$ and hence $\phi^*f=f+c$
for some constant $c$. Writing $c=\phi^*f-f$ and evaluating at
a point where $f$ attains its maximum (resp. minimum) we
conclude that $c\leq 0$ (resp. $c\geq 0$), so $c=0$. Now we
have, by Stokes's theorem,
$$\int_\gamma\alpha-\int_\gamma\beta=\int_\gamma
df=f(\phi(x))-f(x)=0.$$

We now prove that if $G$ is a finite group acting smoothly on
$X$ and trivially on $H^1(X)$ then the map
$$\rho:G\to S^1$$
is a morphism of groups. Since $G$ is finite we may take a
$G$-invariant $1$-form $\alpha$ representing $\theta$. Let
$x\in X$ be any point, let $g_1,g_2\in G$, and let $\gamma_1$
(resp. $\gamma_2$) be a path from $x$ to $g_1x$ (resp. from $x$
to $g_2x$). The concatenation $\zeta$ of $\gamma_2$ and
$g_2\gamma_1$ is a path from $x$ to $g_2g_1x$. Hence
$$\int_\zeta\alpha=\int_{\gamma_2}\alpha+\int_{g_2\gamma_1}\alpha=
\int_{\gamma_2}\alpha+\int_{\gamma_1}\alpha,$$ where the second
equality follows from the fact that $\alpha$ is $G$-invariant.
It now follows that
$$\rho(g_2g_1)=\rho(g_2)\rho(g_1).$$

\begin{lemma}
Let a finite group $G$ act smoothly on $X$ and trivially on
$H^1(X)$, and assume that $\rho(G)=1$. Let $\pi:Z\to X$ be the
abelian universal cover of $X$. There exists a smooth action of
$G$ on $Z$ lifting the action on $X$, in the sense that
$\pi(g\cdot z)=g\cdot \pi(z)$ for every $g\in G$ and $z\in Z$.
\end{lemma}
\begin{proof}
Fix some base point $x_0\in X$. Choose a $1$-form $\alpha$
representing $\theta$. We can identify
$$Z=\{(x,\gamma)\mid x\in X,\,\gamma\text{ path from $x_0$ to
$x$}\}/\sim,$$ where the equivalence relation $\sim$ identifies
$(x,\gamma)$ with $(x',\gamma')$ if and only if $x=x'$ and
$\int_{\gamma}\alpha=\int_{\gamma'}\alpha$. The later equality
is independent of the choice of $\alpha$. Let us assume from
now on that $\alpha$ is $G$-invariant.

Choose, for every $g\in G$, a path $\eta_g$ from $x_0$ to
$g\cdot x_0$ satisfying $$\int_{\eta_g}\alpha=0.$$ This is
possible because $\rho(g)=1$. Define an action of $G$ on $Z$ as
follows. If $[(x,\gamma)]\in Z$ and $g\in G$ then set
$g\cdot[(x,\gamma)]=[(g\cdot x,g\sharp \gamma)],$ where $g\sharp
\gamma=\eta_g*(g\cdot \gamma)$ and the symbol $*$ denotes
concatenation of paths. Since $\alpha$ is $G$-invariant and
$\int_{\eta_g}\alpha=0$, we have
$$\int_{g\sharp\gamma}\alpha=\int_{\gamma}\alpha.$$
This implies that $g_1\cdot(g_2\cdot [(x,\gamma)]) = g_1g_2 \cdot [(x, \gamma)]$ for every
$g_1,g_2\in G$, which combined with some trivial checks implies
that we have defined an action of $G$ on $Z$ lifting the action
on $X$.
\end{proof}

\begin{lemma}
\label{cor:d-2-vanishes} Suppose that a finite group $G$ acts
smoothly and in a CTE way on $X$, and suppose also that
$\rho(G)=1$. For any abelian group $A$ the
differential
$$d_2^{0,2}:E_2^{0,1}=H^1(X;A)\to E_2^{2,0}=H^2(BG;A)$$
in the second page of the Serre spectral sequence for the
fibration $X_G\to BG$ with coefficients in $A$ vanishes identically.
\end{lemma}
\begin{proof}
Let $\pi:Z\to X$ be as in the previous lemma, and take a lift of
the action of $G$ on $X$ to an action on $Z$. We have a
Cartesian diagram of fibrations
$$\xymatrix{Z_G \ar[r]^{\pi} \ar[d] & X_G \ar[d] \\
BG\ar@{=}[r] & BG}.$$ The vanishing of $d_2^{0,2}$ follows from
the naturality of the Serre spectral sequence and the fact that
$H^1(Z;A)=0$.
\end{proof}

\subsection{Fixed points and the rotation morphism}
\label{ss:fixed-points-rotation}

Fix a prime $p$ for the present subsection, and suppose that
$T$ (which, recall, is the torsion of $H_1(X)$) has $ap^t$ elements,
where $a\geq 1,t\geq 0$ are integers and $p$ does not divide $a$. Let $T_p$
be the $p$-part of $T$, i.e., the subgroup of elements whose order is a power of $p$.
We have
$$\sharp T_p=p^t.$$

The proof of the following theorem is given in the appendix.

\begin{theorem}
\label{thm:appendix}
Let $a,b$ be natural numbers and let $c=\min\{a,b\}$.
For any natural number $d$, any nonnegative integer $k$ and any prime $p$ we have
$$H^k((\ZZ/p^a)^d;\ZZ/p^b)\simeq (\ZZ/p^c)^{\left(k+d-1 \atop d-1\right)},$$
where we consider on the coefficient group $\ZZ/p^b$ the trivial $(\ZZ/p^a)^d$-module
structure.
\end{theorem}

By Poincar\'e duality, (\ref{eq:torsio-cohomologia}) implies
$$H_0(X)\simeq H_3(X)\simeq H_4(X)\simeq \ZZ$$
and
$$H_1(X)\simeq\ZZ\oplus T,\qquad H_2(X)\simeq T.$$
Let $r$ be an integer satisfying $r>t$. Then we have (see (\ref{eq:ext}) in the Appendix)
$$\Hom(T,\ZZ/p^r)\simeq T_p,\qquad \Ext(T,\ZZ/p^r)\simeq T_p.$$
Using the universal coefficient theorem (see (\ref{eq:ucf-a}) in the Appendix) we compute
\begin{equation}
\label{eq:coh-coeff-finits-1}
H^0(X;\ZZ/p^r)\simeq H^1(X;\ZZ/p^r)\simeq H^4(X;\ZZ/p^r)
\simeq\ZZ/p^r,
\end{equation}
\begin{equation}
\label{eq:coh-coeff-finits-2}
H^2(X;\ZZ/p^r)\simeq\Hom(T,\ZZ/p^r)\oplus\Ext(\ZZ\oplus T,\ZZ/p^r)\simeq T_p\oplus T_p,
\end{equation}
\begin{equation}
\label{eq:coh-coeff-finits-3}
H^3(X;\ZZ/p^r)\simeq\Hom(Z,\ZZ/p^r)\oplus\Ext(T,\ZZ/p^r)\simeq\ZZ/p^r\oplus T_p.
\end{equation}

\begin{lemma}
\label{lemma:rank-2-trivial-rotation}
Let $r$ be the least integer bigger than $5t/3$.
No smooth CTE action of $\Gamma:=(\ZZ/p^r)^2$ on $X$ satisfying $\rho(\Gamma)=1$
is free.\end{lemma}
\begin{proof}
Suppose that $\Gamma$ acts smoothly, freely, and in a CT way on $X$.
By (\ref{eq:coh-eq-accio-lliure}) we have
\begin{equation}
\label{eq:coh-accio-lliure}
H^4(X_{\Gamma};\ZZ/p^r)\simeq\ZZ/p^r.
\end{equation}
The entries in the second page
of the Serre spectral
sequence $\{(E_s^{ij},d_s^{ij})\}$ for the
fibration $X_{\Gamma}\to B\Gamma$ with coefficients in $\ZZ/p^r$
take the form
$$E_2^{i,j}=H^i(\Gamma;H^j(X;\ZZ/p^r)),$$
where $\Gamma$ acts trivially on $H^j(X;\ZZ/p^r)$.
By Theorem \ref{thm:appendix} and (\ref{eq:coh-coeff-finits-1})--(\ref{eq:coh-coeff-finits-3})
the matrix $(\log_p\sharp E_2^{ij})_{ij}$ has the following entries (note that $r>t$):
$$\begin{array}{|c|c|c|c|c|c|c}
0 & 0 & 0 & 0 & 0 & 0 & \dots \\
\hline
r & 2r & 3r & 4r & 5r & 6r & \dots \\
\hline
\mathbf{r+t} & 2(r+t) & 3(r+t) & 4(r+t) & 5(r+t) & 6(r+t) & \dots \\
\hline
2t & \mathbf{4t} & 6t & 8t & 10t & 12t & \dots \\
\hline
r & 2r & \mathbf{3r} & 4r & 5r & 6r & \dots \\
\hline
r & 2r & 3r & 4r & \mathbf{5r} & 6r & \dots \\
\hline
\end{array}
$$
The isomorphism (\ref{eq:coh-accio-lliure}) implies that
\begin{equation}
\label{eq:tamany-E-infinit}
\sharp E_{\infty}^{4,0}\leq p^r.
\end{equation}
Now assume that $\rho(\Gamma)=1$.
By Lemma \ref{cor:d-2-vanishes} we have $d_2^{0,1}=0$, which implies by
the multiplicativity of the Serre spectral sequence that
$d_2^{2,1}=0$.
Hence, the only way the cardinal of $E_*^{4,0}$ can drop from $p^{5r}$
to an integer not bigger than $p^r$ is by quotienting through the
images of the differentials
$$d_2^{1,2}:E_3^{1,2}\to E_3^{4,0}=E_2^{4,0},\qquad
d_4^{0,3}:E_4^{0,3}\to E_4^{4,0}.$$
More precisely, we can estimate
$$\sharp E_5^{4,0}\geq \sharp E_2^{4,0}(\sharp E_3^{1,2})^{-1}(\sharp E_4^{0,3})^{-1}
\geq \sharp E_2^{4,0}(\sharp E_2^{1,2})^{-1}(\sharp E_2^{0,3})^{-1}=p^{5r-4t-(r+t)}=p^{4r-5t}>p^r,$$
where the last inequality follows from $r>5t/3$.
This contradicts (\ref{eq:tamany-E-infinit}), so the proof of the lemma is complete.
\end{proof}

\begin{lemma}
\label{lemma:subgrup-ciclic-index-petit}
There exists a constant $C$, independent of $p$,
such that any finite abelian $p$-group $A$
acting freely and in a smooth and CT way on $X$ and satisfying $\rho(A)=1$
has a cyclic subgroup $A_c\leq A$ satisfying
$[A:A_c]\leq C$.
\end{lemma}
\begin{proof}
For every prime $q$ let $t_q$ be defined by $\sharp T_q=q^{t_q}$ and
let $r_q$ be the least integer bigger than $5t_q/3$.
Let $R$ be the number resulting from applying Theorem \ref{thm:MS} to $X$. Define
$$C:=\max_q q^{(R-1)(r_q-1)}.$$
This is a finite number, because $\sharp T_q$ is different from $1$ only
for finitely many primes $q$.

Let $A$ be an abelian $p$-group acting freely, smoothly, and in a CT way
on $X$, and satisfying $\rho(A)=1$. Choose an isomorphism
$$A\simeq \ZZ/p^{e_1}\oplus\dots\oplus \ZZ/p^{e_s},$$
where $e_1\geq\dots\geq e_s\geq 1$.
By Theorem \ref{thm:MS}
we have $s\leq R$ and by Lemma \ref{lemma:rank-2-trivial-rotation}
we have $e_i<r_p$ for every $i\geq 2$. 
Define $A_c$ to be the subgroup of $A$
corresponding to the first summand $\ZZ/p^{e_1}$. Then we have
$[A:A_c]\leq p^{(R-1)(r_p-1)}\leq C.$
\end{proof}

\begin{lemma}
\label{lemma:rank-3-or-higher}
No smooth CTE action of $\Gamma=(\ZZ/p^{t+1})^3$ on $X$ is free.
\end{lemma}
\begin{proof}
Suppose that $\Gamma$ acts smoothly and in a CTE way on $X$.
Let $r=t+1$.
Consider the Serre
spectral sequence $\{(E_s^{ij},d_s^{ij})\}$
for the fibration $X_{\Gamma}\to B\Gamma$ with coefficients in
$\ZZ/p^r$. The matrix $(\log_p\sharp E_2^{ij})_{ij}$ has entries
$$\begin{array}{|c|c|c|c|c|c|c}
0 & 0 & 0 & 0 & 0 & 0 & \dots \\
\hline
r & 3r & 6r & 10r & 15r & 21r & \dots \\
\hline
\mathbf{r+t} & 3(r+t) & 6(r+t) & 10(r+t) & 15(r+t) & 21(r+t) & \dots \\
\hline
2t & \mathbf{6t} & 12t & 20t & 30t & 42t & \dots \\
\hline
r & 3r & \mathbf{6r} & 10r & 15r & 21r & \dots \\
\hline
r & 3r & 6r & 10r & \mathbf{15r} & 21r & \dots \\
\hline
\end{array}
$$
Arguing as in the previous lemma we estimate
\begin{align*}
\sharp E_{\infty}^{4,0}&=\sharp E_4^{4,0}\geq \sharp E_2^{4,0}
(\sharp E_2^{2,1})^{-1}(\sharp E_3^{1,2})^{-1}(\sharp E_4^{0,3})^{-1} \\
&\geq\sharp E_2^{4,0}
(\sharp E_2^{2,1})^{-1}(\sharp E_2^{1,2})^{-1}(\sharp E_2^{0,3})^{-1} \\
& =p^{15r-6r-6t-(r+t)}=p^{8r-7t}>p^r,
\end{align*}
so we have $\sharp H_{\Gamma}^4(X;\ZZ/p^r)>p^r$, which is not compatible with the action of $\Gamma$ being free.
\end{proof}

\begin{lemma}
\label{lemma:cyclic-by-cyclic} Let $\Gamma$ be a finite $p$-group
sitting in a short exact sequence
$$1\to K\to \Gamma\stackrel{\pi}{\longrightarrow} Q\to 1$$
with $Q$ cyclic.
Suppose that $A\leq K$ is a cyclic subgroup, and that $A$ is normal
in $\Gamma$.
Assume that $\Gamma$ acts smoothly and in a CT way on $X$ and that
$\rho(A)=1$. If the action of $A$ on $X$ is free, then
there is an abelian subgroup $A'\leq \Gamma$ containing $A$ and
satisfying
$$[\Gamma:A']\leq 2p^t[K:A].$$
\end{lemma}
\begin{proof}
Take an element $\gamma\in \Gamma$ such that $\pi(\gamma)$ generates
$Q$. Since $A$ is normal in $\Gamma$, conjugation defines a morphism $\Gamma\to\Aut(A)$.
Applying this to $\gamma$ we obtain, as in the proof of Lemma \ref{lemma:Gamma-automorphism},
a commutative diagram
$$\xymatrix{A\times X\ar[r]\ar[d] & X \ar[d] \\
A\times X\ar[r] & X,}$$ in which the horizontal arrows are
the maps defining the action of $A$ on $X$, the left hand
vertical arrow sends $(g,x)$ to $(\gamma g\gamma^{-1},\gamma
x)$, and the right hand side vertical arrow sends $x$ to
$\gamma x$. At the level of Borel constructions we obtain a
commutative diagram
$$\xymatrix{X_{A}\ar[r]\ar[d] & BA \ar[d] \\
X_{A}\ar[r] & BA,}$$ in which the right hand side vertical
arrow is induced by the map
$$c(\gamma):A\to A,\qquad c(\gamma)(g)=\gamma g\gamma^{-1}.$$
Consider the Serre
spectral sequence for $X_{A}\to BA$ with integer
coefficients. The previous
diagram implies the existence of an automorphism of the Serre
spectral sequence which is given, at the level of the second
page, by the morphism
$$\phi:H^{\sigma}(BA;H^{\tau}(X))\to H^{\sigma}(BA;H^{\tau}(X)).$$
Crucially, $\phi$ commutes with all the differentials of the spectral sequence.

Since $A$ acts trivially on $H^*(X)$, in order to understand $\phi$ it
will suffice for our purposes to compute
$c(\gamma)^*:H^{\sigma}(BA)\to H^{\sigma}(BA)$. Suppose
that
$$A\simeq\ZZ/p^a.$$
Then, thinking of the group structure
on $A$ in additive terms, the action of $c(\gamma)$ on $A$
is given by multiplication by some
$$\zeta\in
(\ZZ/p^a)^*.$$ We claim that $H^{\sigma}(BA)=0$ if $\sigma$
is odd and $H^{\sigma}(BA)\simeq\ZZ/p^a$ if $\sigma>0$ is
even. Furthermore if $\lambda$ is a generator of $H^2(BA)$
then $\lambda^k$ is a generator of $H^{2k}(BA)$. To prove
these claims, identify $A$ with the group $\mu_{p^a}$ of
$p^a$-th roots of units in $S^1$. Taking as a model for the
classifying space of $A$ the quotient $ES^1/\mu_{p^a}$, we
identify $BA$ with the total space of a circle bundle over
$BS^1$ whose first Chern class is $p^a$ times a generator of
$H^2(BS^1)\simeq\ZZ$. Then the claim follows from applying the
Gysin exact sequence to this bundle.
As a consequence, it suffices to understand $c(\gamma)^*$
acting on $H^2(BA)$. By the universal coefficient theorem we
have
$$H^2(BA)=\Ext^1(H_1(BA),\ZZ).$$
We have a natural identification $H_1(BA)\simeq A$, so
$c(\gamma)^*$ acts on $H_1(BA)$ as multiplication by $\zeta$.
Fix a surjection $\ZZ\to H_1(BA)$ and consider the resulting
commutative diagram with exact rows, and whose vertical arrows
are multiplication by some integer $z\in\ZZ$ representing
$\zeta\in(\ZZ/p^a)^*$:
$$\xymatrix{0\ar[r] & \ZZ \ar[r]\ar[d] & \ZZ \ar[r]\ar[d] & H_1(BA)\ar[d] \ar[r] & 0\\
0\ar[r] & \ZZ \ar[r]& \ZZ \ar[r]& H_1(BA)\ar[r] & 0}.$$
Applying $\Hom_{\ZZ}(\cdot,\ZZ)$ and its derived functors we
get a commutative diagram with exact rows from which it is easy
to conclude that $c(\gamma)^*$ acts on
$\Ext^1(H_1(BA),\ZZ)=H^2(BA)$ as multiplication by
$\zeta$. Consequently, $c(\gamma)^*$ acts on $H^4(BA)$
as multiplication by $\zeta^2$. In other words, we may identify
$H^4(BA)$ with $A$ in such a way that the action of $c(\gamma)$ on $H^4(BA)$
corresponds in $A$ with conjugation by $\gamma^2$.

Suppose from now on that $A$ acts freely on $X$, which
implies
$H^4_{A}(X)\simeq\ZZ$.
From (\ref{eq:torsio-cohomologia}) and the previous description of $H^*(BA)$ we deduce that
the left bottom corner
of the second page of the
spectral sequence with integer coefficients $\{(E_u^{\sigma,\tau},d_u)\}$ for the
fibration $X_{A}\to BA$
is isomorphic to:
$$\begin{array}{|c|c|c|c|c|c}
0 & 0 & 0 & 0 & 0 & 0 \\
\hline
H^0(BA;\ZZ) & 0 & H^2(BA;\ZZ) & 0 & H^4(BA;\ZZ) & 0 \\
\hline
H^0(BA;\ZZ\oplus T_p) & H^1(BA;T_p) & H^2(BA;\ZZ\oplus T_p) & H^3(BA;T_p) & H^4(BA;\ZZ\oplus T_p) & 0 \\
\hline
H^0(BA;T_p)  & H^1(BA;T_p)  & H^2(BA;T_p) & H^3(BA;T_p) & H^4(BA;T_p)  & 0 \\
\hline
H^0(BA;\ZZ) & 0 & H^2(BA;\ZZ) & 0 & H^4(BA;\ZZ) & 0 \\
\hline
H^0(BA;\ZZ) & 0 & H^2(BA;\ZZ) & 0 & H^4(BA;\ZZ) & 0 \\
\hline
\end{array}$$
Furthermore, the action of $\phi$ is induced in each term by the action
of $c(\gamma)$ on $BA$ and the trivial action on coefficients.

The convergence of the spectral sequence to the equivariant
cohomology implies that $E_{\infty}^{4,0}$ can be naturally
identified with a subgroup of $H^4_{A}(X)$. Since
$E_{\infty}^{4,0}$ is a quotient of $E_2^{4,0}\simeq
H^4(BA)\simeq\ZZ/p^a$ and $H^4_{A}(X)$ is torsion free, we
necessarily have $E_{\infty}^{4,0}=0$. There are only three
differentials that can contribute to kill $E_2^{4,0}$:
$$d_2^{2,1}:E^{2,1}_2\to E_2^{4,0}, \qquad
d_3^{1,2}:E^{1,2}_3\to E_3^{4,0} \qquad \text{and} \qquad
d_4^{0,3}:E_4^{0,3}\to E_4^{4,0}.$$ We have $d_2^{2,1}=0$ by
Lemma \ref{cor:d-2-vanishes} and the multiplicativity of
the spectral sequence. So we can naturally identify
$E_3^{4,0}\simeq E_2^{4,0}\simeq H^4(BA)$. We also have
$E_3^{1,2}\simeq E_2^{1,2}\simeq H^1(BA;T_p)$. Hence we can
identify the source and target of $d_3^{1,2}$ with
$$d_3^{1,2}:H^1(BA;T_p)\to H^4(BA).$$
Denote by
$$M\subseteq H^4(BA)$$
the image of $d_3^{1,2}$. We have
$$\sharp M\leq \sharp H^1(BA;T_p)\leq p^t,$$
where the second inequality follows from the universal coefficient theorem
$$H^1(BA;T_p)\simeq \Hom(H_1(BA),T_p)\oplus \Ext(H_0(BA),T_p)=
\Hom(H_1(BA),T_p),$$
the equality $\sharp T_p=p^t$, and the fact that $A$ is cyclic.

Next we can identify the source and target of $d_4^{0,3}$ with
$$d_4^{0,3}:\Ker d_3^{0,3}\to H^4(BA)/M.$$
This map has to be surjective in order for $E_4^{4,0}$ to vanish.
At this point we are going to use the fact that $d_4^{0,3}$ commutes
with the action induced by $\phi$. We can identify $\Ker d_3^{0,3}$ with
a subgroup of $H^0(BA;\ZZ\oplus T_p)$, on which the action of $\phi$ is trivial.
So in order for $d_4^{0,3}$ to be surjective the action on $H^4(BA)/M$
induced by $\phi$ has to be trivial.

Denote for convenience $N=H^4(BA)$. Since the action induced by $\phi\in\Aut(N)$
on $N/M$ is trivial, one can define a morphism (using additive notation)
$$\delta:N\to M,\qquad \delta(n)=n-\phi(n).$$
We have $\Ker\delta=N^{\phi}=\{n\in N\mid \phi(n)=n\}$ so
$$[N:N^{\phi}]=[N:\Ker\delta]\leq \sharp M\leq p^t.$$
We have seen above that we can identify $N\simeq A$ in such a way that
$\phi$ corresponds to the map $A\to A$ given by $a\mapsto \gamma^2a\gamma^{-2}$.
It thus follows that
$$A_\gamma=\{a\in A\mid a=\gamma^2a\gamma^{-2}\}$$
satisfies $[A:A_\gamma]\leq p^t$.
Let $A'\leq\Gamma$ be the subgroup generated by $A_\gamma$ and $\gamma^2$.
It follows from the definition of $A_\gamma$ that $A'$ is abelian.
Since $\pi(\gamma)$ is a generator of $Q$ we may bound
$$[\Gamma:A']\leq 2[K:A_\gamma]\leq 2[K:A][A:A_\gamma]\leq 2p^t [K:A],$$
so the proof of the lemma is complete.
\end{proof}

\subsection{CTE actions of finite $p$-groups}

\begin{lemma}
\label{lemma:Jordan-nonfree-MNA-p} There exists a constant $C$,
independent of $p$, such that for any finite $p$-group $G$ and
any smooth CTE action of $G$ on $X$ the following holds. Let $A\leq
G$ be a MNAS. If the action of $A$ on $X$ is not free, then
$[G:A]\leq C$.
\end{lemma}
\begin{proof}
Let $C_0$ be the constant given by Lemma \ref{lemma:Jordan-nonfree}.
Suppose that $G$ and $A$ satisfy the hypothesis of the statement.
By Lemma \ref{lemma:Jordan-nonfree} there exists a subgroup $G_0\leq G$
satisfying $[G:G_0]\leq C_0$ and such that $G_0$ is abelian or there
exists an embedded connected oriented surface $Z\subset X$ preserved
by $G_0$ and of genus $\leq C_0$. If $G_0$ is abelian then by
Lemma \ref{lemma:index-MNA} we have
$$[G:A]\leq C_0^{r^2+1},$$
where $r$ is the constant given by Mann--Su's Theorem \ref{thm:MS}
applied to $X$.

Now assume that $G_0$ is not abelian, so that we have the surface
$Z$ at our disposal. Since $Z$ is orientable, so is its normal bundle $N\to Z$. We can identify the degree of $N$ with the self-intersection $Z\cdot Z$, which is
equal to $0$ because $b_2(X)=0$. Hence $Z$ is an oriented embedded surface with
trivial normal bundle $N\to Z$.
By Lemma \ref{lemma:orientable-2-group}, $N$ admits a $G_0$-invariant
complex structure.
By (2) in Lemma \ref{lemma:line-bundle-Jordan} there is an abelian subgroup $B\leq G_0$ satisfying
$[G_0:B]\leq C_1$, where the constant $C_1$ depends only on the genus of $Z$
and hence can be bounded above by a constant depending only on $X$. Applying again
Lemma \ref{lemma:index-MNA} we conclude that
$$[G:A]\leq (C_0C_1)^{r^2+1},$$
where $r$ is as above, so the proof of the lemma is now complete.
\end{proof}

\begin{lemma}
\label{prop:p-groups} There exists a constant $C$ such that:
for any prime $p$, any finite $p$-group $G$ and any smooth CTE action
of $G$ on $X$ there is an abelian subgroup $A\leq G$ such that
$[G:A]\leq C$.
\end{lemma}
\begin{proof} Recall that for every prime $p$ we denote by $T_p$ the $p$-torsion
of $H_1(X)$. Since $H_1(X)$ is finitely generated, we have $\sharp T_p=1$ except
for finitely many $p$'s, so
$$C_T:=\max_p\sharp T_p$$
is finite.

Let $C_R$ be the number resulting from applying Lemma \ref{lemma:subgrup-ciclic-index-petit} to $X$.

Let $p$ be a prime. Suppose given a smooth CTE action of a finite $p$-group
$G$ on $X$ and let $\rho:G\to S^1$ be the rotation morphism. Let
$G_0=\Ker\rho$. Let $K\leq G_0$ be a MNAS.

We distinguish two cases, depending on whether the action of $K$ on $X$ is free or not.

Suppose first of all that the action of $K$ on $X$ is not free.
By Lemma \ref{lemma:Jordan-nonfree-MNA-p} we have
$[G_0:K]\leq C$. Let $G'\leq G$ be the normalizer of $K$.
Since $G_0$ is a normal subgroup of $G$,
by Theorem \ref{thm:MS} and Lemma \ref{lemma:quatre-u-dos} we have
$[G:G']\leq C'$. Let $A\leq G'$ be a MNAS containing
$K$. Then the action of $A$ on $X$ is not free, so by Lemma
\ref{lemma:Jordan-nonfree-MNA-p} we have $[G':A]\leq C''$. It
follows that $[G:A]$ is bounded above by a constant which depends
neither on $p$ nor on $G$.

Assume, for the rest of the proof, that the action of $K$ on $X$
is free. By Lemma \ref{lemma:subgrup-ciclic-index-petit}, there is
a cyclic subgroup $A\leq K$ satisfying $[K:A]\leq C_R$.

Let $Q=G_0/K$. Since $K$ is a MNAS of $G_0$, the action of $G_0$ on
$K$ given by conjugation induces an effective action of $Q$ on $K$,
which allows us to identify $Q$ with a subgroup of $\Aut(K)$.
Let
$$f:\Aut_A(K)\to\Aut(A)$$
be the restriction map (we use here and below the notation introduced in Lemma \ref{lemma:MS-1}).
Since $A$ is cyclic, $\Aut(A)$ is also cyclic. Hence $S:=f(Q\cap\Aut_A(K))$
is cyclic. Let $q\in Q\cap \Aut_A(K)$ be an element such that $f(q)$ generates
$S$. Let $Q'=\la q\ra\leq Q$. We have
$$[Q:Q']\leq [Q:Q\cap \Aut_A(K)]\cdot \sharp\Ker f
\leq [\Aut(K):\Aut_A(K)]\cdot\sharp\Ker f.$$
Applying Lemmas \ref{lemma:MS-1} and \ref{lemma:MS-2} (and noting that $\Ker f=\Aut_A^0(K)$)
we conclude that there is a constant $C_L$, depending only on $X$, such that
\begin{equation}
\label{eq:Qhat-vs-Q}
[Q:Q']\leq C_L.
\end{equation}
Let $G_0'$ be the preimage of $Q'$ via the projection $G_0\to Q$. We have a short
exact sequence
$$1\to K\to G_0'\to Q'\to 1,$$
and $A$ is normal in $G_0'$ because the elements of $Q'$ belong to $\Aut_A(K)$.
We may thus apply Lemma
\ref{lemma:cyclic-by-cyclic} and conclude
the existence of an abelian subgroup $A'\leq G_0'$ containing
$A$ and satisfying
$$[G_0':A']\leq 2C_T[K:A]\leq 2C_TC_R.$$
Hence
$$[G_0:A']\leq [G_0:G_0'][G_0':A']\leq 2C_LC_TC_R.$$
By Lemma \ref{lemma:subgrup-ciclic-index-petit} there is a
cyclic subgroup
$$A'_c\leq A'$$
satisfying $[A':A'_c]\leq C_R$.
Since
$$[G_0:A'_c]\leq 2C_LC_TC_R^2$$
gives an upper bound that depends only on $X$, applying Lemma \ref{lemma:quatre-u-dos} and
Theorem \ref{thm:MS} we conclude that the normalizer
$$G'=N_G(A'_c)$$
satisfies $[G:G']\leq C_N$ for a constant $C_N$ depending only on $X$. We have a short exact sequence
$$1\to G'\cap G_0\to G'\to G'/(G'\cap G_0)\to 1.$$
The group $G'/(G'\cap G_0)$ can be identified with a subgroup of $G/G_0\simeq\rho(G)<S^1$,
so $G'/(G'\cap G_0)$ is cyclic, and clearly $A'_c\leq G'\cap G_0$.
So we can 
apply Lemma \ref{lemma:cyclic-by-cyclic} to the inclusion $A'_c\leq G'\cap G_0$ and conclude the existence of an abelian subgroup $A''\leq G'$
satisfying
$$[G':A'']\leq 2C_T[G'\cap G_0:A'_c]\leq 2C_T[G_0:A'_c]\leq  4C_LC_T^2C_R^2$$
and hence
$$[G:A'']\leq 4C_LC_T^2C_R^2C_N.$$
This finishes the proof of the lemma.
\end{proof}

\subsection{Proof of Theorem \ref{thm:b-2-zero}}
\label{ss:proof-thm:b-2-zero}

Let $\pP$ be the collection of all finite $p$-subgroups (for
all primes $p$) of $\Diff(X)$ which act in a CT way on $X$. Let $\tT$ be the
collection of all finite subgroups $G<\Diff(X)$ which act in a CT way on
$X$ and such that there exist different primes $p,q$, a normal
Sylow $p$-subgroup $P\leq G$ and a Sylow $q$-subgroup $Q\leq G$
such that $G=PQ$ and both $P$ and $Q$ are nontrivial. By the
main theorem in \cite{MT} it suffices to prove the existence of
a constant $C$ such that any $G\in\pP\cup\tT$ has an abelian
subgroup $A\leq G$ satisfying $[G:A]\leq C$.

The existence of $C$ for elements of $\pP$ is a consequence of
Lemma \ref{prop:p-groups}.

Let $R$ be the number given by Theorem \ref{thm:MS} applied to $X$, let
$$C_A=\max\{4,\max\{ \sharp\GL(R,\ZZ/p)\mid \sharp T_p\neq 1\}\}.$$

Suppose that $G=PQ\in\tT$, with $P$ a normal $p$-subgroup
of $G$ and $Q$ a $q$-subgroup of $G$, $p\neq q$. By Lemma
\ref{prop:p-groups} there are abelian subgroups $P_0\leq P$ and
$Q_0\leq Q$ satisfying $[P:P_0]\leq C$ and $[Q:Q_0]\leq C$.
Let $Q_0'$ be the normalizer of $P_0$ in $Q_0$. Since $Q_0'=Q_0\cap N_G(P_0)$,
by Theorem \ref{thm:MS} and Lemma \ref{lemma:quatre-u-dos}, there exists a constant $C'$ such that $$[Q_0:Q_0'] \leq [G:N_G(P_0)]\leq C'.$$ Then $G_0 = P_0Q_0'$ satisfies $[G:G_0] \leq CC'$.

Conjugation gives a morphism $c:Q_0'\to \Aut(P_0)$. Let $P_0[p]$
be the $p$-torsion of $P_0$. This is a characteristic subgroup
of $P_0$, so restriction gives a natural morphism
$r:\Aut(P_0)\to\Aut(P_0[p])$. Since $Q_0'$ is a $q$-group and
$q\neq p$, $\Ker c=\Ker r\circ c$. (This is a standard fact in finite
group theory, but we sketch an argument for the reader's convenience:
if $\phi\in\Aut(P_0)$ belongs to $\Ker r$ then we may write $\phi=\Id+\psi$
using additive notation on $P_0$, where $\psi\in\Hom(P_0,P_0)$ satisfies $\psi(x)\in pP_0$
for every $x$, and hence $\psi(p^kP_0)\leq p^{k+1}P_0$ for every $k$;
using the binomial's formula and induction on $r$ we prove that $\phi^{p^r}-\Id$ sends
$P_0$ to $p^rP_0$, so if $r$ is big enough then $\phi^{p^r}=\Id$; this proves
that $\phi$ is a $p$-element in $\Aut(P_0)$ and justifies the equality $\Ker c=\Ker r\circ c$.)

To finish the proof we distinguish two cases.

Suppose first of all that $[Q_0':\Ker r\circ c]>C_A$.
We claim that in this case $\sharp T_p=1$. Indeed, otherwise we would have
$\sharp\Aut P_0[p]\leq \sharp\GL(R,\ZZ/p)\leq C_A$, which combined with
$[Q_0':\Ker r\circ c]\leq \sharp\Aut P_0[p]$ would lead to a contradiction.
Next we claim that the action of $P_0$ on $X$ is not free.
If the rank of $P_0[p]$
is $1$ or $2$ this follows from Lemma \ref{lemma:Gamma-automorphism},
and if it is $\geq 3$ then it follows from Lemma \ref{lemma:rank-3-or-higher}.
Once we know that the action of $P_0$ is not free, applying Lemma
\ref{lemma:Jordan-nonfree-MNA-p}
we conclude that $G_0$ has an abelian subgroup of bounded index.

Next suppose that $[Q_0':\Ker r\circ c]\leq C_A$.
Then the group
$Q_1=\Ker c=\Ker r\circ c$ commutes with $P_0$, so $A=P_0Q_1\leq G$ is
an abelian subgroup satisfying $[G:A]\leq C_ACC'$. This concludes
the proof of the theorem.

\section{Proofs of Theorems \ref{thm:2-step-nilpotent}
and \ref{thm:2-nilpotent-alpha}}
\label{s:proofs-main-theorems}

Assume for the entire present section that $X$ is a closed, connected and oriented $4$-manifold.
This is more restrictive than the situation considered in Theorem \ref{thm:2-step-nilpotent}, but Lemma \ref{lemma:covering-lemma} allows us to reduce the general case to this setting.
If $b_2(X)=0$ then both Theorems \ref{thm:2-step-nilpotent}
and \ref{thm:2-nilpotent-alpha} follow from Theorem \ref{thm:b-2-zero}. Hence,
we also assume in this section that $b_2(X)\neq 0$.

By Lemma \ref{lemma:minkowski}, both in Theorems \ref{thm:2-step-nilpotent}
and \ref{thm:2-nilpotent-alpha} it suffices to consider CTE actions. Indeed, for
Theorem \ref{thm:2-nilpotent-alpha} note that if $G$ is any finite group and $G'\leq G$
is a subgroup we have
$$\alpha(G')\geq\frac{\alpha(G)}{[G:G']},$$
because $[G:A]=[G:G'][G':A]$ for any subgroup $A\leq G'$ (in particular, for any abelian subgroup).

Let $D=\max_p\sum_{j\geq 0}b_j(X;\ZZ/p).$

\subsection{Commutator subgroups}

Let us denote by $\gG_0$ the collection of all finite groups $\Gamma$ such that
there exists a finite group $G$ acting smoothly and in a CTE way on $X$ and a monomorphism
$\Gamma\hookrightarrow [G,G]$.

\begin{lemma}
\label{lemma:commutator-fixed-points}
There exists a constant $C$ with the following property.
Suppose that $\Gamma\in\gG_0$ is a cyclic group of prime power order and that there exists
no $g\in\Gamma$ such that $X^g$ contains a connected component which is a nonorientable surface.
Then there exists a subgroup $\Gamma_0\leq\Gamma$ satisfying $[\Gamma:\Gamma_0]\leq C$ and $X^{\Gamma_0}\neq \emptyset$.
\end{lemma}
\begin{proof}
We are going to prove that $C=2^D$ does the job.

Let $G$ be a finite group acting smoothly in a CTE way on $X$
and let $\Gamma\leq [G,G]$ be a cyclic subgroup of prime power order.
Since $b_2(X)\neq 0$, Poincar\'e duality implies the existence of classes $\alpha,\beta\in H^2(X)$
such that $\alpha\beta$ is a generator of $H^4(X)$. Let $L_\alpha$, $L_\beta$ be complex line bundles on $X$ with first Chern classes $\alpha,\beta$ respectively. By \cite[Theorem 6.5]{M7}
there exists a short exact sequence
$$1\to S\to\widehat{\Gamma}\stackrel{\pi}{\longrightarrow}\Gamma\to 1,$$
where $S$ is a finite cyclic group, and an action of
$\widehat{\Gamma}$ on $L_{\alpha}$ lifting the action of
$\Gamma$ on $X$. Denote by
$$\mu:\widehat{\Gamma}\times L_{\alpha}\to L_{\alpha},\qquad (h,\lambda)\mapsto h\cdot\lambda$$
the map given by this action.

Let $g\in\Gamma$ be a generator, let $\gamma\in\widehat{\Gamma}$ be a lift of $g$, and let $\Gamma'\leq\widehat{\Gamma}$ be the subgroup generated by $\gamma$.
Let $S'=S\cap\Gamma'$, so that we have an exact sequence
$$1\to S'\to\Gamma'\stackrel{\pi}{\longrightarrow}\Gamma\to 1.$$
Since the action of the elements in $S'$ on $L_{\alpha}$
lift the trivial action on $X$, it is given by a morphism of groups $\xi:S'\to S^1$. Since $\Gamma'$
is cyclic, we may choose an extension of $\xi$ to $\Gamma'$, which we denote by the same symbol $\xi:\Gamma'\to S^1$.
Now the map
$$\nu:\Gamma'\times L_{\alpha}\to L_{\alpha},\qquad \nu(h,\lambda)=\xi(h)^{-1}\mu(h,\lambda)$$
defines an action of $\Gamma'$ on $L_{\alpha}$ lifting the action of $\Gamma$ on $X$, whose restriction to $S'$ is trivial. Consequently, this action descends to an action of $\Gamma$ on $L_{\alpha}$ lifting the action of $\Gamma$ on $X$. Replacing $L_{\alpha}$ by $L_{\beta}$ we similarly obtain a lift of the action of $\Gamma$ to $L_{\beta}$.

Let $E=L_{\alpha}\oplus L_{\beta}$. This is a rank $2$ complex vector bundle with $c_2(E)=\alpha\beta$,
and the lifts of the action of $\Gamma$ to $L_{\alpha}$ and $L_{\beta}$ combine to give an action on $E$.

The argument that follows can be seen as a toy model of the proof of \cite[Theorem 1.11]{M7}.
The setting is more restricted in that it applies only to dimension $4$, but more general in
that no almost complex structure on $X$ is assumed to exist.

We first prove that the action of $\Gamma$ on $X$ is not free. Arguing by contradiction,
let us assume that it is free.
Then $X/\Gamma$ is a closed, oriented and connected $4$-manifold. Let $q:X\to X/\Gamma$ be the quotient map and let $p^r=\sharp \Gamma$, where $p$ is prime and $r\geq 1$.
Then the image of the map $q^*:H^4(X/\Gamma)\to H^4(X)$ is equal to the set of integral multiples of $p^r\,\alpha\beta$.
Using the action of $\Gamma$ on $E$ we obtain a rank $2$ complex vector bundle $E_0\to X/\Gamma$ together
with an isomorphism $E\simeq q^*E_0$. By the naturality of Chern classes this implies that $\alpha\beta=c_2(E)=q^*c_2(E_0)$,
which contradicts the previous claim on the image of $q^*:H^4(X/\Gamma)\to H^4(X)$.
Hence $X^g\neq\emptyset$ for every $g\in\Gamma$ of order $p$.

Suppose that there exists some $g\in\Gamma$ such that $X^g$ contains an isolated point. Let
$S\subseteq X^g$ be the set of isolated fixed points.
By Lemma \ref{lemma:betti-fixed-point-set} we have $\sharp S\leq D$.
Since $\Gamma$ is abelian, its action on $X$ preserves $S$. Choose some $s\in S$ and let
$\Gamma_0\leq\Gamma$ be the stabilizer of $S$. Then $[\Gamma:\Gamma_0]\leq D$ and
$s\in X^{\Gamma_0}$, so $X^{\Gamma_0}\neq\emptyset$. Hence we are done in this case.

Assume for the remainder of the proof that there exists no $g\in\Gamma$ such that $X^g$ has an isolated fixed point.

Let $g\in\Gamma$ be an element of order $p$ and let $Y=X^g$. Then $Y$ is a nonempty embedded, possibly disconnected surface. Let $\Theta\leq\Gamma$ be the subgroup generated by $g$.
Recall that $H^r(B\Theta;\ZZ/p)\simeq\ZZ/p$ for every $r\geq 0$.

In the arguments that follow we will somehow abusively denote by $c_k^{\Theta}(V)$ the image of the $k$-th Chern class of an equivariant vector bundle $V$ under the map $H^{2k}_{\Theta}(\cdot)\to H^{2k}_{\Theta}(\cdot;\ZZ/p)$ induced by the projection $\ZZ\to\ZZ/p$.
Let $\pi:X\to\{*\}$ denote the projection to a point. Since $c_2(E)$ is a generator of $H^4(X)$, we have
$$\pi_*c_2^{\Theta}(E)\neq 0,$$
where $\pi_*:H^*_{\Theta}(X;\ZZ/p)\to H^{*-4}_{\Theta}(\{*\};\ZZ/p)=H^{*-4}(B\Theta;\ZZ/p)$
is the pushforward map (see e.g. \cite[\S2.1]{M7}, but take into account that here we use coefficients in $\ZZ/p$ while the discussion in [op.cit.] uses integer coefficients). Our aim is to apply localization to relate the nonvanishing of $\pi_*c_2^{\Theta}(E)$ to the existence of points with big stabilizer, and for that we need to have an invariant orientation of $Y$.

By assumption all connected components of $Y$ are orientable. Let $\nu$ be the number of connected components of $Y$. By Lemma \ref{lemma:betti-fixed-point-set} we have $\nu\leq D$.
Let $o(Y)$ be the set of orientations of $Y$.  We have $\sharp o(Y)=2^{\nu}$, and there is a natural action of $\Gamma$ on $o(Y)$. Choose some element $o\in o(Y)$  and let $\Gamma_0\leq\Gamma$ be the stabilizer of $o$. We have $[\Gamma:\Gamma_0]\leq 2^{\nu}\leq 2^D$. If $\Gamma_0=\{1\}$ then we are done, since clearly $X^{\Gamma_0}\neq\emptyset$.

Suppose from now on that $\Gamma_0\neq\{1\}$. Then $\Theta\leq\Gamma_0$.
Let $N\to Y$ be the normal bundle of the inclusion $Y\hookrightarrow X$, and orient it in a way compatible with the orientation of $X$ and with $o\in o(Y)$ . The bundle $N$ carries a natural action of $\Gamma_0$ which preserves the orientation. Hence we may consider the equivariant Euler class $e^{\Gamma_0}(N)$, which we assume, abusively as before, to lie in $H^2_{\Gamma_0}(Y;\ZZ/p)$.
By Lemma \ref{lemma:orientable-2-group}
we may endow $N$ with a $\Gamma_0$-invariant complex structure compatible with the orientation.
Then we have $e^{\Gamma_0}(N)=c_1^{\Gamma_0}(N)$, and the same formula holds replacing $\Gamma_0$ by
any of its subgroups.

Let $\rho:Y\to\{*\}$ be the projection to a point.
Fix some monomorphism $\zeta:\Theta\to S^1$ and let $t=c_1(E\Theta\times_{\zeta}S^1)\in H^2(B\Theta;\ZZ/p)$.
By the localization formula we have
\begin{equation}
\label{eq:localization}
\pi_*c_2^{\Theta}(E)=\rho_*\left(\frac{c_2^{\Theta}(E|_Y)}{c_1^{\Theta}(N)}\right).
\end{equation}
This follows from the properties of the pushforward map listed in \cite[\S2.1]{M7}, together with the fact
that $H^4_{\Theta}(X\setminus Y;\ZZ/p)=0$, so that $c_2^{\Theta}(E)$ can be lifted to
$H^4_{\Theta}(X,X\setminus Y;\ZZ/p)$ and hence belongs to the image of $i_*:H^*_{\Theta}(Y;\ZZ/p)\to H^{*+2}_{\Theta}(X;\ZZ/p)$ (here $i:Y\hookrightarrow X$ is the inclusion). The term inside
$\rho_*(\cdot)$ in the RHS of (\ref{eq:localization}) should be understood as an element
of the localized ring $H^*_{\Theta}(Y;\ZZ/p)[t^{-1}]$.
The invertibility of $c_1^{\Theta}(N)$ inside this ring is a standard fact, but the computations below give a proof of it.

The RHS of (\ref{eq:localization}) can be written as a sum of contributions from each connected component of $Y$. We next compute in concrete (nonequivariant) terms these contributions.

Fix some connected component $Z\subseteq Y$.
Suppose the action of $\Theta$ on $L_{\alpha}|_Z$ (resp. $L_{\beta}|_Z$, $N|_Z$) is given by a character
$\zeta^{a_Z}:\Theta\to S^1$ (resp. $\zeta^{b_Z}:\Theta\to S^1$, $\zeta^{n_Z}:\Theta\to S^1$).
The integers $a_Z,b_Z,n_Z$ are of course well defined only up to multiples of $p$. With respect to the
K\"unneth isomorphism
$$H^*_{\Theta}(Z;\ZZ/p)\simeq H^*(Z;\ZZ/p)\otimes H^*(B\Theta;\ZZ/p)$$
we have
$c_1^{\Theta}(L_{\alpha}|_Z)=c_1(L_{\alpha}|_Z)+a_Zt=\alpha|_Z+a_Zt$,
$c_1^{\Theta}(L_{\beta}|_Z)=c_1(L_{\beta}|_Z)+b_Zt=\beta|_Z+b_Zt$,
and $c_1^{\Theta}(N|_Z)=c_1(N|_Z)+n_Zt$.

The fact that $\Theta$ acts effectively on $N$ (which follows from (1) in Lemma \ref{lemma:lin-invariant-surface}) implies that $n_Z$ is not divisible by $p$, so we may
choose an integer $m_Z$ such that $m_Zn_Z\equiv 1\mod p$. Then we compute
in $H^*_{\Theta}(Z;\ZZ/p)[t^{-1}]$:
$$(c_1(N|_Z)+n_Zt)^{-1}=m_Zt^{-1}(1-t^{-1}m_Zc_1(N|_Z)).$$
Hence we have:
\begin{align*}
\rho_*\frac{c_2^{\Theta}(E|_Z)}{c_1^{\Theta}(N)} &= \rho_*((\alpha|_Z+a_Zt)(\beta|_Z+b_Zt)m_Zt^{-1}(1-t^{-1}m_Zc_1(N|_Z))) \\
&=\rho_*(m_Zb_Z\alpha|_Z+m_Za_Z\beta|_Z-m_Z^2a_Zb_Zc_1(N|_Z)) \\
&= m_Zb_Z\la\alpha,[Z]\ra+m_Za_Z\la\beta,[Z]\ra-m_Z^2a_Zb_Z\la c_1(N|_Z),[Z]\ra,
\end{align*}
where $[Z]\in H_2(Z)$ denotes the fundamental class. Let use denote for convenience
$$f(Z):=m_Zb_Z\la\alpha,[Z]\ra+m_Za_Z\la\beta,[Z]\ra-m_Z^2a_Zb_Z\la c_1(N|_Z),[Z]\ra.$$
We can now translate the fact that $\pi_*c_2^{\Theta}(E)$ is nonzero into the following statement:
$$\sum_Z f(Z)\qquad\text{is not divisible by $p$},$$
where $Z$ runs over the set of connected components of $Y$.

Let us decompose $Y=Y_1\sqcup\dots\sqcup Y_s$, where for each $j$ there is a connected component $Z$ of $Y$ such that $Y_j=\bigcup_{g\in\Gamma_0}gZ$. We claim that at least for one $j$ we have $Y_j\subseteq X^{\Gamma_0}$. With this claim the proof of the lemma will be complete.
The claim is an immediate consequence of the following two observations.

If $Y_j$ contains more than one connected component then $\sum_{Z\subset Y_j}f(Z)$ is divisible by $p$. Indeed, on the one hand for every connected component $Z$ of $Y$ and any $g\in \Gamma_0$ we have $f(Z)=f(gZ)$, because $\Gamma_0$ is abelian and $\Theta\leq\Gamma_0$, and on the other hand
the cardinality of $\pi_0(Y_j)$ divides $\sharp\Gamma_0$, which is a power of $p$.

If $Y_j$ is connected but $Y_j\not\subset X^{\Gamma_0}$ then $f(Y_j)$ is divisible by $p$. To see this, let us denote $Z=Y_j$ and let $\Gamma_1\leq\Gamma_0$ be the subgroup of elements acting trivially on $Z$. Then $\Xi:=\Gamma_0/\Gamma_1$ acts on $Z$ preserving the orientation and without isolated fixed points (this follows easily from the assumption that there is no $g\in\Gamma$ such that $X^g$ has an isolated fixed point). Hence $\Xi$ acts freely on $Z$. If we now prove that the action of $\Xi$ on $Z$ lifts to actions on $L_{\alpha}|_Z$, $L_{\beta}|_Z$ and $N|_Z$ then we will deduce that $f(Z)$ is divisible by $p$, by the same arguments that we used at the beginning of the proof to justify that the action of $\Gamma$ on $X$ is not free. Since $\Gamma_1$ acts trivially on $Z$, its action on $L_{\alpha}|_Z$, $L_{\beta}|_Z$ and $N|_Z$ will be given by characters
    $\xi_{\alpha},\xi_{\beta},\nu:\Gamma_1\to S^1$ respectively. Using the fact that $\Gamma_0$ is abelian we deduce that these characters can be extended to characters of $\Gamma_0$. Denote the extensions by the same symbols. Then we may twist the action of $\Gamma_0$ on $L_{\alpha}|_Z$, $L_{\beta}|_Z$ and $N|_Z$ by $\xi_{\alpha}^{-1}$, $\xi_{\beta}^{-1}$, $\nu^{-1}$ respectively. The resulting new action will be trivial on $\Gamma_1$, and hence will define a lift of the action of $\Xi$ on $Z$ to the bundles $L_{\alpha}|_Z$, $L_{\beta}|_Z$ and $N|_Z$.
\end{proof}

\begin{lemma}
\label{lemma:commutator-p-group}
There exists a constant $C$ such that for every prime $p$ and any $p$-group $\Gamma\in\gG_0$
there exists an abelian subgroup $B\leq\Gamma$ satisfying $[\Gamma:B]\leq C$. Furthermore,
at least one of the following statements is true.
\begin{enumerate}
\item for every $b\in B$ we have $X^b\neq\emptyset$;
\item there exists some $b\in B$ such that $X^b$ has a connected component
which is a nonorientable surface of genus not bigger than $C$.
\end{enumerate}
\end{lemma}
\begin{proof}
Let $p$ be any prime and let $\Gamma\in\gG_0$ be a $p$-group. Choose a MNAS $A\leq\Gamma$.
Recall that since $A\leq\Gamma$ is a MNAS, conjugation gives a monomorphism $c:\Gamma/A\hookrightarrow\Aut(A)$ (see \cite[\S 5.2.3]{Rob}).

Suppose that there exists some $a\in A$ such that $X^a$ contains an isolated fixed point. (resp. a connected component $Z$ which is a
nonorientable surface).
Then we may apply Lemma \ref{lemma:punt-fix-aillat}
(resp. Lemma \ref{lemma:superficie-no-orientable})
and conclude the existence of an abelian
subgroup $B\leq \Gamma$ satisfying $[\Gamma:B]\leq C$ (where $C$
depends only on $X$) and furthermore one of the following statements are true:
\begin{enumerate}
\item $X^B\neq\emptyset$ (this happens if we are applying
Lemma \ref{lemma:punt-fix-aillat}), or
\item there is some $b\in B$ such that
$X^b$ has $Z$ as a connected component (this happens if we are applying Lemma \ref{lemma:superficie-no-orientable}).
\end{enumerate}
So we are done in this case.

Suppose from now on that the fixed point set of every $a\in A\setminus\{1\}$
is a possibly disconnected orientable embedded surface.
Let $C_1$ be the constant given by Lemma \ref{lemma:sense-punts-fixos-aillats} and
let $W=W(X,A)$.
Since $A$ is normal in $\Gamma$, the action of $\Gamma$ on $X$ preserves $W$.
By (1) in Lemma \ref{lemma:sense-punts-fixos-aillats}, $W\subset X$ is a possibly disconnected closed embedded surface, and each connected component of $W$ is a connected component of $X^a$ for some $a\in A$.
So, by our assumption, $W$ is orientable.
By (3) in Lemma \ref{lemma:sense-punts-fixos-aillats},
$W$ contains at most $C_1$ connected components
(but beware that we have not proved that $W$ is nonempty).

Let $r$ be the number given by Theorem \ref{thm:MS} applied to $X$, so
that every finite abelian group acting effectively on $X$ can be generated by $r$ elements.

Let $C_2$ be the constant given by Lemma \ref{lemma:commutator-fixed-points}.
Let $p^k$ be the biggest power of $p$ not bigger than $C_2$.
Let $A_0\leq A$ be the image of the multiplication map $A\to A$, $a\mapsto
p^ka$ (we use additive notation on $A$).
Since $A$ can be generated by $r$ or fewer elements, we have
\begin{equation}
\label{eq:bound-A-A-0}
[A:A_0]\leq p^{kr}\leq C_2^r.
\end{equation}
Hence if $A_0=\{1\}$ then $\sharp A\leq C_2^r$, so $\sharp\Aut(A)\leq (C_2^r)!$.
Since there is a monomorphism $\Gamma/A\to\Aut(A)$, we have
$\sharp\Gamma\leq C_2^r(C_2^r)!$. Setting  $B=\{1\}$ we are done in this case.

Suppose from now on that $A_0\neq\{1\}$.
By Lemma \ref{lemma:commutator-fixed-points} we have $X^a\neq\emptyset$ for every $a\in A_0$.
Indeed, any $a\in A_0$ can be written as $a=p^kb$ for some $b\in A$, so $a$ is contained
in any subgroup $F\leq \la b\ra$ satisfying $[\la b\ra:F]\leq C_2$.
It follows that $W\neq\emptyset$.

Since $[A:A_0]$ is bounded above by a constant depending only on $X$, by Lemma \ref{lemma:quatre-u-dos} and Theorem \ref{thm:MS}
the normalizer $\Gamma_0\leq\Gamma$ of $A_0\leq A$
satisfies
$$[\Gamma:\Gamma_0]\leq C_3$$
for some constant $C_3$ that depends only on $X$.

Since the action of $\Gamma$ on $X$ preserves $W$, so does the action of $\Gamma_0$.
Let $\Gamma_1\leq\Gamma_0$ be the subgroup of elements preserving each connected component of $W$.
Then
$$[\Gamma_0:\Gamma_1]\leq C_1!.$$
Choose some orientation of $W$. The set of possible orientations of $W$ contains
$2^{\sharp\pi_0(W)}\leq 2^{C_1}$ elements, so the subgroup $\Gamma_2\leq\Gamma_1$ preserving
the orientation of $W$ satisfies
$$[\Gamma_1:\Gamma_2]\leq 2^{C_1}.$$

We claim that the elements of $\Gamma_2$ centralize $A_0$. Let $g\in\Gamma_2$ and $a\in A_0$.
Let $Z\subseteq X^a$ be a connected component. Then $Z$ is a connected
component of $W$ as well and thus $g$ preserves $Z$ and acts on $Z$ preserving the orientation, while $a$ acts trivially on $Z$. This implies that $g$ and $a$ commute, by (1) in Lemma \ref{lemma:lin-invariant-surface} and Lemma \ref{lemma:commuten-infinitesimalment}.

Let $\Aut_{A_0}^0(A)\leq\Aut_{A_0}(A)$ denote the automorphisms of $A$ which fix each element of $A_0$. From the bound (\ref{eq:bound-A-A-0}), Lemma \ref{lemma:MS-2}, and Theorem \ref{thm:MS},
we conclude that
$$\sharp \Aut_{A_0}^0(A)\leq C_4$$
for some constant $C_4$ depending only on $X$.
Using once again the fact that $A\leq\Gamma$
is a MNAS, we deduce that conjugation gives a monomorphism
$\Gamma_2/\Gamma_2\cap A\hookrightarrow\Aut(A)$. Since $\Gamma_2$ centralizes $A_0$, the image of this monomorphism lies in $\Aut_{A_0}^0(A)$. Hence
$$[\Gamma_2:\Gamma_2\cap A]=\sharp(\Gamma_2/\Gamma_2\cap A)\leq C_4.$$
Define $B:=\Gamma_2\cap A_0$. Then we have $X^b\neq\emptyset$ for every $b\in B$ and
\begin{align*}
[\Gamma:B] &= [\Gamma:\Gamma_0][\Gamma_0:\Gamma_1][\Gamma_1:\Gamma_2][\Gamma_2:\Gamma_2\cap A]
[\Gamma_2\cap A:\Gamma_2\cap A_0] \\
&\leq [\Gamma:\Gamma_0][\Gamma_0:\Gamma_1][\Gamma_1:\Gamma_2][\Gamma_2:\Gamma_2\cap A]
[A:A_0] \\
&\leq C_3C_1!2^{C_1}C_4C_2^r.
\end{align*}
The proof of the lemma is now complete.
\end{proof}

Let $C$ and $d$ be positive integers.
Recall that a collection of finite groups $\cC$ satisfies $\jJ(C,d)$
if each $G\in\cC$ has an abelian subgroup $A$ such that $[G:A]\leq C$ and
$A$ can be generated by $d$ elements.
Denote by
$\tT(\cC)$ the set of all $T \in \cC$ such that
there exist primes $p$ and $q$, a normal Sylow $p$-subgroup $P$ of $T$,
and a Sylow $q$-subgroup $Q$ of $T$, such that $T = PQ$. Note that here
$Q$ might be trivial.
The following is the main result in \cite{MT}:

\begin{theorem}
\label{thm:TM}
Let $d$ and $C_0$ be positive integers.
Let $\cC$ be a collection of finite groups which is closed under taking subgroups
and such that $\tT(\cC)$ satisfies $\jJ(C_0,d)$.
Then there exists a positive integer $C$ such that
$\cC$ satisfies $\jJ(C,d)$.
\end{theorem}

\begin{lemma}
\label{lemma:Jordan-commutator}
$\gG_0$ satisfies the property $\jJ(C,r)$ for some constant $C$.
\end{lemma}
\begin{proof}
By Theorem \ref{thm:TM} it suffices to prove the existence of a constant $C_0$
such that $\tT(\gG_0)$ satisfies $\jJ(C_0,r)$.

Let $\Gamma\in\tT(\gG_0)$ and write $\Gamma=PQ$, where $P\leq \Gamma$ (resp. $Q\leq \Gamma$) is a
Sylow $p$-subgroup (resp. $q$-subgroup), $p,q$ are different primes, and $P$ is a normal
subgroup of $\Gamma$. By Lemma \ref{lemma:commutator-p-group} there is an abelian subgroup
$P_0\leq P$ satisfying $[P:P_0]\leq C_1$, where $C_1$ depends only on $X$, and, furthermore, at least one of these statements is
true:
\begin{enumerate}
\item for any $g\in P_0$ we have $X^g\neq\emptyset$,
\item there is some $g\in P_0$
such that $X^g$ has a connected component which is a nonorientable surface.
\end{enumerate}
Using again Lemma \ref{lemma:commutator-p-group} we may pick an abelian subgroup
$Q'\leq Q$ satisfying $[Q:Q']\leq C_1$.
Let $Q_0\leq Q'$ be the normalizer of $P_0$ in $Q'$.
Since $Q_0=Q'\cap N_{\Gamma}(P_0)$,
by Theorem \ref{thm:MS} and Lemma \ref{lemma:quatre-u-dos}, there exists a constant $C_2$
depending only on $X$ such that $$[Q':Q_0] \leq [\Gamma:N_\Gamma(P_0)]\leq C_2.$$

By Lemmas \ref{lemma:punt-fix-aillat} and \ref{lemma:superficie-no-orientable},
if there exists some $g\in P_0$ such that $X^g$ has a connected component which is an isolated
point or a nonorientable surface then there exists an abelian subgroup $B\leq P_0Q_0$ satisfying $[P_0Q_0:B]\leq C_3$,
where $C_3$ only depends on $X$. Since $[PQ:P_0Q_0]\leq C_1^2C_2$, it follows that
$$[\Gamma:B]=[PQ:B]\leq C_1^2C_2C_3$$
and we are done in this case.

Let us assume for the remainder of the proof that
the fixed point set of every $g\in P_0\setminus\{1\}$
is a possibly disconnected orientable embedded surface.
Define $W=W(X,P_0)$.
By Lemma \ref{lemma:sense-punts-fixos-aillats},
$W$ is a possibly disconnected embedded closed surface (orientable,
by our previous assumption), for each
$g\in P_0\setminus\{1\}$ the fixed point set
$X^g$ is equal to the union of some connected components of $W$,
and $W$ has at most $C_4$ connected components, where $C_4$ only depends on $X$.
Furthermore, the genus of each connected component of $W$ is at most $C_4$.
Since $Q_0$ normalizes $P_0$, the action of $Q_0$ on $X$ preserves $W$.

Our hypothesis implies that statement (1) above holds true.
Let $Q_1\leq Q_0$ be the subgroup of those elements preserving each connected component of
$W$, and acting orientation preservingly on each connected component of $W$.
We have $[Q_0:Q_1]\leq 2^{C_4}C_4!$. We claim that if $p\in P_0$ and $q\in Q_1$ then $p$ and $q$
commute. To see this, take a connected component $Z$ of $X^p$. By (2) in Lemma \ref{lemma:sense-punts-fixos-aillats}, $Z$ is a connected component of $W$, so $Q_1$ preserves $Z$ and acts on $Z$ orientation preservingly.
Then the commutativity of $p$ and $q$ follows from (1) in Lemma \ref{lemma:lin-invariant-surface}
and from Lemma \ref{lemma:commuten-infinitesimalment}. Hence $P_0Q_1$ is abelian, and combining
our previous bounds we obtain
$$[\Gamma:P_0Q_1]\leq C_1^2C_22^{C_4}C_4!,$$
so the proof of the lemma is now complete.
\end{proof}

\subsection{Proof of Theorem \ref{thm:2-step-nilpotent}}
\label{ss:proof-thm:2-nilpotent}
Let $X$ be an oriented and connected closed $4$-manifold.
Let $r$ be the number resulting from applying Theorem \ref{thm:MS}
to $X$, so that every finite abelian group $A$ acting effectively
on $X$ can be generated by $r$ elements.

Let $G$ be a finite group acting in a smooth and CTE way on $X$.
Let $\Gamma=[G,G]$.
By Lemma \ref{lemma:Jordan-commutator} there is an abelian
subgroup $A\leq\Gamma$ satisfying
$$[\Gamma:A]\leq C_1,$$
where $C_1$ depends only on $X$.
We distinguish two cases, according to whether the action of $A$ on $X$ is free or not.

Suppose that $A$ acts freely on $X$. Let $r,C_F$ be the constants given by
Theorem \ref{thm:MS} and Lemma \ref{lemma:commutator-fixed-points} applied to $X$.
If $p$ is a prime bigger than $C_F$ and the $p$-part $A_p\leq A$ is nontrivial,
then by Lemma \ref{lemma:commutator-fixed-points} the action of $A_p$ on $X$ has
nontrivial fixed points, which contradicts the assumption that $A$ acts freely. Hence
we may write
$$A\simeq A_{p_1}\times\dots\times A_{p_s},$$
where $p_1,\dots,p_s$ are the prime numbers in $\{1,\dots,C_F\}$.
By Lemma \ref{lemma:commutator-fixed-points} the exponent of $A_{p_i}$ cannot
be bigger than $C_F$, for otherwise the action of $A_{p_i}$ would not be free.
This implies that $\sharp A_{p_i}\leq C_F^r$, and consequently
$$\sharp A\leq C_F^{rs},$$
so $\sharp\Gamma\leq C_1C_F^{rs}$. Applying Lemma \ref{lemma:proceedings} to the
exact sequence $1\to\Gamma\to G\to G/\Gamma\to 1$ we conclude the
existence of an abelian subgroup $B\leq G$ such that $[G:B]$ is bounded
above by a constant depending only on $X$. In this case we set $G_0:=B$ and we are done.

Assume, for the remainder of the proof, that $A$ does not act freely on $X$.
Let
$$G'=N_{G}(A)\leq G$$
be the normalizer of $A$ in $G$.
By Lemma \ref{lemma:quatre-u-dos}
we have
$$[G:G']\leq C_2,$$
where $C_2$ depends only on $X$.
Let $p$ be a prime such that $A_p\neq 1$ and the action of $A_p$ on $X$ is not free.
Since $A_p$ is a characteristic subgroup of $A$, $A_p$ is normal in $G'$.

If there is some $a\in A_p$ such that $X^a$ has an isolated fixed point then
by Lemma \ref{lemma:punt-fix-aillat} there is an abelian subgroup $B\leq G'$
such that $[G':B]$ is bounded above by a constant depending only on $X$, so
setting $G_0:=B$ we are done in this case. 

Now assume that there is no $a\in A_p$ such that $X^a$ has an isolated fixed point.
If there is some $b\in A_p$ such that $X^b$ has a connected component which is a nonorientable surface then by Lemma \ref{lemma:superficie-no-orientable} there is an abelian subgroup $B\leq G'$ such that $[G':B]$ is bounded above by a constant depending only on $X$, and hence setting $G_0:=B$ we are also done in this case.

At this point we may assume that for every $a\in A_p\setminus\{1\}$ the fixed point set
$X^a\subset X$ is a possibly empty embedded orientable surface and that the set
$W=W(X,A_p)$ defined in Subsection \ref{ss:def-W(X,A)} is nonempty. Let $C_3$ be the constant given by applying Lemma \ref{lemma:sense-punts-fixos-aillats} to $X$ (so $C_3$ only depends on $X$).
Then $W$ has at most $C_3$ connected components and the absolute value of the genus of each of its connected components is not bigger than $C_3$. Furthermore, since $A_p$ is a normal subgroup of $G'$ the action of $G'$ on $X$ preserves $W$.

We distinguish two cases according to whether $\chi(Z)$ vanishes for all connected components
$Z\subseteq W$ or not.

Suppose first that there is a connected component $Z\subseteq W$ such that $\chi(Z)\neq 0$.
Let $G''\leq G'$ be the subgroup of elements preserving $Z$. We have $[G':G'']\leq C_3$.
From Lemmas \ref{lemma:lin-invariant-surface} and \ref{lemma:line-bundle-Jordan}
we deduce the existence of an abelian subgroup $B\leq G''$ such that $[G'':B]$
is bounded above by a constant depending only on $X$. Hence, setting $G_0:=B$ we are done.

Finally, suppose that $\chi(Z)=0$ for all connected components
$Z\subseteq W$. Choose any connected component $Z\subset X$ and
let $G''\leq G'$ be the subgroup of elements preserving $Z$. We have $[G':G'']\leq C_3$
and by Lemmas \ref{lemma:lin-invariant-surface} and \ref{lemma:line-bundle-on-torus}
there exists a nilpotent subgroup $G_0\leq G''$ of class at most $2$ satisfying $[G'':G_0]\leq 12$
and, furthermore, $[G_0,G_0]$ is cyclic and acts trivially on $Z$. We thus have
$$Z\subseteq X^{[G_0,G_0]}\subseteq W,$$
so $X^{[G_0,G_0]}$ is a nonempty union of embedded tori because all connected components
of $W$ are orientable and have zero Euler characteristic.
Combining the previous estimates we have
$$[G:G_0]\leq [G:G'][G':G''][G'':G_0]\leq 12\,C_2C_3,$$
so the proof of the theorem is now complete.

\subsection{Proof of Theorem \ref{thm:2-nilpotent-alpha}}

Suppose that $N$ is a finite nilpotent group of class at most $2$ acting in a
smooth and CTE way on $X$.
Then $[N,N]$ is abelian and central in $N$.
The arguments in Subsection \ref{ss:proof-thm:2-nilpotent}
imply the existence of a constant $C_1$, depending only on $X$, such that if
$\alpha(N)\geq C_1$ then the group $[N,N]$ does not act freely on $X$, and
any nontrivial $g\in [N,N]$ whose action on $X$ has fixed points satisfies
(2) in the statement of Theorem \ref{thm:2-nilpotent-alpha}.
Furthermore, (3) holds
for any such $g$ (with a suitable choice of $C$ depending only on $X$) thanks to (1) in Lemma \ref{lemma:line-bundle-on-torus}.

To conclude the proof of Theorem \ref{thm:2-nilpotent-alpha} assume that
$\alpha(N)\geq C_1$ and let us prove that
there exists a nontrivial $g\in [N,N]$ which does not act freely on $X$ and
whose order satisfies $\ord(g)\geq f(\alpha(N))$
for some function $f$ depending on $X$ and satisfying $\lim_{n\to\infty} f(n)=\infty$.

We may write
$$[N,N]\simeq \Gamma_1\times\dots\times\Gamma_s,$$
where each $\Gamma_i$ is cyclic of prime power order. By Theorem \ref{thm:MS},
$s\leq C_2$, where $C_2$ depends only on $X$.
For any $g\in[N,N]$ which does not act freely on $X$ the fixed
point set $X^g$ is the disjoint union of some tori (because we are assuming $\alpha(N)\geq C_1$),
so in particular $X^g$ has no connected component which is a nonorientable surface.
Consequently, Lemma \ref{lemma:commutator-fixed-points}
implies that for every $i$ there exists some $\Gamma_i'\leq\Gamma_i$
such that $X^{\Gamma_i'}\neq\emptyset$ and $[\Gamma_i:\Gamma_i']\leq C_3$
for some constant $C_3$ depending only on $X$.
Then we have
$$\max_i\sharp\Gamma_i'\geq\frac{\max_i\sharp\Gamma_i}{C_3}\geq \frac{[N,N]^{1/C_2}}{C_3}.$$

By Lemma \ref{lemma:proceedings} there exists a function $h:\NN\to\NN$ depending only on $X$
and satisfying $\lim_{n\to\infty}h(n)=\infty$ and $\sharp[N,N]\geq h(\alpha(N))$
(just take $G=N$ and $G_0=[N,N]$, so that $G_1=N/[N,N]$ is abelian).
The function $f:\NN\to\NN$ defined as
$$f(n):=\frac{h(n)^{1/C_2}}{C_3}$$
depends only on $X$, it satisfies $\lim_{n\to\infty}f(n)=\infty$, and by the previous estimate
we have
$\max_i\sharp\Gamma_i'\geq f(\alpha(N))$, so picking some $i$ realizing the previous maximum,
any generator $g$ of $\Gamma_i'$ satisfies $\ord(g)\geq f(\alpha(N))$.

\section{Using the Atiyah--Singer $G$-signature theorem}
\label{s:Atiyah-Singer}

\begin{theorem}
\label{thm:Atiyah-Singer-G-0}
Let $X$ be a closed connected and oriented $4$-manifold satisfying $\sigma(X)\neq 0$.
If $\phi\in\Diff(X)$ has finite order and acts trivially on cohomology then $X^\phi\neq\emptyset$.
\end{theorem}
\begin{proof}
This is an immediate consequence of the $G$-signature theorem \cite[Theorem 6.12]{AS}
and the fact that $\sigma(\phi,X)=\sigma(X)\neq 0$ if $\phi$ acts trivially on cohomology.
\end{proof}

\begin{theorem}
\label{thm:Atiyah-Singer-G}
Let $X$ be a closed connected and oriented $4$-manifold. Suppose that $\phi\in\Diff(X)$ has finite order bigger than $2$ and acts trivially on cohomology, and that
the fixed point set $X^{\phi}$ has no isolated fixed points
(so all the connected components of $X^{\phi}$ are embedded surfaces). Suppose that $X^{\phi}=S_1\sqcup\dots\sqcup S_n$ with each $S_i$ connected, and that
the action of $\phi$ on the normal bundle of $S_k$ is by rotation of angle $\theta_k\in S^1$. Then
all connected components of $X^{\phi}$ are orientable and
$$\sigma(X)=\sum_{k=1}^n\sin^{-2}(\theta_k/2)\,S_k\cdot S_k.$$
\end{theorem}
\begin{proof}
The orientability of the connected components of $X^{\phi}$ is guaranteed
by (1) in Lemma \ref{lemma:orientable-normal-bundle}.
If the order of $\phi$ is odd then the formula for $\sigma(X)$
follows from \cite[Proposition 6.18]{AS}. For the general case
note that the proof of \cite[Proposition 6.18]{AS} works equally
well if the order of $\phi$ is even and bigger than $2$.
Indeed, in this case the normal bundle $N$ of every connected component $Y\subseteq X^{\phi}$
supports an invariant almost complex structure (by Lemma \ref{lemma:orientable-2-group},
because $Y$ is orientable and hence so is $N$) and $\phi$ acts on $N$ through multiplication
by a complex number different from $\pm 1$ (so in the notation of \cite[\S 6]{AS} we
have $N^\phi(-1)=0$).
\end{proof}

\begin{theorem}
\label{thm:signature-nonzero-jordan}
Let $X$ be a closed, connected and oriented $4$-manifold.
If $\sigma(X)\neq 0$ then $\Diff(X)$ is Jordan.
\end{theorem}
\begin{proof}
The same argument that we used in Lemma \ref{lemma:Jordan-commutator} to prove
that the family of finite groups $\gG_0$ is Jordan works in our case if we replace
Lemma \ref{lemma:commutator-fixed-points} by Theorem
\ref{thm:Atiyah-Singer-G-0}.
\end{proof}

The following lemma is used in the proof of Theorem \ref{thm:criterion-degree}.

\begin{lemma}
\label{lemma:positius-i-negatius}
Let $X$ be a closed connected and oriented $4$-manifold satisfying $\sigma(X)=0$.
There exists a real number $\lambda>0$ with the following property.
Suppose that $\phi\in\Diff(X)$ has finite order and acts trivially on cohomology, that
the fixed point set $X^{\phi}$ has no isolated fixed points, and that all connected
components of $X^{\phi}$ (which, by assumption, are embedded surfaces) are orientable. Write
$X^{\phi}=S_1\sqcup\dots\sqcup S_n$ and define
$$\mu_M=\max_i S_i\cdot S_i,\qquad \mu_m=\min_i S_i\cdot S_i.$$
Then $\mu_M\geq -\lambda\mu_m\geq 0$ and $\mu_m\leq-\lambda\mu_M\leq 0$.
\end{lemma}
\begin{proof}
We first prove that the number $n$ of connected components of $X^{\phi}$ is bounded
above by a constant depending only on $X$: more precisely, if we define $D=\max_p\sum_{j\geq 0}b_j(X;\ZZ/p)$ then $n\leq D/2$. Indeed, if $\phi\in\Diff(X)$ satisfies the hypothesis of
the lemma and its order is equal to $ps$, where $p$ is a prime and $s$ an integer, then
applying Lemma \ref{lemma:betti-fixed-point-set} to the fixed point set of $\phi^s$ and noting
that each connected component of $X^{\phi}$ is a connected component of $X^{\phi^s}$ we conclude
that
$$\sum_i\sum_k b_k(S_i;\ZZ/p)\leq D.$$
Since each $S_i$ contributes at least two units to the left hand side, the bound $n\leq D/2$
follows.

Once we have an upper bound on the number of connected components of $X^{\phi}$,
the proof is concluded combining Theorem \ref{thm:Atiyah-Singer-G}
and the following lemma.
\end{proof}

\begin{lemma}
Given an integer $n>0$ there exists a real number
$\delta>0$ and an integer $k_0>0$ such that for every
integer $k\geq k_0$ and any choice of primitive $k$-th roots of
unity
$$\theta_1,\dots,\theta_n\in S^1$$
there is an integer $a$ such that
$|\sin \theta_j^a|\geq\delta$ for every $j$.
\end{lemma}
\begin{proof}
We consider the standard measure on $S^1$ of total volume $2\pi$.
For every integer $k\neq 0$ we denote by $\mu_k$ the set of all $k$-th roots
of unity, and for every $\epsilon>0$ we denote $A_{\epsilon}=
\{e^{2\pi\imag\theta}\mid |\theta|<\epsilon\}\subseteq S^1$ and
$S_{\epsilon}=A_{\epsilon}\cup(-A_{\epsilon})$.

Define $\epsilon=1/(4(n+1))$ and $k_0=4(n+1)$. Suppose that
$k\geq k_0$.
For every $\theta,\theta'\in \mu_k$ the sets
$\theta S_{1/(2k)}$ and $\theta' S_{1/(2k)}$ are disjoint.
Since $1/2k\leq 1/2k_0=\epsilon/2$, we have
$$\bigcup_{\theta\in \mu_k\cap A_{\epsilon/2}}\theta S_{1/(2k)}\subseteq A_{\epsilon}.$$
Combining this inclusion with
$\Vol(A_{\epsilon})=8\pi\epsilon=2\pi/(n+1)$ and $\Vol(\theta S_{1/(2k)})=2\pi/k$,
it follows that
$$\sharp \mu_k\cap A_{\epsilon/2}\leq \frac{2\pi/(n+1)}{2\pi/k}=\frac{k}{n+1}.$$

Let $[k]=\{1,2,\dots,k\}$.
Suppose that $\theta_1,\dots,\theta_n$ are $k$-th primitive roots of unity.
Then for every $j$ the map $e_j:[k]\to\mu_k$ defined as $e_j(a)=\theta_j^a$
is a bijection. Define $C_j=\{a\in[k]\mid \theta_j^a\in A_{\epsilon/2}\}$.
The previous estimate implies that $\sharp C_j\leq k/(n+1)$, and hence
the set $C=C_1\cup C_2\cup\dots\cup C_n$ satisfies
$\sharp C<k$. Therefore $[k]\setminus C$ is nonempty.
Take any $a\in [k]\setminus C$. For every $j$ we have
$\theta_j^a\notin A_{\epsilon/2}$, so
$$|\sin\theta_j^a|\geq\delta:=\sin 2\pi\epsilon/2=\sin \pi/4(n+1),$$
so the proof of the lemma is complete.
\end{proof}

\section{Automorphisms of almost complex manifolds: proof of Theorem \ref{thm:almost-complex}}
\label{s:almost-complex}

We will use the positivity of intersection of holomorphic curves in $4$-dimensional almost complex
manifolds. This was first stated by Gromov in \cite[2.1.$C_2$]{Gromov} and a detailed proof was given by McDuff in \cite[Theorem 2.1.1]{McDuff} (see page 36 in \cite{McDS3} for some
comments on earlier proofs).
We next give a detailed statement of the result adapted to our needs, and for the
reader's convenience we reduce its proof (using basically the same ideas as
\cite{McDuff}) to results proved in full detail in the book \cite{McDS3}.

\begin{lemma}
\label{lemma:positive-intersection}
Let $(X,J)$ be an almost complex $4$-dimensional manifold,
let $\Sigma_1,\Sigma_2$ be closed and connected Riemann surfaces,
and let $\phi_i:\Sigma_i\to X$, $i=1,2$, be $J$-holomorphic maps.
Assume that $\phi_1$ is an immersion. Let $[\Sigma_i]\in H_2(\Sigma_i)$
denote the fundamental class corresponding to the canonical orientation
as a Riemann surface.
Then
$$(\phi_1)_*[\Sigma_1]\cdot (\phi_2)_*[\Sigma_2]\geq 0$$
unless $\phi_1(\Sigma_1)=\phi_2(\Sigma_2)$.
\end{lemma}
\begin{proof}
Let $I=\{(x,y)\in\Sigma_1\times\Sigma_2\mid \phi_1(x)=\phi_2(y)\}$.
Let $I^*\subseteq I$ be the set of isolated points. Clearly, $I\setminus I^*$ is closed
in $\Sigma_1\times\Sigma_2$, and hence is compact. Let $\pi_i:\Sigma_1\times\Sigma_2\to\Sigma_i$
denote the projection. Then $\pi_i(I\setminus I^*)\subseteq\Sigma_i$ is closed for $i=1,2$.

We next prove that $\pi_i(I\setminus I^*)$ is open for $i=1,2$.
Assume that $(x,y)\in I\setminus I^*$ and that $(x_i,y_i)$ is a
sequence in $I$ converging to $(x,y)$ and satisfying $(x_i,y_i)\neq (x,y)$ for
every $i$. Choose charts $f:\Omega\to\Sigma_1$ and $g:\Omega\to\Sigma_2$ with
$f(0)=x_1$, $g(0)=x_2$, where $\Omega\subset\CC$ is an open subset containing the origin,
and define $u=\phi_1\circ f$, $v=\phi_2\circ g$. Ignoring if necessary some of the initial points in the sequence $(x_i,y_i)$, we may assume that $x_i=f(z_i)$ and $y_i=g(\zeta_i)$ for
$z_i,\zeta_i\in\Omega$. We claim that $\zeta_i\neq 0$ for infinitely many indices $i$.
Otherwise for every $i$ we would have $y_i=y$ and hence $\phi_1(x_i)=\phi_2(y_i)=\phi_2(y)=\phi_1(x)$, which
would imply that $\phi_1^{-1}(\phi_1(x))\supset\{x_1,x_2,\dots\}$ is infinite;
but this, by \cite[Lemma 2.4.1]{McDS3}, is incompatible with the assumption that
$\Sigma_1$ is compact and $d\phi_1(x)\neq 0$. Hence the claim is proved.
We are thus in a position to apply
\cite[Lemma 2.4.3]{McDS3} and conclude that $\pi_1(I\setminus I^*)$ (resp. $\pi_2(I\setminus I^*)$) contains $x$ (resp. $y$) in its interior.

To finish the proof we distinguish two possibilities.
If $I\setminus I^*\neq\emptyset$ then both projections
$\pi_1(I\setminus I^*)$ and $\pi_2(I\setminus I^*)$
are nonempty. Since these projections are open and closed and
both $\Sigma_1$ and $\Sigma_2$ are connected,
we have
$\pi_i(I\setminus I^*)=\Sigma_i$ for $i=1,2$.
The equality $\pi_1(I\setminus I^*)=\Sigma_1$ means that for each $x\in \Sigma_1$ there is some
$y\in\Sigma_2$ such that $\phi_1(x)=\phi_2(y)$, so
$\phi_1(\Sigma_1)\subseteq\phi_2(\Sigma_2)$.
Similarly (exchanging $\Sigma_1$ and $\Sigma_2$)
we have $\phi_2(\Sigma_2)\subseteq\phi_1(\Sigma_1)$.
Consequently in this case we have $\phi_1(\Sigma_1)=\phi_2(\Sigma_2)$.

The other possibility is that $I\setminus I^*=\emptyset$,
so $I=I^*$ and hence $I$ is finite. Then
\cite[Theorem 2.6.3]{McDS3} implies that
$(\phi_1)_*[\Sigma_1]\cdot (\phi_2)_*[\Sigma_2]\geq 0$.
\end{proof}

Let us now prove Theorem \ref{thm:almost-complex}.
Let $(X,J)$ be a closed almost complex $4$-manifold, and let $\gG=\Aut(X,J)$ be its
group of automorphisms. Assume that $\gG$ is not Jordan. Then, by Theorem \ref{thm:criterion-degree}
we can find some $\phi\in\gG$ of finite order such that $X^{\phi}$ has a connected component
$T$ which is an embedded torus of negative self-intersection. 
Since $\phi$ preserves $J$ and has finite order, its fixed point locus is a
(possibly disconnected) almost complex submanifold.
In particular $T$ is an almost complex submanifold of $(X,J)$ and hence can
be identified with the image of a holomorphic embedding $\psi:\Sigma\to (X,J)$
where $\Sigma$ is a closed connected Riemann surface of genus $1$.

Let $\gG_0\leq\gG$ denote the subgroup of automorphisms acting trivially on $H^*(X)$.
We claim that the elements of $\gG_0$ preserve $T$. Indeed, if
$\zeta\in\gG_0$ then applying Lemma \ref{lemma:positive-intersection} to
$\psi$ and $\zeta\circ\psi$ we conclude that $\zeta(T)=T$ because
$T\cdot T<0$.

Let $G\leq\gG$ be a finite subgroup.
By Lemma \ref{lemma:minkowski} the intersection $G_0=G\cap\gG_0$
satisfies $[G:G_0]\leq C$ for some constant $C$ depending only on $X$.
By our previous observation, every element of $G_0$ preserves $T$.
So, if we denote by $N\to T$ the normal bundle of the inclusion in $X$
then by Lemma \ref{lemma:lin-invariant-surface} the action of $G_0$ on $X$
induces a monomorphism $G_0\hookrightarrow\Aut(N)$. By Lemma 6.5 there is an abelian subgroup
$A\leq G_0$ satisfying $$[G_0:A]\leq 12\, |T\cdot T|,$$ so we have
$$[G:A]\leq 12\, C |T\cdot T|.$$
We have thus proved that $\gG$ is Jordan,
contradicting our initial assumption that it was not.

\section{Symplectomorphisms: proof of Theorem \ref{thm:symp}}
\label{s:symplectomorphisms}

By (1), (2) and (4) in Theorem \ref{thm:non-Jordan}, in order to prove Theorem \ref{thm:symp} it suffices to consider
the case where
\begin{equation}
\label{eq:condicions-X}
\chi(X)=\sigma(X)=0,\qquad b_2^+(X)=1.
\end{equation}
For the latter condition note that
$b_2(X)>0$ on any symplectic manifold, and the vanishing of the signature implies that
$b_2(X)=2b_2^+(X)$. Under these conditions we have $b_2(X)=2$ and consequently
(using the vanishing of $\chi$ and Poincar\'e duality) $b_1(X) = b_3(X) = 2$.
In particular, statement (3) in Theorem \ref{thm:symp} follows from Theorem
\ref{thm:non-Jordan}.

So throughout this section $(X, \omega)$ will denote a fixed closed symplectic $4$-manifold satisfying the previous conditions (\ref{eq:condicions-X}).

Let $J$ be any $\omega$-compatible almost complex structure on $X$. We can define the canonical bundle $K_X$ of $X$ as the complex line bundle $K_X = \bigwedge^2 T^*X$, where the complex structure is induced by $J$.
We denote by
$$K \in H_2(X)$$
the Poincar\'e dual of $c_1(K_X)$.  Since the space of $\omega$-compatible almost complex structures on $X$ is contractible, $K$ is independent of the chosen $J$.

We say that a class $A \in H_2(X)$ is representable by $J$-holomorphic curves if there is a
possibly disconnected closed Riemann surface $\Sigma$ and a
$J$-holomorphic map $\psi:\Sigma\to X$ such that $\psi_*[\Sigma]=A$.

\begin{lemma}
\label{lemma:taubes}
Suppose that $X$ is not an $S^2$-bundle over $T^2$. Then, for every $\omega$-compatible almost complex structure $J$ on $X$, $K$ or $2K$ are representable by $J$-holomorphic curves.
\end{lemma}

\begin{proof}
Before we prove the lemma, let us recall some facts about Seiberg--Witten invariants of symplectic manifolds with $b_2^+(X)=1$.

For any closed connected $4$-manifold $X$
the set of Spin$^c$ structures on $X$ has a natural
structure of torsor over $H^2(X)$ (see e.g. \cite[\S 3.1]{Morgan}). If $\slie$ is a Spin$^c$ structure
and $\beta\in H^2(X)$ then we denote by $\beta\cdot\slie$ the Spin$^c$ structure given by the
action of $\beta$ on $\slie$.
If $(X, \omega)$ is a symplectic manifold (which we assume in all the following discussion) then there is a canonical Spin$^c$ structure on $X$, denoted by $\slie_{can}$, with determinant line bundle $K_X^{-1}$
(actually to define this structure we need to choose an almost complex structure compatible with $\omega$, but the outcome only depends on $\omega$, see e.g. \cite[\S 3.4]{Morgan}).
This Spin$^c$ structure allows us to identify $H^2(X)$ with the set of Spin$^c$ structures on $X$, by assigning to $\beta\in H^2(X)$ the Spin$^c$ structure $\beta\cdot \slie_{can}$.
In terms of this identification we can regard the Seiberg--Witten invariant as a map
$$SW:H^2(X) \to \ZZ.$$

For closed $4$-manifolds $X$ with $b_2^+(X) > 1$ the moduli spaces of Seiberg--Witten solutions for two generic pairs of metric and perturbation $(g_1, \eta_1), (g_2, \eta_2)$ can be connected by a smooth cobordism\footnote{Here and below generic means as usual that the Seiberg--Witten equations define a section of a Banach vector bundle over the parameter space
(connections) $\times$ (sections of the spinor bundle) which is transverse to the zero section,
so in particular the moduli space is a smooth manifold of the expected dimension.}. This implies that the invariant $SW$ is independent of the generic metric and perturbation chosen to define it.

When $b_2^+(X) = 1$ this is not true anymore, as there might exist generic pairs $(g_1, \eta_1)$, $(g_2, \eta_2)$ whose moduli spaces cannot be connected by any smooth cobordism.
More precisely, for any $\beta\in H^2(X)$ 
the space $\mathcal{S}_\beta$ of all pairs (metric, perturbation) 
whose moduli space of Seiberg--Witten solutions contain no reducible solution (that is, solutions $(A, \psi)$ with $\psi=0$) has two connected components. The Seiberg--Witten moduli spaces associated to two generic elements of $\mathcal{S}_\beta$ can be connected by a smooth cobordism if the two elements belong to the same connected component of $\mathcal{S}_\beta$, but
there is no reason to expect the existence of such a cobordism if they belong to different connected components. Hence, we should consider two possibly different Seiberg--Witten invariants, one for each connected component of $\mathcal{S}_\beta$.

One can prove that it is possible to label the connected components of $\mathcal{S}_\beta$ as
$\mathcal{S}_\beta^+$ and $\mathcal{S}_\beta^-$ in such a way that the following holds.
For any  metric $g$ on $X$ let us denote by $\omega_g$ the unique self-dual $g$-harmonic $2$-form of $L^2$-norm $1$ whose cohomology class belongs to the same connected component of
$H^2_+(X; \RR)\setminus\{0\}$ as $[\omega]$. Then
$(g,\pm\imag\lambda\omega_g)\in\mathcal{S}_\beta^{\pm}$
for $\lambda>0$ sufficiently big.
Hence we may encode the Seiberg--Witten invariants of $X$ through two maps
$$SW^\pm: H^2(X) \to \ZZ,$$
where $SW^\pm(\beta)$ is the invariant obtained from a generic pair
belonging to $\mathcal{S}_\beta^{\pm}$.
For further details, see Section 7.4 of \cite{Sal}.

Define
$$w(\beta) = SW^+(\beta) - SW^-(\beta).$$
This difference $w(\beta)$ can be computed by means of a wall-crossing formula. We will just decribe the relevant formula needed for our purposes. For the general form of the wall-crossing formulas we refer the reader to \cite[Theorem 9.4]{Sal}.
By \cite[Proposition 12.5]{Sal} (see \cite[Remark 13.7]{Sal}) we have
$$SW^-(\beta) = SW^+(c_1(K_X) - \beta)$$
for every $\beta$. Therefore,
\begin{equation}
\label{eq:wall-crossing}
w(\beta) = SW^+(\beta) - SW^+(c_1(K_X) - \beta).
\end{equation}
A theorem of Taubes implies that $SW^+(0) = 1$ (see \cite{T1} and \cite[Theorem 13.8]{Sal}).

For a manifold with $b_1(X)=2$, we can compute $w(\beta)$ as follows (see \cite[Definition 2.2]{McDS}).
Let $\alpha_1, \alpha_2$ be a basis of $\in H^1(X)$, and define $a = \alpha_1 \cup \alpha_2$.
Let $\beta \in H^2(X)$. Let 
$$d(\beta) = - \la \beta,  K \ra  + \int_X \beta^2.$$
Then
\begin{equation}
\label{eq:cas-b-1-igual-a-2}
w(\beta) = \int_X  a \cup (\beta - c_1(K_X)/2)\qquad\qquad \text{if $d(\beta) \geq 0$,}
\end{equation}
and $w(\beta)=0$ if $d(\beta) < 0$.

We are now ready to prove the lemma.

We claim that if $SW^+(\beta) \neq 0$, then for any $\omega$-compatible almost complex structure $J$, $PD(\beta)$ is representable by $J$-holomorphic curves. Indeed, let $g_J = \omega( \cdot, J \cdot)$ be the metric associated with $\omega$ and $J$. With respect to this metric, $\omega$ is self-dual and of positive norm. The fact that $SW^+(\beta) \neq 0$ means that for any perturbation $\eta$ satisfying $(g_J, \eta)\in{\mathcal S}_{\beta}^+$ there exists some solution to the Seiberg--Witten equations with metric $g_J$ and perturbation $\eta$ (this follows by definition for generic $\eta$ and by a compactness argument for general perturbations).
Then, the existence of the $J$-holomorphic curve representing $PD(\beta)$ follows from \cite[Theorem 1.3]{T2}.

Therefore, we only need to show that $SW^+(c_1(K_X))$ and $SW^+(2c_1(K_X))$ cannot be
both zero.

If $w(0) \neq 1$, from (\ref{eq:wall-crossing}) and $SW^+(0)=1$
we obtain $SW^+(c_1(K_X)) \neq 0$, so $K$ is representable by $J$-holomorphic curves, and we are done in this case.

Suppose for the remainder of the proof that $w(0) = 1$.
We have $d(0)=0$.
By the Hirzebruch signature theorem we have
$K \cdot K = 2 \chi(X) + 3 \sigma(X) = 0$, and this implies
$d(2c_1(K_X))=0$. We then compute, using (\ref{eq:cas-b-1-igual-a-2}),
\begin{align*}
w(2c_1(K_X)) &= \int_X a \cup 3c_1(K_X)/2  \qquad\text{because $d(2c_1(K_X))=0$}\\
&= 3\int_X a \cup c_1(K_X)/2  \\
&= -3w(0) \qquad\qquad\qquad\,\,\text{because $d(0)=0$} \\
&= -3.
\end{align*}
Hence,
$$SW^+(2c_1(K_X)) - SW^+(-c_1(K_X)) = -3.$$
We claim that $-K$ is not representable by $J$-holomorphic curves and therefore $$SW^+(-c_1(K_X))=0.$$ Indeed, if $-K$ were representable by $J$-holomorphic curves, then by the positive energy condition of $J$-holomorphic curves we would have $\langle [\omega], -K \rangle > 0$. However, Theorem B in \cite{Liu} implies that in this case $(X, \omega)$ is a ruled or rational surface or a blow up
of a ruled or rational surface. Since $b_1(X)=b_2(X)=2$, $(X, \omega)$ must be a ruled surface over $T^2$. Hence, $X$ is an $S^2$-bundle over $T^2$, contradicting the assumption of the lemma.

Therefore, $SW^+(2c_1(K_X))=-3$ and consequently $2K$ is representable by $J$-holomorphic curves, thus finishing the proof of the lemma.
\end{proof}

\begin{lemma}
\label{lemma:symplectic-torus}
Suppose that $X$ is not an $S^2$-bundle over $T^2$. Let $\phi \in \Symp(X, \omega)$ be an element of finite order acting on $X$ in a CT way, and suppose that $X^\phi$ is a disjoint union of embedded tori. Then, every connected component
$T \subseteq X^\phi$ satisfies $T \cdot T = 0$.
\end{lemma}

\begin{proof}
Choose some almost complex structure $J$ on $X$ which is $\omega$-compatible and $\phi$-invariant
(see e.g. \cite[Lemma 5.5.6]{McDS2}).
If there is some connected component $T' \subseteq X^\phi$ with negative self-intersection then
from Theorem \ref{thm:Atiyah-Singer-G} and the assumption $\sigma(X)=0$ we conclude that there is some $T \subseteq X^\phi$ with $T \cdot T > 0$.

We prove the lemma by contradiction. By the previous comment, it suffices to assume that there is a connected
component $T\subseteq X^\phi$ satisfying $T \cdot T > 0$.
Since $d \phi$ and $J$ commute, $J$ preserves the tangent bundle of $X^\phi$, and hence the tangent bundle of $T$. In particular, $T$ is $J$-holomorphic.
Denote by $[T]\in H_2(X)$ the fundamental class of $T$ corresponding to the standard
orientation as a closed Riemann surface.
We have
$$0<T \cdot T = [T]\cdot [T]=- K \cdot [T],$$
where the second equality is given by the adjunction formula.
By Lemma \ref{lemma:taubes}, $K$ or $2K$ are representable by $J$-holomorphic curves. Let $n=1$ if $K$ is representable, and let $n=2$ if $2K$ is representable and $K$ is not.

By definition there is a possibly disconnected closed Riemann surface $\Sigma$
and a $J$-holomorphic map $\psi:\Sigma\to X$ such that $nK=\psi_*[\Sigma]$.
Let $\{\Sigma_i\}$ be the connected components of $\Sigma$ and let
$A_i=\psi_*[\Sigma_i]$, so that
$nK = \sum_i A_i.$

We have  $A_i \cdot [T] \geq 0$ for all $i$.
This follows from Lemma \ref{lemma:positive-intersection}
if $\psi(\Sigma_i)\neq T$, and from the assumption $T\cdot T>0$ if $\psi(\Sigma_i)=T$
because in this case $A_i$ is a positive multiple of $[T]$.
Therefore we have
$$0 < n[T] \cdot [T]  = - nK \cdot [T] = - \sum_i A_i \cdot [T] \leq 0.$$
We have thus reached a contradiction, finishing the proof of the lemma.
\end{proof}

\begin{lemma}
\label{lemma:ruled-symplectic-jordan}
Let $X$ be an $S^2$-bundle over $T^2$. For any symplectic form $\omega$ on $X$
the symplectomorphism group $\Symp(X,\omega)$ is Jordan.
\end{lemma}
\begin{proof}
This is \cite[Corollary 1.5]{M5}.
\end{proof}

\subsection{Proof of statements (1) and (2) in Theorem \ref{thm:symp}}

If $X$ is an $S^2$-bundle over $T^2$, this follows from Lemma \ref{lemma:ruled-symplectic-jordan}.
Assume that $(X, \omega)$ is not an $S^2$-bundle over $T^2$ and
that $\Symp(X, \omega)$ is not Jordan. Then, by Theorem \ref{thm:criterion-degree}, there is some $\phi \in \Symp(X, \omega)$ of finite order and acting in a CT way on $X$ with the property that some connected component $T$ of $X^\phi$ is diffeomorphic to a torus and has positive self-intersection. This contradicts Lemma \ref{lemma:symplectic-torus},
so the proof of the first statement of Theorem \ref{thm:symp} is complete.

Statement (2) in Theorem \ref{thm:symp} follows from combining
Theorem \ref{thm:2-step-nilpotent}, (4) in Theorem \ref{thm:2-nilpotent-alpha}, and
Lemma \ref{lemma:symplectic-torus}.

\appendix

\section{Computing $H^*((\ZZ/p^r)^d;\ZZ/n)$: proof of Theorem \ref{thm:appendix}}

If $X$ and $Y$ are topological spaces
such that $H^k(X)$ and $H^k(Y)$ are finitely generated abelian
groups for every $k$, then K\"unneth's formula gives
\begin{equation}
\label{eq:Kunneth}
H^k(X\times Y)\simeq \bigoplus_{p+q=k}H^p(X)\otimes H^q(Y)
\oplus\bigoplus_{p'+q'=k+1}\Tor(H^{p'}(X),H^{q'}(Y))
\end{equation}
(see e.g. \cite[Chap. VII, Prop. 7.6]{Dold}). The universal coefficient theorem gives
isomorphisms
\begin{equation}
\label{eq:ucf}
H^k(X)\simeq \Hom(H_k(X),\ZZ)\oplus\Ext(H_{k-1}(X),\ZZ).
\end{equation}
and
\begin{equation}
\label{eq:ucf-a}
H^k(X;\ZZ/p^b)\simeq \Hom(H_k(X),\ZZ/p^b)\oplus\Ext(H_{k-1}(X),\ZZ/p^b).
\end{equation}
Let $a,b$ be positive integers and let $c=\min\{a,b\}$.
Fix a prime $p$ and define for convenience
$$G_a=\ZZ/p^a,\qquad G_b=\ZZ/p^b,\qquad G_c=\ZZ/p^c.$$
There are non canonical isomorphisms
\begin{equation}
\label{eq:tor}
\Tor(G_a,G_a)\simeq G_a,\qquad \Tor(\ZZ,G_a)=\Tor(G_a,\ZZ)=0,
\end{equation}
and
\begin{equation}
\label{eq:ext}
\Ext(G_a,G_b)\simeq G_c,\qquad \Ext(G_a,\ZZ)\simeq G_a,\qquad \Ext(\ZZ,G_b)=0.
\end{equation}

\begin{lemma}
There exists a function $e:\ZZ_{\geq 0}\times\NN\to\ZZ_{\geq 0}$
such that
\begin{equation}
\label{eq:function-e}
H^k(G_a^d;G_b)\simeq G_c^{e(k,d)}
\end{equation}
for every $(k,d)\in\ZZ_{\geq 0}\times\NN$ (by convention $G_c^0=0$).
\end{lemma}
The crucial fact here is that $e(k,d)$ is independent
of $p$ and $a,b,c$.
\begin{proof}
We first prove that there exists a function $f:\NN\times\NN\to\ZZ_{\geq 0}$
such that
\begin{equation}
\label{eq:function-f}
H^k(G_a^d)\simeq G_a^{f(k,d)}\quad\text{for every $k,d\in\NN$, and}\qquad
H^0(G_a^d)\simeq\ZZ
\end{equation}
(again we take the convention that $G_a^0=0$). We prove the existence of $f(k,d)$ using induction on $d$. First note that
setting
$$f(k,1)=\left\{\begin{array}{ll}1 &\qquad\text{if $k$ is even,}\\
0 & \qquad\text{if $k$ is odd,}\end{array}\right.$$
formulas (\ref{eq:function-f}) hold for $d=1$. For the inductive step we note
that if $BG$ denotes the classifying space of a group $G$ we have
$$BG_a^d\simeq BG_a^{d-1}\times BG_a,$$
so we can relate the cohomology of $G_a^d$ (which coincides with the singular cohomology of $BG_a^d$)
to that of $G_a^{d-1}$ and $G_a$ using (\ref{eq:Kunneth}).
To be specific, using (\ref{eq:tor}) we have, for every $k\in\NN$ and every $d\geq 2$,
$$f(k,d)=\sum_{0\leq 2l\leq k}f(k-2l,d-1)+\sum_{0<2l< k+1}f(k+1-2l,d-1),$$
where $l$ takes integer values. The first summation comes from the terms with $\otimes$ in
K\"unneth's formula (more concretely, the summand for each value of $l$ corresponds to
$H^{k-2l}(G_a^{d-1})\otimes H^{2l}(G_a)\simeq H^{k-2l}(G_a^{d-1})\simeq G_a^{f(k-2l,d-1)}$)
and the second summation comes from the terms with $\Tor$
(more concretely, the summand for each $l$ corresponds to
$$\Tor(H^{k+1-2l}(G_a^{d-1}),H^{2l}(G_a))\simeq H^{k+1-2l}(G_a^{d-1})\simeq f(k+1-2l,d-1);$$
we avoid the extreme values $2l=0$ and $2l=k+1$ because $\Tor(G_a,\ZZ)=\Tor(\ZZ,G_a)=0$).
This proves the existence of a function $f$ satisfying (\ref{eq:function-f}).

Now, using the universal coefficients theorem (\ref{eq:ucf}), the fact that the homology of a finite $p$-group
is a finite $p$-group in each degree $>0$, and (\ref{eq:ext}), we deduce that
$$H_k(G_a^d)\simeq G_a^{f(k+1,d)}\quad\text{for every $k,d\in\NN$, and}\qquad
H_0(G_a^d)\simeq\ZZ.
$$
Combining this formulas with (\ref{eq:ucf-a}) it follows that
$$H^1(G_a^d;G_b)\simeq G_c^{f(2,d)}$$
and that for every $k\geq 2$ we have
$$H^k(G_a^d;G_b)\simeq G_c^{f(k+1,d)}\oplus G_c^{f(k,d)}\simeq G_c^{f(k+1,d)+f(k,d)}.$$
Thus setting
$$e(1,d):=f(2,d),\qquad e(k,d):=f(k+1,d)+f(k,d)\text{ for every $k\geq 2$}$$
we obtain a function $e:\NN\times\NN\to\ZZ_{\geq 0}$ which
satisfies (\ref{eq:function-e}).
\end{proof}

In view of the lemma, to compute the function $e$ it suffices to consider the
case $a=b=c=1$, i.e., to compute
$$H^*((\ZZ/p)^d;\ZZ/p).$$
But this is much easier than the general case because, $\ZZ/p$ being a field,
we may apply K\"unneth's formula for fields, which does not contain $\Tor$ terms:
$$H^k(X\times Y;\ZZ/p)\simeq \bigoplus_{p+q=k}H^p(X;\ZZ/p)\otimes H^q(Y;\ZZ/p)$$
(again, under finiteness assumptions for $H^*(X;\ZZ/p)$ and $H^*(Y;\ZZ/p)$ on
each degree). This formula, together with the standard computation
\begin{equation}
\label{eq:case-a-1}
H^k(\ZZ/p;\ZZ/p)\simeq\ZZ/p\qquad\text{for every $k\geq 0$}
\end{equation}
implies the following recursion
formula for $d\geq 2$, which is much easier than the previous ones:
\begin{equation}
\label{eq:easy-recursion}
e(k,d)=e(0,d-1)+e(1,d-1)+\dots+e(k,d-1).
\end{equation}
It is now elementary to prove, e.g. using induction on $d$ (with (\ref{eq:case-a-1}) at
the initial step and (\ref{eq:easy-recursion}) at the induction step), that
$$e(k,d)=\left(k+d-1 \atop d-1\right).$$
The proof of the theorem is thus complete.

\end{document}